\documentclass{amsart}

\title{Associativity and Integrability}

\author[Rui L Fernandes]{Rui Loja Fernandes}
\email{ruiloja@illinois.edu}
\author{Daan Michiels}
\email{michiel2@illinois.edu}
\address{Department of Mathematics, University of Illinois at Urbana-Champaign, 1409 W. Green Street, Urbana, IL 61801, USA}
\thanks{This work was partially supported by NSF grants DMS 13-08472, DMS 14-05671, DMS-1710884 and FCT/Portugal.}
\date{\today}

\usepackage[T1]{fontenc}
\usepackage{lmodern}
\usepackage{microtype}
\usepackage{enumerate}
\usepackage{hyperref}
\usepackage{tikz}
\usetikzlibrary{patterns,decorations.pathreplacing}
\usetikzlibrary{decorations.markings}
\usetikzlibrary{matrix,arrows,calc}
\usetikzlibrary{arrows.meta}
\usepackage[utf8]{inputenc}
\usepackage{graphicx}
\usepackage{amsfonts, amsmath, amsthm, amssymb}
\usepackage[percent]{overpic}
\usepackage{tensor}
\usepackage{comment}
\usepackage[all]{xy}
\usepackage{listings}
\usepackage{placeins}

\DeclareFixedFont{\ttb}{T1}{txtt}{bx}{n}{9} 
\DeclareFixedFont{\ttm}{T1}{txtt}{m}{n}{9}  
\usepackage{color}
\definecolor{lightgrey}{rgb}{0.7,0.7,0.7}
\definecolor{deepblue}{rgb}{0,0,0.5}
\definecolor{deepred}{rgb}{0.6,0,0}
\definecolor{deepgreen}{rgb}{0,0.5,0}

\newcommand\pythonstyle{\lstset{
language=Python,
basicstyle=\ttm,
numbers=left,
stepnumber=1,
numberstyle=\itshape\color{lightgrey},
commentstyle=\itshape\color{lightgrey},
otherkeywords={as, self, yield},
keywordstyle=\ttb\color{deepblue},
emph={__name__},
emphstyle=\ttb\color{deepred},
stringstyle=\color{deepgreen},
frame=tb,
showstringspaces=false
}}
\lstnewenvironment{python}[1][]
{
\pythonstyle
\lstset{#1}
}
{}
 
\newcommand{\Z}{\ensuremath{\mathbb{Z}}}

\newcommand{\realization}[1]{{\lvert {#1}\rvert}}

\newcommand{\N}{\ensuremath{\mathcal{N}}}
\newcommand{\R}{\ensuremath{\mathbb{R}}}

\newcommand{\Q}{\ensuremath{\mathbb{Q}}}
\renewcommand{\S}{\ensuremath{\mathbb{S}}}
\newcommand{\F}{\mathcal{F}}
\newcommand{\eps}{\varepsilon}
\newcommand{\G}{\mathcal{G}}
\renewcommand{\NG}{\mathbf{G}}
\newcommand{\g}{\mathfrak{g}}
\newcommand{\HH}{\mathcal{H}}
\newcommand{\V}{\mathcal{V}}
\newcommand{\U}{\mathcal{U}}
\newcommand{\W}{\mathcal{W}}
\newcommand{\X}{\ensuremath{\mathfrak{X}}}
\newcommand{\AsCo}{\mathcal{AC}}
\newcommand{\Mon}{\tilde{\mathcal{N}}}
\newcommand{\Orbit}{\mathcal{O}}
\newcommand{\into}{\hookrightarrow}
\newcommand{\timesst}{\tensor[_s]{\times}{_t}}
\newcommand{\timesss}{\tensor[_s]{\times}{_s}}
\newcommand{\tto}{\rightrightarrows}
\newcommand{\wmc}{\omega_{\textrm{MC}}}

\DeclareMathOperator{\id}{Id}

\DeclareMathOperator{\Assoc}{Assoc}
\DeclareMathOperator{\im}{Im}

\DeclareMathOperator{\Ker}{Ker}
\renewcommand{\d}{\mathrm d}               
\DeclareMathOperator{\pr}{pr}      

\newcommand{\al}{\alpha}

\newtheorem{theorem}{Theorem}[section]
\newtheorem{lemma}[theorem]{Lemma}
\newtheorem{corollary}[theorem]{Corollary}
\newtheorem{proposition}[theorem]{Proposition}
\newtheorem{example}[theorem]{Example}

\theoremstyle{definition}
\newtheorem{definition}[theorem]{Definition}
\newtheorem{remark}[theorem]{Remark}

\begin{document}

\begin{abstract}
We provide a complete solution to the problem of extending a local Lie groupoid to a global Lie
groupoid. First, we show that the classical Mal'cev's theorem, which
characterizes local Lie groups that can be extended to global Lie groups, also
holds in the groupoid setting. Next, we describe a construction that can be
used to obtain any local Lie groupoid with integrable algebroid. Last, our main result
establishes a precise relationship between the integrability of a Lie algebroid
and the failure in associativity of a local integration. We give a simplicial
interpretation of this result showing that the monodromy groups of a Lie
algebroid manifest themselves combinatorially in a local integration, as a lack
of associativity.
\end{abstract}

\maketitle

\setcounter{tocdepth}{1}
\tableofcontents

\section{Introduction}

One of the main differences between the Lie theory of groupoids and the
ordinary Lie theory of groups is the failure of Lie's Third Theorem: not every
Lie algebroid integrates to a Lie groupoid. A fundamental fact discovered in
\cite{integ-of-lie-article} is that the lack of integrability of a Lie
algebroid can be measured by the so-called monodromy groups. On the other hand,
it is well known that every Lie algebroid integrates to a \emph{local} Lie
groupoid. In this paper we show that the failure of integrability of a Lie
algebroid can also be measured by the failure of associativity of any of its
local integrations. 

In connection with his proof of Lie's Third Theorem \cite{Cartan} for ordinary
Lie theory, E.~Cartan
already considered the question whether a local Lie group is contained in a
global Lie group, i.e.\ whether a local Lie group is ``globalizable'' (sometimes
called ``enlargeable'').
Mal'cev \cite{malcev} was the first one to define a notion of a local group
and he showed that the lack of associativity is the only
obstruction to embedding a local group in a global one. More precisely, the
associativity axiom for a local group requires that for every triple of
elements one has:
\[ g(hk)=(gh)k, \]
provided both sides are defined. While for a global group this implies that all
higher associativities hold, this is not true for a local Lie group. So, for
example, there exist local Lie groups in which one can find 4 elements such
that:
\[ (gh)(kl)\not = g((hk)l), \]
so that 4-associativity does not hold. An obvious necessary condition for a
local group to be globalizable is that $n$-associativity holds for all $n\ge 3$,
in which case we say that the local Lie group is \emph{globally associative}.
Mal'cev's theorem states that this condition is also sufficient. 

Mal'cev's notion of a local group was used by Smith \cite{Smith}, who applied a
simplicial complex approach to establish another criterion for globalizability.
Both the Mal'cev criterion and the Smith criterion were used in the 1960's to
understand the failure of Lie's Third Theorem for Banach Lie algebras. For an
account of these works, and the fact that the two criteria are in fact the
same, we refer to \cite{EstLee}. More recently, these kind of problems have
resurfaced in connection with Goldbring's solution (\cite{Goldbring}) to the
local version of Hilbert's fifth problem (see, e.g, the recent book by Tao
\cite{Tao}).

All these notions for local groups extend to local \emph{groupoids} so,
perhaps, it is not too surprising that a version of Mal'cev's result holds for
local groupoids. We will be interested mainly in the case of local \emph{Lie}
groupoids, and we will establish that:

\begin{theorem}[Mal'cev's theorem for local Lie groupoids]
    A sufficiently connected local Lie groupoid is globalizable if and only if it
    is globally associative.
\end{theorem}

The connectedness needed in the statement of this theorem as well as in the results that follow will be explained in the main body of the paper.
\medskip

Simple examples of non-globalizable local Lie groups exist, and Olver in
\cite{olver} gives a general method of constructing local Lie groups that are
not contained in a global Lie group, leading to a classification of local Lie
groups. The main idea is to start with a (global)
Lie group and consider an open neighborhood $U$ of the unit. Obviously, $U$
is a globalizable local Lie group. However, if $U$ is chosen not
to be 1-connected, a covering $\widetilde{U}$ inherits a local Lie group
structure from $U$ which, in general, is non-globalizable. Olver's theorem
essentially states that every local Lie group is covered by a local Lie group
which covers a globalizable local Lie group.

The theory of covers for local Lie groupoids is more subtle than for local Lie
groups. Still, we will show that when the underlying Lie algebroid is
\emph{integrable}, Olver's theorem also extends to local Lie groupoids with the
appropriate assumptions:

\begin{theorem}[Classification of local Lie groupoids]
    Suppose $G$ is a source-connected local Lie groupoid with
    integrable Lie algebroid $A$.
    Let $\tilde{G}$ be the source-simply connected cover of $G$ and let
    $\G(A)$ be the source-simply connected integration of $A$.
    Then we have a diagram:
    \begin{center}
     \begin{tikzpicture}[baseline=(current  bounding  box.center)]
        \matrix(m)[matrix of math nodes, row sep=1.0em, column sep=2.4em]{
            & \tilde{G} & \\
            G & & U\subseteq \G(A) \\
        };
        \path[->] (m-1-2) edge node[auto,swap]{$p_1$} (m-2-1.north east);
        \path[->] (m-1-2) edge node[auto]{$p_2$} (m-2-3.north west);
    \end{tikzpicture}
    \end{center}
    where $p_1$ is the covering map and $p_2$ is a generalized covering of local
    Lie groupoids.
\end{theorem}

Our version of Mal'cev's theorem shows that global associativity must fail for
a local Lie groupoid whose Lie algebroid is non-integrable. This is a
completely new aspect of the theory of groupoids which has no
counterpart for groups. One of the main goals of this paper is to establish a precise
connection between the lack of associativity and the lack of integrability.

Given a local groupoid $G$ over $M$ we build a new groupoid $\AsCo(G)$ over
$M$, called the {\em associative completion} of $G$ as follows. One introduces
the set of \emph{well-formed} words on $G$:
\[ W(G) :=\bigsqcup_{n\geq 1} \underbrace{G\timesst G\timesst \cdots \timesst G}_{\text{$n$ times}}, \]
namely the words
$(w_1,\dots,w_n)$ in $G$ formed by arrows whose source and target match:
$s(w_i) = t(w_{i+1})$ for all $i\in\{1,\dots,n-1\}$.
Given a well-formed word
$w=(w_1,\dots,w_k,w_{k+1},\dots,w_n)$ such that  $w_k$ and $w_{k+1}$ can be
composed, we have a new well-formed word $w' = (w_1,\dots,w_kw_{k+1},\dots,w_n)$.
We say that $w'$ is is obtained from $w$ by
\emph{contraction} or that $w$ is obtained from $w'$ by \emph{expansion}.
Contractions and expansions generate an equivalence relation $\sim$ on $W(G)$,
and one defines the associative completion of $G$ to be the space of
equivalence classes:
\[ \AsCo(G) := W(G)/{\sim}.\]
Notice that there is an obvious map $G\to \AsCo(G)$. 

Under a very mild assumption on $G$, \emph{concatenation} of well-formed words
gives $\AsCo(G)$ a groupoid structure over $M$. Moreover, if $G$ is a local
topological groupoid and we equip $\AsCo(G)$ with the quotient topology inherited
from $W(G)$, we have:

\begin{proposition}
    If $G$ is a local topological groupoid, then $\AsCo(G)$ is a
    topological groupoid and $G\to \AsCo(G)$ is a morphism of local topological
    groupoids. Given any morphism $F:G\to \HH$, where $\HH$ is a topological
    groupoid, there exists a unique morphism of topological groupoids
    $\tilde{F}:\AsCo(G)\to\HH$ such that the following diagram commutes:
   \[ \xymatrix{ G\ar[r]^F\ar[d] & \HH \\
                       \AsCo(G)\ar@{-->}[ur]_{\tilde{F}}}\]
\end{proposition}

If $F:G\to H$ is a morphism of local topological groupoids, there is an obvious
map  $\AsCo(F):\AsCo(G)\to\AsCo(H)$, which is a morphism of topological
groupoids, and makes the following diagram commutative:
\[
\xymatrix{
G\ar[d]\ar[r]^F & H\ar[d] \\
\AsCo(G)\ar[r]_{\AsCo(F)} & \AsCo(H)
}
\]
Hence, $\AsCo$ is a functor from the category of local topological groupoids to the category of topological groupoids. 

In general, if $G$ is a local Lie groupoid, $\AsCo(G)$ is not a Lie groupoid.
We call an element in the isotropy 
\[ g\in G_x=s^{-1}(x)\cap t^{-1}(x) \] 
an \emph{associator} at $x$ if there is a well-formed word $(w_1,\dots,w_n)$ which admits two sequences of contractions: one ending at $g$ and the other one ending at the unit $1_x$. The set of all associators, denoted $\Assoc(G)$, is contained in the kernel of the map $G\to \AsCo(G)$. Under mild assumptions on $G$, we show that it coincides with this kernel, and moreover that it controls the smoothness of $\AsCo(G)$:

\begin{theorem}
    If $G$ is a enough connected local Lie groupoid, then $\AsCo(G)$ is
    smooth if and only if $\Assoc(G)$ is uniformly discrete in $G$. In this case, $G \to \AsCo(G)$ is a local diffeomorphism.
\end{theorem}

Notice the obvious similarities between the functor $\AsCo(-)$ and the
integration functor $\G(-)$, which associates to a Lie algebroid the space of
$A$-paths modulo $A$-homotopies (see
\cite{integ-of-lie-article,lectures-integrability-lie}):
\begin{itemize}
\item If $A$ is a Lie algebroid, then $\G(A)$ is a topological groupoid. It is
    a Lie groupoid iff the monodromy $\N(A)$ is uniformly discrete in $A$;
\item If $G$ is a local Lie groupoid, then $\AsCo(G)$  is a topological
    groupoid. It is a Lie groupoid iff the associators $\Assoc(G)$ are
    uniformly discrete in $G$.
\end{itemize}
Moreover, if $\AsCo(G)$ is smooth then the Lie algebroid of $G$ is integrable
since $G$ and $\AsCo(G)$ have the same Lie algebroid. One of our main results
establishes a precise relationship between the monodromy groups and the
associators as follows:

\begin{theorem}
    Let $G$ be a shrunk local Lie groupoid $G$ with Lie algebroid $A\to M$. For each $x\in M$, there is an 
    embedding of local groups $G_x\hookrightarrow \G(\g_x)$ such that:
    \[ \Assoc_x(G)=G_x\cap \N_x(A). \]
\end{theorem}

The assumption in this theorem that $G$ is shrunk is of a topological nature.
Any local Lie groupoid has a neighborhood of the units which is shrunk. Such
kind of assumption on $G$ is required in order to be able to establish a
relationship between the monodromy groups and the associators, since we will
see that even local Lie groups may have non-discrete associators.

The associative completion $\AsCo(G)$ has a simplicial interpretation: a local Lie groupoid $G$ has a nerve which is a simplical set $\mathbf{G}=\{G^{(m)}\}$ and $\AsCo(G)\simeq \Pi_1(\mathbf{G})$, the path groupoid of $\mathbf{G}$. The theory is more complicated than for (global) groupoids, since now the nerve fails to be a Kan complex. Still, just like the monodromy groups $\N_x(A)$, that can be defined as the image of a monodromy homomorphism $\partial:\pi_2(M,x) \to\G(\g_x)$, which appears as the connecting homomorphism of a long exact sequence (see \cite{integ-of-lie-article}):
\[ \xymatrix{\cdots \ar[r] & \pi_2(M,x) \ar[r]^\partial & \G(\g_x)\ar[r] & \G_x(A) \ar[r] & \pi_1(M,x)\ar[r] & \cdots} \]
we will see that there is a simplicial version of the monodromy map involving
the simplicial homotopy groups:
\[ \xymatrix{\cdots \ar[r] & \pi_2(\mathbf{U},x) \ar[r]^{\partial_s}& \AsCo(G_x) \ar[r] & \AsCo_x(G) \ar[r] & \pi_1(\mathbf{U},x)\ar[r] & \cdots} \]
Here, $U$ is the local groupoid over $M$ obtained as the image of the
source/target map $(t,s):G\to M\times M$, and $\mathbf{U}$ denotes its nerve.
We will see that then it follows that for a shrunk local Lie groupoid there is
an isomorphism $\AsCo(G)\simeq \G(A)$ (as topological groupoids).
\medskip

This paper is organized as follows. In Section~\ref{sec:preliminaries} we discuss some basic properties of local Lie groupoids, including 
our definition, various connectedness assumptions and constructions. 
In Section~\ref{ss:example-non-globalizable} we give a few examples illustrating the failure of  associativity. In Section~\ref{sec:malcev} we show that Mal'cev's theorem holds for local Lie groupoids.
In Section~\ref{sec:classification} we extend a construction of local Lie groups
due to Olver \cite{olver}
to the case of local Lie groupoids. This leads to the classification of
local Lie groupoids with integrable algebroid.
In Section~\ref{sec:integrability}, we deduce our results concerning the relationship between associativity
of a local Lie groupoid and the integrability of its Lie algebroid, via
associators and monodromy groups. Finally, Section~\ref{s:simplicial:monodromy}
discusses the simplicial approach to $\AsCo(G)$ and the simplicial version of
the monodromy map.  The appendix contains Python code for a video that
illustrates a construction related to a triangulation used in the proof of the main
theorem of Section 6.
\medskip

{\bf Acknowledgements.} We would like to thank Alejandro Cabrera, Marius Crainic, Matias del Hoyo, Gustavo Granja, Peter Olver and Ioan Marcut for 
comments and remarks at various stages of this project. In particular, the last section of this paper originated from discussions with Matias del Hoyo.

\medskip

{\bf Notations and Conventions.}
All manifolds and maps under consideration will be smooth.
Throughout the text, local groupoids will be denoted by Latin letters (e.g.\ $G$),
and global groupoids will be denoted by calligraphic versions (e.g.\ $\G$).
If $A$ is a Lie algebroid, then $\G(A)$ will denote the associated source-simply connected
groupoid (it is global, so it has a calligraphic symbol.)
Arrows of a groupoid compose from right to left, so that
$gh$ can only be defined if $s(g) = t(h)$.

\section{Local Lie groupoids}
\label{sec:preliminaries}

In this section, we discuss some basic concepts and properties of local Lie groupoids that will be relevant for this paper.

\subsection{Definition of a local Lie groupoid}

There are several different definitions of local Lie groupoids. We adopt the following ``weak'' version as definition (there are even weaker versions, see e.g. \cite{CMS}). Later we will consider stronger versions.

\begin{definition}
    A \emph{local Lie groupoid} $G$ over a manifold $M$ is a manifold $G$,
    together with maps:
    \begin{itemize}
        \item $s, t : G \to M$ submersions (the \emph{source} and \emph{target} maps);
        \item $u : M \to G$ a smooth map (the \emph{unit} map);
        \item $m : \U \to G$ a submersion (the \emph{multiplication}), where $\U\subset  G\timesst G$ is an open neighborhood of:
            \[ (G\timesst M) \cup (M\timesst G)
                = \bigcup_{g\in G} \{ (g,u(s(g))), (u(t(g)), g) \} ;\]
        \item $i : \V \to \V$ a smooth map (the \emph{inversion}), where $\V \subset  G$ is an open neighborhood of $u(M)$
            such that $\V\timesst \V \subset \U$;
    \end{itemize}
    such that the following axioms hold:
    \begin{itemize}
        \item $s(m(g,h)) = s(h)$ and $t(m(g,h)) = t(g)$ for all $(g,h) \in \U$;
        \item $m(m(g,h),k) = m(g,m(h,k))$ in a subset $\W\subset G\timesst G\timesst G$ whose interior contains
        \[ (M\timesst M\timesst G) \cup (M\timesst G\timesst M) \cup (G\timesst M\timesst M);\]
        \item $m(g,u(s(g))) = m(u(t(g)),g) = g$ for all $g\in G$;
        \item $s(i(g)) = t(g)$ and $t(i(g)) = s(g)$ for all $g\in\V$;
        \item $m(i(g),g) = u(s(g))$ and $m(g,i(g)) = u(t(g))$ for all $g\in \V$.
    \end{itemize}\label{def:llg}
\end{definition}

\begin{remark} A few remarks concerning this definition are in order:
\begin{enumerate}
\item By a \emph{local Lie group} we will always mean a local Lie groupoid $G\tto M$ with objects the singleton: $M=\{*\}$.
\item It follows from this definition that $u : M \to G$ is an embedding and we
    will consider $M$ as a submanifold of $G$ using this embedding.  In
    particular, if $x\in M$, we will write the unit at $x$ as just $x$, leaving
    out the $u$.
\item It also follows from the definition that $i:\V\to \V$ is an involutive
    diffeomorphism. An element of $\V$ will be called \emph{invertible}. Note
    that we do not require all elements in $G$ to be invertible, unlike some of
    the definitions in the literature, and this is relevant for us (see Remarks
    \ref{rmk:MC:inverses}, \ref{rmk:Malcev:inverses} and
    \ref{rmk:Covers:inverses}).
\item For the multiplication we usually write $gh$ instead of $m(g,h)$. We
    require it to be defined near $(G\timesst M) \cup (M\timesst G)$, not just
    near $M\timesst M$. In other words, $t(g)g$ and $g s(g)$ are defined (and
    equal $g$) for all $g\in G$.
\item For a Lie groupoid the fact that multiplication is a submersion follows
    from the other axioms, but this is not the case for a local Lie groupoid. 
\item Just as in the case of Lie groupoids, we allow the manifold $G$ to be
    non-Hausdorff. However, all other manifolds, including $M$, the source and
    target fibers, are assumed to be Hausdorff. 
\item Occasionally, we will deal with \emph{local topological groupoids}: the
    definition is analogous but one works in the topological category instead.
    Similarly, one can define \emph{local groupoids} completely free from
    topology (where the sets $\U$, $\V$ and $\W$, define the ``local'' nature of $G$).
\end{enumerate}
\end{remark}

There are two different ways of obtaining a smaller local Lie groupoid $G'$
from a given local Lie groupoid $G$. Both of them are relevant for us:
\begin{itemize}
    \item We say that $G'$ is obtained by \emph{restricting} $G$, if both local
        groupoids have the same manifolds of arrows and objects, the same
        source and target maps, and the multiplication and inversion in $G'$
        are obtained by restricting the ones of $G$ to smaller domains;
    \item We say that $G'$ is obtained by \emph{shrinking} $G$ if $G'$ is an
        open neighborhood of $M$ in $G$, the source and target maps are the
        restrictions of $s$ and $t$ to $G'$, multiplication is the restriction
        of $m$ to $\U \cap (G' \timesst G')\cap m^{-1}(G')$, and
        inversion is the restriction of $i$ to $\V'=(\V \cap G') \cap i(\V \cap
        G')$.
\end{itemize}

\begin{example}[Restriction of a Lie groupoid]
Let $\G\tto M$ be a Lie groupoid. Any open $\U\subset \G^{(2)}$ containing
$(\G\timesst M) \cup (M\timesst \G)$ determines a restriction $G$ of $\G$. On
the other hand, any open neighborhood $G'\subset \G$ of the unit section
$M$ determines a local Lie groupoid $G'$ shrinking $\G$. 
\end{example}

\begin{example}[One-point compactification]
The one-point compactification
$\R_\infty=\R\cup\{\infty\}$, becomes a local Lie group if one extends addition by setting:
\[ g+\infty=\infty=\infty+ g,\quad \forall g\in \R. \]
Although we could define $\infty+\infty=\infty$, multiplication would not
be smooth at $(\infty,\infty)$. So the domain of multiplication is
$\U=\R_\infty\times \R_\infty\setminus\{(\infty,\infty)\}$ and the domain of inversion is
$\V=\R$. 
\end{example}

\begin{example}[$A$-path local integration]
\label{canonical-integ}
   Let $A$ be a Lie algebroid and let $P(A)$ denote the space of $A$-paths. A choice of $A$-connection $\nabla$ on $A$  
   determines an exponential map $\exp_\nabla:A\to P(A)$ (see, e.g., \cite{lectures-integrability-lie}). If $\sim$ denotes 
   the equivalence relation on $P(A)$ determined by $A$-homotopy, it is shown in \cite{integ-of-lie-article} that, 
   for a sufficiently small neighborhood $V\subset A$ of the zero section, the quotient 
   \[ G:=\exp_\nabla(V)/\sim \]
   is a local Lie groupoid with (partial) multiplication defined by concatenation of $A$-paths. We will call $G$ 
   an \emph{$A$-path local integration} of $A$. 
   There are other methods to associate a local groupoid to a Lie algebroid (see, e.g., \cite{CMS}).
\end{example}

Morphisms between local Lie groupoids are defined in a more or less obvious way.

\begin{definition}
    Suppose that $G_1$ and $G_2$ are local Lie groupoids over $M_1$ and $M_2$,
    respectively.
    A \emph{morphism of local Lie groupoids} is a pair $(F,f)$, where
    $F : G_1\to G_2$ and $f:M_1\to M_2$ are smooth maps
    such that
    \begin{itemize}
        \item $F \circ u_1 = u_2 \circ f$,
        \item $f \circ s_1 = s_2 \circ F$ and $f\circ t_1 = t_2\circ F$,
        \item $(F\times F)(\U_1) \subset \U_2$ and $F\circ m_1= m_2\circ(F\times F)$
            on $\U_1$,
        \item $F(\V_1) \subset \V_2$ and $F \circ i_1 = i_2\circ F$ on $\V_1$.
    \end{itemize}
\end{definition}

Note that the map $f$ is really just the restriction of $F$ to the unit manifold,
so $F$ completely determines the morphism, and it is therefore natural to think of the
morphism as just $F$.
If $G'$ is obtained from $G$ by either restricting or shrinking, the inclusion $G' \to G$
is a morphism of local Lie groupoids.

\subsection{Associativity}
\label{ss:associativity}

One can strengthen the associativity axiom in the definition of a local Lie
groupoid in various ways. The following definition is taken from \cite{olver},
with group replaced by groupoid:

\begin{definition}
    A local Lie groupoid is called \emph{associative to order $n$} (or
    \emph{$n$-associative}) if for every $3 \leq m\leq n$ and every ordered
    $m$-tuple of groupoid elements 
    \[ (x_1,\dots,x_m)\in G \timesst G \timesst \cdots \timesst G,\] 
    all corresponding defined $m$-fold products are equal (in other words, if
    an $m$-fold product can be evaluated by putting in brackets in distinct
    ways, all result in the same answer).
\end{definition}

For example, a local groupoid $G$ is 3-associative if 
\[ (gh)k=g(hk) \quad\text{whenever $(gh)k$ and $g(hk)$ are both defined}. \]
This is stronger than the associativity axiom of a local groupoid, where this
only needs to hold in some neighborhood of 
\[ (M\timesst M\timesst G) \cup (M\timesst G\timesst M) \cup (G\timesst M\timesst M).\]

It is a (perhaps surprising) fact that for every $n\geq 3$ there are
$n$-associative local Lie groupoids that are not $(n+1)$-associative. This is
already true in the case of local Lie groups (\cite[section 3]{olver}) and we
will discuss more examples later.

\begin{definition}
    A local Lie groupoid is called \emph{globally associative} if it is
    associative to every order $n\geq 3$.
\end{definition}

Clearly, the local Lie groupoids obtained by restricting or shrinking a Lie
groupoid are always globally associative. On the other hand, the $A$-path local
integrations of a Lie algebroid $A$ may fail to be even 3-associative. However:

\begin{proposition}
Let $G$ be a local Lie groupoid. For each $n\ge 3$ there is a restriction of $G$
which is $n$-associative.
\end{proposition}

\begin{proof}
We will show that any local Lie groupoid $G$ admits a restriction which is 3-associative. The case $n>3$ is
similar. 

Let $\W \subset G\timesst G\timesst G$ be the region where 3-associativity holds. First we note that we can assume that the domain of multiplication
$\U\subset G\timesst G$ is such that:
\[ \{(g,s(g),h) \mid (g,h)\in \U\}\subset \textrm{int}(\W).\]
This follows because the  set:
\[ C=\{(g,h)\in G\timesst G \mid (g,s(g),h)\in \partial(\W)\}, \]
is closed and does not intersect $(M\timesst G)\cup (G\timesst M)$, so we can replace $\U$ by the open set $\U-C$. Now we split the proof into two steps:

\emph{Step 1: $G$ compact.} Assume that there is no restriction $\U'\subset \U$
where 3-associativity holds, i.e., for every open set $\U'$
\[ (M\timesst G)\cup (G\timesst M)\subset \U'\subset \U\quad \Longrightarrow\quad (\U'\times_G \U')\setminus \W\ne\emptyset. \] 
Choose a sequence of open sets
\[ \U\supset \U_1\supset\cdots \supset \U_n\supset \dots\supset (M\timesst G)\cup (G\timesst M) \]
such that
\begin{enumerate}
\item[(a)] $\U_{n-1}\supset \overline{\U}_n$;
\item[(b)] $\bigcap_{n=1}^\infty \U_n=(M\timesst G)\cup (G\timesst M)$.
\end{enumerate}
If we pick elements
\[ (g_n,h_n,k_n)\in  (\U_n \times_G \U_n)\setminus \W\]
we obtain a sequence which can be assumed to converge to an element:
\[ (g_n,h_n,k_n)\to (g,h,k)\in (\Delta\times_G \Delta)\cap(\U\times_G \U), \]
where $\Delta=(M\timesst G)\cup (G\timesst M)$. This means that the limit is either of the form:
\[ (g,h,k)=(t(h),h,s(h))\in \text{int}(\W), \]
(by the definition of a local topological groupoid) or of the form:
\[ (g,h,k)=(g,s(g),k)\in \text{int}(\W),\]
(by the remark at the beginning of the proof).
This contradicts the fact that $(g_n,h_n,k_n)\not\in \W$ for all $n$.
Hence, there must exist a restriction $\U'\subset \U$ for which
3-associativity holds.

\vskip 10 pt

\emph{Step 2: $G$ not compact.} We can find a sequence 
\[ G_1\subset G_2 \subset \cdots \subset G_n\subset \cdots G \]
where each $G_i$ is an open, precompact subgroupoid of $G$
(the inclusion is a local groupoid morphism), and such that:
\[ \bigcup_{n=1}^\infty G_n =G. \]
The proof of step 1 shows that for each $G_n$ there is a restriction $\U_n\subset G_n\timesst G_n$ where 3-associativity holds. But then:
\[ \U=\bigcup_{n=1}^\infty \U_n, \]
gives a restriction of $G$ where 3-associativity holds.
\end{proof}

As we shall see later, a local Lie groupoid may fail to have a restriction
which is globally associative or even a shrinking which is globally associative. On
the other hand, by Lie's Third Theorem, a local Lie group can always be shrunk
to a globally associative one.

From now on, to simplify the discussion, we will assume that local groupoids
are always 3-associative, unless otherwise indicated. The previous result guarantees that this can always
be achieved by restricting.

\subsection{The nerve of a local Lie groupoid}
\label{ss:nerve}

Given a local Lie groupoid, we introduce a simplicial set 
\[ \xymatrix@R=8pt{G^{(0)} & \ar@<-.4ex>[l] \ar@<.4ex>[l] G^{(1)}  & \ar@<-.8ex>[l] \ar[l] \ar@<.8ex>[l] G^{(2)} \cdots} \]
as follows. We let $G^{(0)}=M$, $G^{(1)}=G$ and for $m>1$ we denote by:
\[ G^{(m)}\subset G \timesst G \timesst \cdots \timesst G, \]
the set formed by $m$-tuples $(g_1,\dots,g_m)$ for which \emph{all} $m$-fold
products are defined and are equal. If $m=1$, we have $d_0(g)=s(g)$ and
$d_1(g)=t(g)$, while for $m>1$ we have the usual formulas for the face maps:
\begin{align*}
&d_0(g_1,\dots,g_m)=(g_2,\dots,g_m),\\
&d_i(g_1,\dots,g_m)=(g_1,\dots,g_ig_{i+1},\dots, g_m),\quad (i=1,\dots,m-1)\\
&d_m(g_1,\dots,g_m)=(g_1,\dots,g_{m-1}).
\end{align*}
On the other hand, the degeneracy maps are given by:
\[ s_i(g_1,\dots,g_m)=(g_1,\dots,g_i,s(g_i),g_{i+1},\dots,g_m),\quad (i=0,\dots,m). \]

Notice that $G^{(m)}$ is, in general, strictly smaller than the subset of
$m$-tuples in $G \timesst G \timesst \cdots \timesst G$ for which
$m$-associativity holds: an $m$-tuple $(g_1,\dots,g_m)$ can satisfy
$m$-associativity but some $m$-fold products may fail to be defined. 

\begin{proposition}
Given a local Lie groupoid $G$, its nerve $G^\bullet=\{G^{(m)}\}$ is a
simplicial manifold. It is a Kan complex if and only if $G$ is a Lie groupoid.
\end{proposition}

\begin{proof}
We first prove that $G^{(m)}$ is a manifold, by showing that it is open
in $G\timesst ...\timesst G$. We do this by induction: clearly $G^{(0)} = M$
is open in $M$, $G^{(1)} = G$ is open in $G$, and $G^{(2)} = \U$ is open in
$G\timesst G$. Now if $G^{(n)}$ is open in $G\timesst ...\timesst G$ ($n$ copies),
then we see that
\[ G^{(n+1)} = \bigcap_i d_i^{-1}(G^{(n)}) \]
is a finite intersection of open subsets of $\U\times_G ...\times_G \U$ ($n-1$ copies),
and therefore an open subset of $G\timesst ...\timesst G$.
Moreover, because multiplication is assumed to be a submersion,
all the face maps are submersions.
The nerve is therefore a simplicial manifold.

The (differential versions \cite{Henriques} of the) Kan conditions hold for $G^{(2)}$ if and
only if $G^{(2)}=G\timesst G$ and inversion has domain $G$.
This already shows that $\{G^{(m)}\}$ is Kan if and only if $G$ is a Lie
groupoid (for more details see \cite{Henriques,Zhu2009}).
\end{proof}

\subsection{Connectedness and bi-regularity for local Lie groupoids}
Most of the material in this section is an adaptation of material in \cite{olver} to the groupoid case. We write $T^sG$ for the kernel of the map $\d s : TG \to TM$, and $T^tG$ for the kernel of $\d t : TG\to TM$.


\begin{definition}
    Let $G$ be a local Lie groupoid and fix $g\in G$.
    We say that $G$ is \emph{right-regular at $g$} if right multiplication by
    $g$ induces an isomorphism $\d R_g:T^s_{h} G\to T^s_{hg}{G}$, for every $h$
    such that $(h,g)\in \U$. We say that $G$ is \emph{left-regular at $g$} if
    left multiplication by $g$ induces an isomorphism
    $\d L_g:T^t_{h} G\to T^t_{gh} G$, for every $h$ such that $(g,h)\in \U$.
    We say that $G$ is \emph{bi-regular at $g$} if it is both left-regular and
    right-regular at $g$.  The local groupoid $G$ is called \emph{bi-regular}
    if it is bi-regular at all $g\in G$.
\end{definition}

For example, the groupoid $\R_\infty$ is bi-regular at all points except
$\infty$. Another peculiarity of the element $\infty$ is that it is not
invertible. Note, however, that an element may fail to be invertible, while still
being expressible as a product of invertible elements. For that reason we
introduce:

\begin{definition}
    An element $g$ of a local Lie groupoid is called \emph{inversional}
    if it can be written as a well-defined product of invertible elements.
    A local Lie groupoid is called \emph{inversional} if all of its elements are inversional.
\end{definition}

Every Lie groupoid is obviously inversional and every local groupoid has a shrinking which is inversional. The example of $\R_\infty$ shows that a local groupoid may not have a restriction which is inversional. Also, restricting a local groupoid 
can change it from inversional to not inversional.

\begin{example}
   Consider the Lie group $\R \times \Z$ as a local Lie groupoid (with
   the usual addition). If we restrict its inversion map to a domain $\V'$,
   the resulting local Lie groupoid is inversional precisely if
   $\{ n\in \Z \mid \R\times\{n\} \cap \V' \neq \emptyset \} \subset \Z$
   generates $\Z$.
\end{example}


Let us turn now to connectedness of local Lie groupoids. There are several
notions that will play a role in the sequel. First, we have the following
obvious concepts that already play an important role for Lie groupoids:

\begin{definition}
    We say that a local Lie groupoid is \emph{source-connected} (or $s$-connected) if
    all of its source fibers are connected.
    We say that a local Lie groupoid is \emph{target-connected} (or $t$-connected) if
    all of its target fibers are connected.
\end{definition}


In \cite{olver}, the notion of connectedness of a local Lie group requires that every neighborhood of the unit generates
the local group. We have an analogous notion for groupoids. First:

\begin{definition}
    Let $U$ be a neighborhood of $M$ in $G$.
    We say that \emph{$U$ generates $G$} if every element of $G$ can be written
    as a well-defined product of elements in $U$.
\end{definition}

The following stronger version of connectedness will play a crucial role:

\begin{definition}
    A local Lie groupoid $G$ over $M$ is \emph{strongly connected} if
    \begin{enumerate}[(a)]
        \item \label{item:Mconn} $M$ is connected;
        \item  \label{item:stconn0} the domains $\U$ and $\V$ of the
            multiplication and inversion maps are connected;
        \item \label{item:stconn1} the set
            $\{ (g,h) \in \U \mid s(g) = t(h) = x \}$
            is connected for all $x \in M$;
        \item \label{item:stconn2} $G$ is $s$-connected and $t$-connected, and
        \item $G$ is bi-regular.
    \end{enumerate}
\end{definition}

In the case of a local Lie group, condition (\ref{item:Mconn}) is automatic, (\ref{item:stconn1}) follows from (\ref{item:stconn0}), and (\ref{item:stconn2}) amounts to connectedness of the manifold $G$.

We have the following result which implies, in particular, that a strongly connected
local Lie groupoid is inversional.

\begin{proposition}
    If $G$ is right-regular and $s$-connected,
    then any neighborhood of $M$ generates $G$.
    If $G$ is left-regular and $t$-connected, the same conclusion holds.
\end{proposition}

\begin{proof}
    We will discuss the case where $G$ is right-regular and $s$-connected. The
    other case is analogous.
    Let $U$ be a neighborhood of $M$ in $G$.
    Pick an element $x\in M$. We will show that the set
    \[ S = \{ g\in s^{-1}(x) \mid \text{$g$ can be written as a product of elements in $U$} \} \]
    is both open and closed in $s^{-1}(x)$ (and therefore, by $s$-connectedness, it is the
    entire source fiber).
    This suffices to prove the result.

    By right-regularity of $G$, it is clear that $S$ is open.
    Pick $g\in \partial S$. We will show that $g\in S$.
    There is a neighborhood $V$ of $t(g)$ in $s^{-1}(t(g))$, consisting
    of invertible elements, contained in $U$, such that
    for all $h \in V$ we have
    \[ (h,g) \in \U \quad\text{and}\quad (h^{-1},hg) \in \U .\]
    By right-regularity of $G$, the set $\{ hg \mid h \in V \}$ is a neighborhood
    of $g$ in $s^{-1}(x)$, and so contains an element of $S$.
    Pick such an element $k\in S$.
    Then $k = hg$ for some $h\in V$, and by our choice of $V$, the product $g = h^{-1}k$ is defined.
    But since $k$ can be written as a product of elements of $U$, this last equation
    shows that $g$ also can.
\end{proof}


The following lemma shows that, in some sense,
little information is lost by restricting a strongly connected local Lie groupoid.

\begin{lemma}
    Suppose that we have two strongly connected 3-associative local Lie groupoids
    structures on $G$ over $M$, with the same source, target and unit maps,
    and the same domains for multiplication and inversion $(s,t,u, \U, \V,m_j : \U \to G,i_j : \V\to G)$
    for $j\in\{1,2\}$.
    If the two structures have a common restriction, i.e., there are neighborhoods
    $\U' \subset \U\subset M\timesst G \cup G\timesst M$ and $\V'\subset \V\subset M$ such that
    $m_1{\restriction_{\U'}} = m_2{\restriction_{\U'}}$ and
    $i_1{\restriction_{\V'}} = i_2{\restriction_{\V'}}$,
    then $m_1 = m_2$ and $i_1 = i_2$.
    \label{lem:stronglyconnecteddeterminedbyrestriction}
\end{lemma}
\begin{proof}
    Let us first show that the multiplication maps agree.
    Pick a point $x\in M$.
    Let $F = \{ (g,h) \in \U \mid s(g) = t(h) = x \}$.
    We will show that $m_1{\restriction_F} = m_2{\restriction_F}$.
    Since $x$ was arbitrary, this will prove that $m_1 = m_2$. By strong connectedness of $G$, the set $F$ is connected, and 
    we argue that the set
    \[ F' = \{ (g,h) \in F \mid m_1(g,h) = m_2(g,h) \} \]
    is open and closed (so that it is all of $F$).
    Closedness of $F'$ is obvious, so we just have to show that $F'$ is open.

    Pick $(g,h) \in F'$.
    Let us write $gh$ for $m_1(g,h) = m_2(g,h)$.
    Let $N_h$ be a compact neighborhood of $s(h)$ in $t^{-1}(s(h))$, small enough
    such that for all $\overline{h} \in N_h$ we have
    \begin{itemize}
        \item $(h,\overline{h}) \in \U'$ (so that we can write $h\overline{h}$ without ambiguity),
        \item $(g,h\overline{h}) \in \U$,
        \item $(gh, \overline{h}) \in \U'$.
    \end{itemize}
    Then
    \[ m_1(g,h\overline{h})= m_1(gh,\overline{h}) = m_2(gh,\overline{h}) = m_2(g,h\overline{h}), \]
    showing that $(g,h\overline{h}) \in F'$
    (so that we can write $gh\overline{h}$ without ambiguity).
    Now let $N_g$ be a neighborhood of $t(g)$ in $s^{-1}(t(g))$ small enough
    such that for all $\overline{h} \in N_h$ and $\overline{g} \in N_g$ we have
    \begin{itemize}
        \item $(\overline{g},g) \in \U'$ (so that we can write $\overline{g}g = m_1(\overline{g},g) = m_2(\overline{g},g)$),
        \item $(\overline{g},gh\overline{h}) \in \U'$,
        \item $(\overline{g}g,h\overline{h}) \in \U$.
    \end{itemize}
    Then
    \[m_1(\overline{g}g,h\overline{h}) = m_1(\overline{g},gh\overline{h}) = m_2(\overline{g},gh\overline{h}) = m_2(\overline{g}g,h\overline{h}), \]
    showing that $(\overline{g}g,h\overline{h}) \in F'$.
    But by bi-regularity of $m_1$ and $m_2$, elements of the form $(\overline{g}g,h\overline{h})$ form
    a neighborhood of $(g,h)$ in $F$.
    This shows that $F'$ is open, and therefore that the multiplications agree.
    By uniqueness of inverses, the inversion maps also agree.
\end{proof}

There is one more type of connectedness that we will need:

\begin{definition}
    We say that a local Lie groupoid $G$ \emph{has products connected to the axes}
    if for every $(g,h) \in \U$, there is either a path $\gamma$ from $t(h)$ to $g$ in $G$
    such that $(\gamma(\tau),h) \in \U$ for all $\tau$,
    or there is a path $\gamma$ from $s(g)$ to $h$ in $G$ such that
    $(g,\gamma(\tau)) \in \U$ for all $\tau$.
\end{definition}

Note that every local Lie groupoid has a restriction with products connected to
the axes. The most important consequence of this assumption is stated in the
following proposition, which will be needed later. In particular, it shows that
this property together with bi-regularity implies that the groupoid is
inversional.

\begin{proposition}
	\label{prop:inverses:decomp}
    Let $G$ be a bi-regular local Lie groupoid with products connected to axes.
    For any $(g,h)\in\U$ there are either invertible elements $a_1,\dots,a_l\in G$
    such that
	\begin{equation}
	\label{eq:inverses:decomp:1} 
	g=a_l(\cdots(a_2 a_1))\quad\text{ and }\quad h=a_1^{-1}(\cdots(a_l^{-1}(gh)))),
	\end{equation}
	or invertible elements $b_1,\dots,b_m\in G$ such that
	\begin{equation}
	\label{eq:inverses:decomp:2}
	h=(((b_1b_2)\cdots)b_m\quad\text{ and }\quad g=((((gh)b_m^{-1})\cdots)b_1^{-1}.
	\end{equation}
\end{proposition}

\begin{proof}
    Because $G$ has products connected to the axes,
    there is either a curve in $G$ from $t(h)$ to $g$ in $s^{-1}(t(h))$
    that can be right-translated by $h$, or a curve from $s(g)$ to $h$ in
    $t^{-1}(s(g))$ that can be left-translated by $g$.
    We will assume the former, and show that there are invertible elements $a_1,\dots,a_l\in G$  
    such that \eqref{eq:inverses:decomp:1} holds. If one assumes the latter, then an entirely similar
    argument will show that \eqref{eq:inverses:decomp:2} holds. 
       
    Let $\gamma$ be such a curve in $s^{-1}(t(h))$ from $t(h)$ to $g$
    and let $\gamma'$ be the curve obtained from $\gamma$ by right-translating by $h$. 
    Let $T \subset [0,1]$ be the set of all $\tau\in[0,1]$ such that
    there is a sequence $a_1, \dots, a_l \in G$ such that
    \[ \gamma(\tau) = a_l(\cdots(a_1(s(g)))) \]
    and
    \[ h = a_1^{-1}(\cdots(a_l^{-1}(\gamma'(\tau)))) \]
    are defined and true. We claim that $T = [0,1]$.
    We will prove this by showing that $T\subset[0,1]$ is both open and closed (clearly it is
    nonempty, because $0\in T$).

    We first prove that $T$ is open. Pick $\tau_0\in T$.
    Let $a_1,\dots,a_l$ be a sequence such that $\gamma(\tau_0) = a_1(\cdots(a_l(s(g))))$
    and
    $h = a_1^{-1}(\cdots(a_l^{-1}(\gamma'(\tau_0))))$.
    Now $t(\gamma(\tau_0))$ has a neighborhood in $s^{-1}(t(\gamma(\tau_0)))$
    of invertible elements
    such that
    for every $g$ in this neighborhood we have
    \[ (g,\gamma(\tau_0)) \in \U \text{ and } (g^{-1}, g\gamma'(\tau_0)) \in \U .\]
    Then by bi-regularity of $G$, for all $\tau$ close enough to $\tau_0$ we can write
    \[ \gamma(\tau) = g\gamma(\tau_0) = g(a_l(\cdots(a_1(s(g))))) \]
    and
    \[ h = a_1^{-1}(\cdots(a_l^{-1}(g^{-1}(\gamma'(\tau))))) .\]
    This shows that $T$ is open.
    
    We now show that $T$ is closed.
    Pick $\tau_0 \in \partial T$.
    Now $t(\gamma(\tau_0))$ has a neighborhood in $t^{-1}(t(\gamma(\tau_0)))$
    of invertible elements
    such that
    for every $g$ in this neighborhood we have
    \[ (g^{-1}, \gamma'(\tau_0)) \in \U \text{ and } (g^{-1}, \gamma(\tau_0)) \in \U \text{ and } (g,g^{-1}\gamma(\tau_0)) \in \U .\]
    Using bi-regularity of $G$ and picking $\tau \in T$ sufficiently close to $\tau_0$ we can find a $g$
    in this neighborhood
    such that $\gamma(\tau) = g^{-1}\gamma(\tau_0)$. If $a_1,\dots,a_l$ are such that
    \[ \gamma(\tau) = a_l(\cdots(a_1(s(g)))) \]
    and
    \[ h = a_1^{-1}(\cdots(a_l^{-1}(\gamma'(\tau)))) ,\]
    then
    \[ \gamma(\tau_0) = g(a_l(\cdots(a_1(s(g))))) \]
    and
    \[ h = a_1^{-1}(\cdots(a_l^{-1}(g^{-1}(\gamma'(\tau_0))))) .\]
    This shows that $\tau_0 \in T$, proving that $T$ is closed.

    We have proved that $1\in T$, so we can find a sequence $a_1, \dots, a_l$ such that
    \[ g = a_l(\cdots(a_1(s(g)))) = a_l(\cdots(a_2 a_1)) \]
    and
    \[ h = a_1^{-1}(\cdots(a_l^{-1}(\gamma'(1)))) = a_1^{-1}(\cdots(a_l^{-1}(gh))) .\]
 \end{proof}

\subsection{The Lie algebroid of a local Lie groupoid}
\label{sec:algebroid}

The construction of the Lie algebroid of a local Lie groupoid follows the same
pattern as for a Lie groupoid, but there are a few subtleties related to
left-invariant and right-invariant vector fields that are relevant for us.

Given a local Lie groupoid $G$ over $M$, just like for a Lie groupoid, we
consider the vector bundle
\[ A=T^s_M G \]
over $M$ and define the \emph{anchor} $\rho:A\to TM$ to be $\rho:=\d t|_A$.
Given a section $\al\in\Gamma(A)$, we associate to it the vector field
$\widetilde{\al}$ on $G$ given by:
\[ \widetilde{\al}_g=\d_{t(g)} R_g (\al_{t(g)}), \]
where $R_g$ denotes right translation by $g$. This is well-defined since $R_g$
is a smooth map from a neighborhood of $t(g)$ in the source fiber
$s^{-1}(t(g))$ to a neighborhood of $g$ in the source fiber $s^{-1}(s(g))$. 

One may wonder if the vector field $\widetilde{\al}$ is right-invariant. We
will say that a vector field $X\in\X(G)$ is \emph{right-invariant} if it is
tangent to the source fibers and:
\[ \d R_h(X_g)=X_{gh}, \quad \forall (g,h)\in \U. \]
Clearly, a right-invariant vector field $X$ is of the form $\widetilde{\al}$,
where $\alpha=X|_M\in\Gamma(A)$. The converse holds for a right-regular local
Lie groupoid $G$, and we have:

\begin{lemma}
  Let $G$ be a right-regular local Lie groupoid.  There is a 1:1 correspondence between right-invariant vector fields 
  in $G$ and sections of $A$ given by:
  \[ \X_{\textrm{R-inv}}(G)\to \Gamma(A),\quad X\mapsto X|_M. \]
  The inverse associates to $\al\in\Gamma(A)$ the vector field $\widetilde{\al}\in\X_{\textrm{R-inv}}(G)$.
\end{lemma}

\begin{proof}
  All we have to show is that for any section $\al\in\Gamma(A)$ the vector field $\widetilde{\al}$ is right-invariant. 
  For a right-regular local Lie groupoid this follows from associativity in the usual way.
\end{proof}

The Lie bracket of right-invariant vector fields is clearly a right-invariant vector field. Hence, if $G$ is a right-regular local Lie groupoid we have an induced Lie bracket $[~,~]$ on the space of sections $\Gamma(A)$ and the triple $(A,\rho,[~,~])$ is easily seen to be a Lie algebroid. Observing that any local Lie groupoid $G$ has a shrinking to a right-regular local Lie groupoid $G'$, we define:

\begin{definition}
  The \emph{Lie algebroid} of a local Lie groupoid $G$ is the Lie algebroid $(A,\rho,[~,~])$ of any right-regular shrinking of $G$.
\end{definition}

Obviously, all this discussion is valid by replacing ``right'' by ``left''
everywhere. For bi-regular local Lie groupoids there is 1:1 correspondence
between left-invariant and right-invariant vector fields. Moreover, every local
Lie groupoid admits a bi-regular shrinking.

\subsection{The Maurer-Cartan form of a local Lie groupoid}

For a bi-regular local Lie groupoid, one can define the Maurer-Cartan form much
the same way as one does in the case of a Lie groupoid (for the latter see,
e.g., \cite{coframes}). 

First of all, since we have chosen to define the Lie algebroid of a local Lie
groupoid using $s$-fibers and right-invariant vector fields, a
\emph{right-invariant} $k$-form on local Lie groupoid $G$ is an $s$-foliated
$k$-form satisfying
\[ (R_g)^*\omega=\omega, \quad \forall g\in G, \]
whenever this makes sense. More generally, given a vector bundle $E\to M$, a
right-invariant $k$-form on $G$ with values in $E$ is given by a bundle map
\[
\xymatrix{
\wedge^k T^s G \ar[r]^{\omega} \ar[d] & E\ar[d]\\
G\ar[r]_{t} & M}
\]
which, for any pair $(g,h)\in\U$, satisfies the invariance condition:
\[ \omega(\d R_g\cdot v_1,\dots, \d R_g\cdot v_k)=\omega(v_1,\dots,v_k), \quad v_1,\dots,v_k\in T^s_h G \]
(note that both sides are elements of $E_{t(g)}$).  The usual right-invariant
forms correspond to the case where $E=M\times\R\to M$ is the trivial line
bundle.

A right-invariant $k$-form $\omega$ on $G$ with values in a vector bundle $E$ determines an $A$-differential form with values in $E$:
\[ \omega|_M\in\Omega^k(A;E). \]
For a right-regular local Lie groupoid the assignment $\omega\mapsto \omega|_M$ establishes a bijection between right-invariant forms on $G$ with values in $E$ and $A$-differential forms with values in $E$.  

For a bi-regular local Lie groupoid we define its (right) \emph{Maurer-Cartan form} $\wmc$ to be the unique right-invariant 1-form on $G$ with values in the Lie algebroid $A$, whose values along the unit section is the identity map $A\to A$. Explicitly, the  Maurer-Cartan form is given by:
\[ \wmc : T^sG\to A, \quad v \mapsto (d_{t(g)}R_{g})^{-1}(v). \]
The left Maurer-Cartan form is defined analogously. If $G$ is not bi-regular, we can consider the restriction of $G$ to a bi-regular local Lie groupoid and define the Maurer-Cartan form of $G$ to be the one of the restriction.

The Maurer-Cartan form satisfies a version of the Maurer-Cartan equation for Lie groupoids/algebroids which reads:
\[ \d_\nabla \wmc+\frac{1}{2}[\wmc,\wmc]=0. \]
We will not get into details here. We refer the reader to \cite{coframes}, where it is proved that the Maurer-Cartan equation is equivalent to the statement that the bundle map:
\[
\xymatrix{
T^s G \ar[r]^{\wmc} \ar[d] & A\ar[d]\\
G\ar[r]_{t} & M}
\]
is a  morphism of Lie algebroids. We will use the fact that for a local Lie groupoid this statement still holds (just mimic the proof in \cite{coframes}).

Our next result is well-known in the case of Lie algebras and (local) Lie groups (see, e.g., Theorem 8.7 and Remark 8.8 in \cite{sharpe}). 
However, in the case of Lie algebroids and groupoids we do not know of any reference in the literature (but see the discussion in \cite{coframes} concerning the universal property of the Maurer-Cartan form of a Lie groupoid). 

\begin{theorem}
\label{thm:universal:covering}
  Let $G$ and $M$ be manifolds, let $s,t:G\to M$ be submersions with 1-connected fibers, and $u:M\to G$ a section of both $s$ and $t$. 
  Assume also that $A:=T^s_MG$ is an integrable Lie algebroid with anchor $\rho:=\d t|_A$ and one is given a Lie algebroid morphism:
  \[ \xymatrix{
     T^s G \ar[r]^{\omega} \ar[d] & A\ar[d]\\
     G\ar[r]_{t} & M}
  \]
  which is a fiberwise isomorphism with values along $M$ the identity $A\to A$. 
  Then there is a bi-regular local Lie groupoid structure on $G$ with structure
  maps $s,t,u$ for which the Maurer-Cartan form is $\omega$, unique up to
  restriction.
\end{theorem}

\begin{remark}
\label{rmk:MC:inverses}
The proof below shows that the local groupoid whose existence is stated in the theorem, in general, will have domain of inversion $\V$ a strict subset of $G$. This is one of the reasons that in the definition of a local Lie groupoid we do not require the domain of inversion to be all of $G$. Also, the groupoid constructed in the proof is not, in general, 3-associative (this will be clear in Example~\ref{example:3-associative}). However, as we observed before, it has a restriction which is 3-associative.
\end{remark}

\begin{proof}
  The Lie algebroid $T^s G$ is integrable and has source 1-connected integration $\Pi_1(s)$, the monodromy groupoid of 
  the $s$-foliation of $G$. Since we assume the $s$-fibers to be 1-connected, this groupoid coincides 
  with the submersion groupoid of $s:G\to M$:
  \[ \Pi_1(s)= G\timesss G\tto G. \]
  We denote by $\bar{s}=\pr_2$ and $\bar{t}=\pr_1$ the source and target maps of $G\timesss G$. 
  
  The Lie algebroid morphism $\omega:T^sG\to A$ integrates to a Lie groupoid morphism 
  \[
  \xymatrix@R=20pt{
  G\timesss G\ar[r]^{\Phi} \ar@<-.6ex>[d] \ar@<.6ex>[d] & \G(A) \ar@<-.6ex>[d] \ar@<.6ex>[d] \\
  G\ar[r]_{t} & M}
  \] 
  where $\G(A)$ is the source 1-connected integration of $A$. Explicitly, the map $\Phi$ takes a 
  pair $(g_1,g_2)$ to the A-homotopy class $[a]\in \G(A)$, with $a$ the $A$-path $a(t)=\omega(\gamma'(t))$, for any choice of
  path $\gamma:I\to G$ in the $s$-fiber connecting $g_1$ to $g_2$.
  
  We claim that $\Phi$ is a submersion. Indeed, its restriction to the units is $t$, 
  which is a submersion onto the unit section of $\G(A)$, and its restriction to the tangent to the source fibers along $G$ 
  is $\omega$, which is an isomorphism onto $A$, the tangent to the source fibers of $\G(A)$. This shows that $\Phi$ is a submersion in 
  a neighborhood of the unit section, but since it is a groupoid morphism it is a submersion everywhere.
  
  Consider now the embedding:
  \[ G\hookrightarrow G\timesss G, \quad g\mapsto (g,s(g)). \]
  Composing this embedding with $\Phi$, we obtain an \'etale map denoted by the same letter, $\Phi:G\to \G(A)$, which 
  commutes with the unit maps $M\to G$ and $M\to \G(A)$, and makes the following diagram commute:
  	\[ \xymatrix@R=20pt{
  	G\ar[r]^{\Phi} \ar@<-.6ex>[d]_{s} \ar@<.6ex>[d]^{t} & \G(A) \ar@<-.6ex>[d]_s \ar@<.6ex>[d]^t \\
  	M\ar[r]_{\textrm{id}} & M}
  	\] 
	
  Next we define a (local) multiplication on $G$ which makes $\Phi$ a morphism of local Lie groupoids. Let 
  $M\subset \V\subset \V'\subset G$ be open sets such that the restriction $\Phi|_{\V'}:\V'\to G$ is an embedding, $\Phi(\V)$ is invariant under inversion,
  and the product of elements in $\Phi(\V)$ belongs to $\Phi(\V')$. Also we choose a Riemannian metric on $G$ and we can shrink 
  $\V$ further so that for each $x\in M$ we have:
  \begin{enumerate}
  \item[(a)] $\V\cap s^{-1}(x)$ is a convex neighborhood of $x$ for the induced metric on $s^{-1}(x)$;
  \item[(b)] $\V\cap t^{-1}(x)$ is a convex neighborhood of $x$ for the induced metric on $t^{-1}(x)$.
  \end{enumerate}
  If $g\in \V$ we define its inverse to be $g^{-1}=\Phi^{-1}(\Phi(g)^{-1})$. We
  define multiplication in an open subset of $(G\timesst \V) \cup(\V\timesst G)$
  containing $(G\timesst M) \cup(M\timesst G)$ as follows:
  if $(g,h)\in G\timesst U$ let $\gamma_h(t)$ be the geodesic in the $t$-fiber connecting
  $t(h)$ to $h$. The product $gh$ is defined provided the curve
  $\Phi(g)\Phi(\gamma_h(t))$ belongs to the image of $\Phi$ and its $\Phi$-lift
  exists, in which case $gh$ is the endpoint of this lift.
  One similarly defines the product on some open subset of $\V\timesst G$.
  If $(g,h)\in \V\timesst \V$ the two definitions are compatible, since then
  $\Phi(g)\Phi(h)\in\Phi(\V')$ and $\Phi|_{\V'}$ is an embedding.
  
  This gives a structure of local Lie groupoid to $G$ making $\Phi:G\to \G(A)$ a
  local Lie groupoid morphism. Since $\Phi$ is \'etale, this implies that $G$
  is bi-regular with Lie algebroid $A$ and its Maurer-Cartan form is the given
  1-form.
%
\end{proof}

\section{Non-globalizable local Lie groupoids}
\label{ss:example-non-globalizable}

Before we discuss Mal'cev's theorem, we look at some examples of local Lie groupoids where associativity fails to some degree. Such groupoids are not globalizable in the following sense  (the term ``enlargeable'' is also used in the literature):

\begin{definition}
    A local Lie groupoid is called \emph{globalizable} if it is a restriction
    of an open neighborhood of the unit section in a Lie groupoid.
\end{definition}
 
Even if we consider only local Lie \emph{groups}, there are examples where
3-associativity fails, and which therefore are not globalizable.

\begin{example}[Failure of 3-associativity \cite{olver}]
\label{example:3-associative}
Let $X$ be the plane $\R^2$, with a small ball centered at $(1,0)$ taken out.
Let $G$ be the universal cover of $X$.
Let $\tilde{0}\in G$ (the unit in $G$) be a preimage of $0\in X$ under the projection.
Let $B \subset X$ be the open ball of radius $\frac34$ centered at $0$,
and let $\tilde{B}$ be the ball above $B$ containing $\tilde{0}$.

The multiplication will be defined on a subset of
$G\times\tilde{B} \cup \tilde{B}\times G$.
Let $(\tilde{g},\tilde{h}) \in G\times\tilde{B}$,
and write $g$ and $h$ for the projections to $X$. The product
$\tilde{g}\tilde{h}$ will be defined if the line segment from $g$ to
$g+h$ lies entirely in $X$, in which case the product $\tilde{g}\tilde{h}$ is
the endpoint of the lift of this line segment, starting at
$\tilde{g}$. Similarly, if $(\tilde{g},\tilde{h}) \in \tilde{B}\times\G$, the
product is the endpoint of the lift of the line segment from $h$ to
$g+h$, starting at $\tilde{h}$ (if that segment lies entirely in $X$).
Inverses are defined in the obvious way.

This local Lie group is locally associative but not not 3-associative.
By picking $a,b,c \in \tilde{B}$ appropriately, one can have $(ab)c\neq a(bc)$
as illustrated in Figure~\ref{fig:lallg}.

\begin{figure}[h]
    \centering
    \begin{overpic}[width=0.6\linewidth]{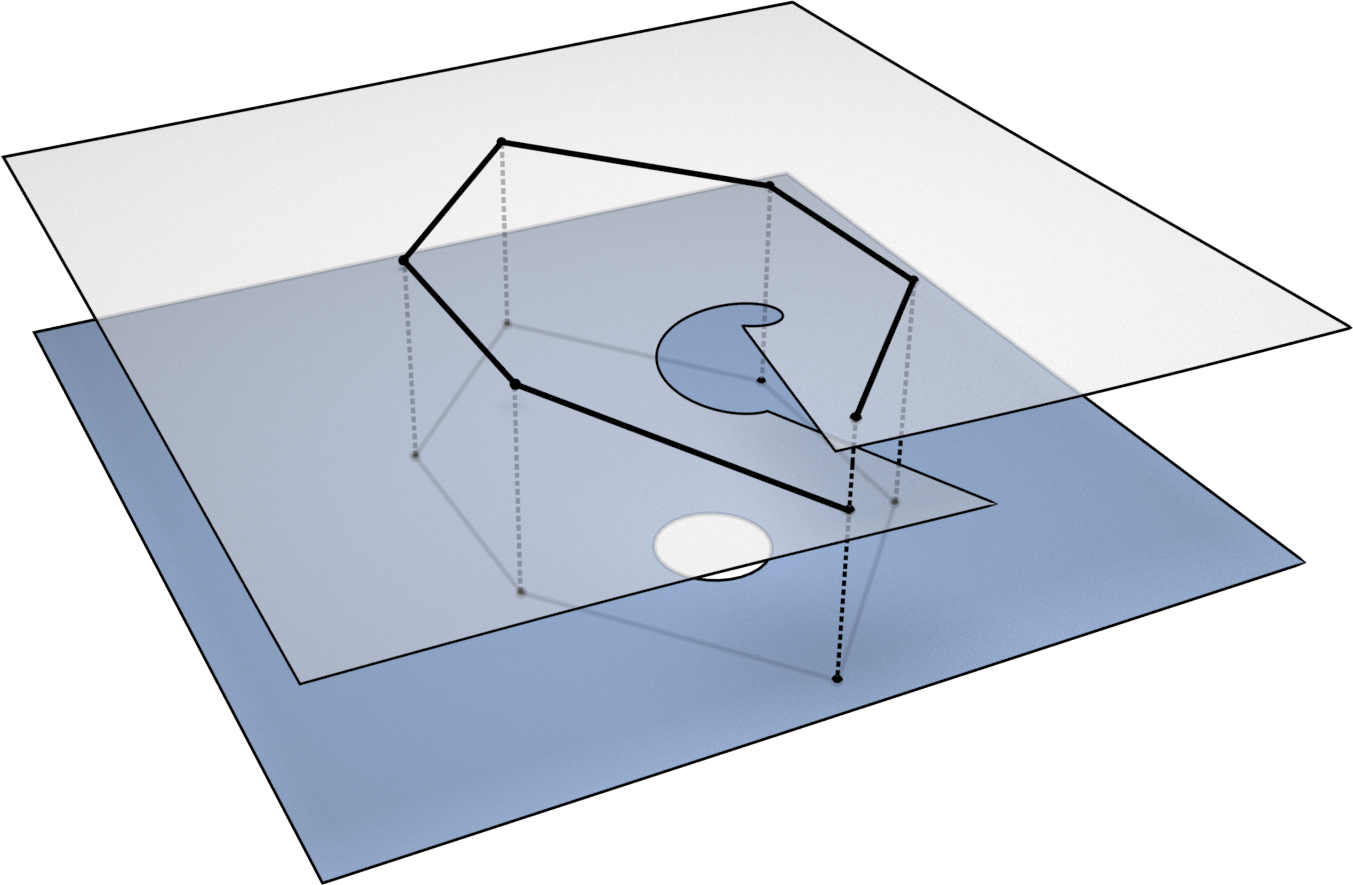}
        \put (37,56.3) {$\tilde{0}$}
    \end{overpic}
    \caption{A local Lie group $G$ that is not 3-associative.
    $G$ is the universal cover (shown on top) of the plane with a disk removed (shown at the bottom).
    Starting at the unit $\tilde{0}$, we can evaluate a triple
    product in two different ways, ending up over the same point of the plane,
    but on different sheets.}
    \label{fig:lallg}
\end{figure}
\end{example}

Notice that we can view this example as a special case of the construction
furnished in the proof of Theorem~\ref{thm:universal:covering}, where the Lie
algebroid $A$ is just the abelian 2-dimensional Lie algebra. Hence, the local
groupoid obtained there is, in general, not 3-associative.
The resulting group(oid)s always have a restriction that is 3-associative,
however.

We now turn to examples of local groupoids which are not globalizable.
Besides the failure of associativity, there is another obvious obstruction
to globalizability of a local Lie groupoid $G$:
the underlying Lie algebroid must be integrable.
One of our main aims will be to relate these two obstructions.
Still, the examples in this section have underlying Lie algebroids which are
integrable. In fact, they will all be local Lie groupoids integrating the
so-called \emph{pre-quantum Lie algebroid} $A_\omega$ of a closed 2-form
$\omega\in\Omega^2(M)$ (see \cite[Example 2.26]{lectures-integrability-lie}).

As a bundle, $A_\omega=TM \oplus \R$. The anchor is the projection to the first summand and the Lie bracket is given by
\[ [(X,f),(Y,g)] = ([X,Y],\mathcal{L}_X(g) - \mathcal{L}_Y(f) + \omega(X,Y)) .\]
The integrability of $A_\omega$ is controlled by the group of spherical periods of $\omega$:
\[ \N=\left\{\int_\eta \omega\mid[\eta]\in\pi_2(M)\right\}\subset \R. \]
It follows from the general integrability criteria that this algebroid is integrable if and only if $\N\subset\R$ is a discrete subgroup (see \cite[example 3.1]{lectures-integrability-lie}). In the examples below we consider the case where $M = \S^2$, equipped with its usual area form $\omega$ of total area $4\pi$, so $A_\omega$ is integrable.

\begin{example}[The source 1-connected integration]
Let us describe the global source 1-connected Lie groupoid integrating $A_\omega$. We will write $A$ for the map that calculates the area enclosed by a loop:
\[ A : \Omega(\S^2) \to \R/4\pi\Z,\quad \gamma \mapsto \int_{\Gamma} \omega ,\]
where $\gamma : [0,1] \to \S^2$ is a loop and $\Gamma : [0,1]\times[0,1] \to \S^2$
is any homotopy contracting $\gamma$ to the trivial loop
at $\gamma(0) = \gamma(1)$.
Note that this area is well-defined up to $4\pi$, because any two such homotopies will
have areas differing by an element $\int_{\eta} \omega \in 4\pi\Z$ where $[\eta]\in\pi_2(\S^2)$. 

Now, we set
\[ P = \{ \text{piecewise smooth paths in $\S^2$} \} \times \R/4\pi\Z ,\]
and we define an equivalence relation $\sim$ in $P$ by letting 
\[ (\gamma_1,a_1)\sim (\gamma_2,a_2)\quad\text{if}\quad 
\left\{
\begin{array}{l}
\gamma_1(0) = \gamma_2(0),\\ 
\gamma_1(1) = \gamma_2(1),\text{ and }\\
a_2 = a_1 + A(\gamma_2^{-1}\cdot\gamma_1)
\end{array}
\right.
\]
%
%
Here $\cdot$ represents concatenation of paths and $\gamma_2^{-1}$ denotes the reverse of $\gamma_2$. Then $\G = P/{\sim}$ is a smooth manifold of dimension 5.

We think of an equivalence class $[(\gamma, a)]$ as an arrow from $\gamma(0)$ to $\gamma(1)$, so we obtain a groupoid
\[ \G \tto \S^2 \]
with multiplication
\[ [\gamma_1,a_1] [\gamma_2,a_2] = [\gamma_1\cdot\gamma_2, a_1+a_2] .\]
One checks that the $s$-fibers of $\G$ are diffeomorphic to $\S^3$ so that $\G\simeq\G(A_\omega)$, the source 1-connected Lie groupoid with Lie algebroid $A_\omega$.
\end{example}

Our next example is just an open subset of the previous Lie groupoid. But it admits an alternative, more geometric, description.

\begin{example}[A globalizable local Lie groupoid] Let us denote by $G'\subset \G$ the open subset
consisting of arrows $[(\gamma,a)]$ whose source and target \emph{are not} antipodal: $\gamma(0)\ne -\gamma(1)$.
This is, of course, a (globalizable) local Lie groupoid.

Each element of $G'$ has a unique representative whose path is a geodesic for the usual round metric.
Hence, if we represent each $G'$ by its geodesic representative, we can write
down this local Lie groupoid explicitly as
\[ G' =\left\{ (y,x,a) \in \S^2\times \S^2\times \R/4\pi\Z \mid  x + y\neq 0 \right\}  .\]
The multiplication of $(z,y,a), (y,x,a') \in G'$ is then defined whenever $x+z\neq 0$,
and is given by
\[ (z,y,a) \cdot (y,x,a') = (z,x,a+a'+A(\Delta xyz)) ,\]
where $A(\Delta xyz) \in \R/4\pi$ is the signed area of the spherical triangle $\Delta xyz$.

In this notation the inclusion of $G'$ in $\G$, establishing its globalizability, becomes:
\[ G'\hookrightarrow \G:\quad  (y,x,a) \mapsto [\textrm{geodesic from $x$ to $y$}, a] .\]
\end{example}

By introducing a slight variation, we can make the previous local Lie groupoid non-globalizable.

\begin{example}[A non-globalizable local Lie groupoid]
\label{sss:non-glob}
We now define a local Lie groupoid $G'$ over $\S^2$, similar to $G'$ in the previous example, but
we will no longer quotient out the areas by $4\pi$:
\[ G'' = \left\{(x,y,a)\in \S^2\times \S^2\times \R \mid x+y\neq 0 \right\}.\]
We will define the multiplication of $(z,y,a)$ and $(y,x,a')$ only if $x+z\neq 0$
and $-\pi < A(\Delta xyz) < \pi$.
In that case, it is defined as above, by the formula
\[ (z,y,a) \cdot (y,x,a') = (z,x,a+a'+A(\Delta xyz)) .\]
This local Lie groupoid is still 3-associative: if $(z,y,a_1), (y,x,a_2), (x,w,a_3) \in G''$,
and we assume the products are defined, then
\begin{align*}
    ((z,y,a_1)\cdot(y,x,a_2))\cdot(x,w,a_3) &= (z,w,a_1+a_2+a_3+A(\Delta xyz)+A(\Delta wxz)) \\
    (z,y,a_1)\cdot((y,x,a_2)\cdot(x,w,a_3)) &= (z,w,a_1+a_2+a_3+A(\Delta wxy)+A(\Delta wyz))
\end{align*}
where
\[ A(\Delta xyz)+A(\Delta wxz)=A(\Delta wxy)+A(\Delta wyz) \pmod{4\pi}, \] 
because both sides equal the area of the quadrangle $wxyz$ (which is defined up to $4\pi$). Since we have
restricted the areas to be in $(-\pi,\pi)$, this implies 
\[ A(\Delta xyz)+A(\Delta wxz)=A(\Delta wxy)+A(\Delta wyz), \]
proving 3-associativity.

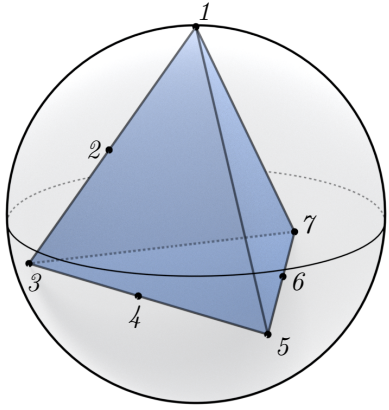
\begin{figure}
    \centering
    \begin{overpic}[width=0.4\linewidth]{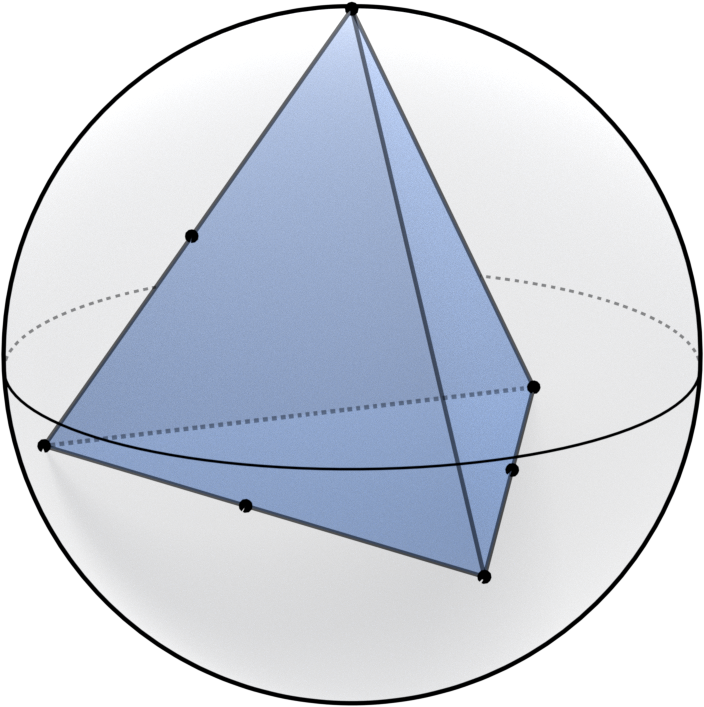}
        \put (50,100.6) {1}
        \put (21,64.5) {2}
        \put (5.4,29.8) {3}
        \put (31.9,21.5) {4}
        \put (70.5,13) {5}
        \put (74.4,30) {6}
        \put (77,45) {7}
    \end{overpic}
    \caption{Demonstrating the lack of 6-associativity. The tetrahedron is inscribed in a sphere, and the even-numbered points are midpoints of the edges they lie on.}
    \label{fig:tetrahedron}
\end{figure}

This local Lie groupoid is not globally associative.
Consider Figure~\ref{fig:tetrahedron}, which shows a regular tetrahedron inscribed in $\S^2$.
Let $x_j \in \S^2$ be the normalization of the $j$'th marked point in the picture (i.e.\ the radial projection of the point to the circumscribed sphere).
Let
\begin{align*}
    A &= (x_2,x_1,0) & D &= (x_5,x_4,0) \\
    B &= (x_3,x_2,0) & E &= (x_6,x_5,0) \\
    C &= (x_4,x_3,0) & F &= (x_7,x_6,0).
\end{align*}
Then
\[ (x_7,x_1,2\pi) = F(E((D(CB))A)) \neq ((F((ED)C))B)A = (x_7,x_1,-2\pi) .\]

Of course, global associativity is necessary for globalizability.
Our groupoid $G''$ is therefore not globalizable.
Note, however, that there is a neighborhood of the unit section in $G''$
that \emph{is} globalizable:
if $G''' = \{ (y,x,a) \in G'' \mid -2\pi < a < 2\pi \}$,
then $G'''$ is globalizable and in fact isomorphic to an open part of $G'$ and
hence of $\G$. It also follows that the Lie algebroid of $G''$ is still
$A_\omega$.
\end{example}

\section{Mal'cev's theorem for local Lie groupoids}
\label{sec:malcev}

In this section, we prove Mal'cev's theorem for local Lie groupoids. It states
that global associativity is the only obstruction for globalizability under the
appropriate connectedness assumptions:

\begin{theorem}[Mal'cev's Theorem for local Lie groupoids]
    A strongly connected local Lie groupoid is globalizable if and only if it is globally associative.
\end{theorem}

Our proof will be analogous to the proof for local groups (cf.~\cite{olver,EstLee}).  A little bit of extra effort is required in proving smoothness of the constructed groupoid. The main technical ingredient in the proof is the \emph{associative completion} $\AsCo(G)$ of a local Lie groupoid $G$, which we will now study in detail.

\subsection{Associative completions}
\label{ss:asco}

In this section, we review a construction that associates to a local groupoid a global groupoid. In the case of local groups this construction is well-known (see, e.g., \cite{EstLee}). For groupoids, it appears in \cite[appendix A]{bailey-gualtieri}, where it is called the \emph{formal completion} of the local groupoid (and it is used only for
globally associative structures).

We will call this construction the \emph{associative completion} of the local groupoid, to highlight that it has the effect of
\begin{itemize}
    \item making the structure globally associative,
    \item globalizing the structure to a (global) groupoid, so that the multiplication is complete.
\end{itemize}
Further justification is provided by the universal property satisfied by the associative completion (see Proposition~\ref{prop:AsCo}).

Suppose that $G$ is a local groupoid. A word $(w_1,\dots,w_n)\in G^n$ is called \emph{well-formed} if $s(w_i) = t(w_{i+1})$ for all $i\in\{1,\dots,n-1\}$. Let us denote the set of all well-formed words on $G$ by:
\[ W(G) = \bigsqcup_{n\ge 1} \underbrace{G \timesst G \timesst \cdots \timesst G}_{\text{$n$ copies}}.\]
Note that if $G$ happens to be an open part of a global Lie groupoid, then a groupoid globalizing it may be obtained as a quotient of $W(G)$ by an appropriate equivalence relation, as follows.

\begin{definition}
If $w = (w_1,\dots,w_k,w_{k+1},\dots,w_n)\in W(G)$ is a well-formed word and $(w_k,w_{k+1}) \in \mathcal{U}$,
we will say that the word
\[ w' = (w_1,\dots,w_kw_{k+1},\dots,w_n) \]
is obtained from $w$ by \emph{contraction} and that $w$ is obtained from $w'$ by \emph{expansion}.
If two words are related by an expansion/contraction, we will say that they are \emph{elementarily equivalent}. We denote by $\sim$ the equivalence relation on $W(G)$ generated by elementary equivalences. We will also write $w'\le w$ if 
there is a sequence of expansions starting at $w'$ and ending at $w$ ($\le$ is a partial order relation on $W(G)$).
\end{definition}

\begin{definition}
If $G$ is a local groupoid, then
\[ \AsCo(G) = W(G)/{\sim} \]
is called the \emph{associative completion} of $G$. The \emph{completion map} $G\to\AsCo(G)$ associates to an element $g\in G$ the equivalence class of the word $w=(g)$.
\end{definition}

When $G$ is a local topological groupoid, $W(G)$ is a disjoint union of topological spaces, so it is a topological space. We consider on $\AsCo(G)$ the quotient topology.

\begin{proposition}
\label{prop:AsCo}
    If $G$ is an inversional local topological groupoid, then $\AsCo(G)$ is a topological groupoid and $G\to \AsCo(G)$ is a 
    morphism of local topological groupoids. Given any morphism $F:G\to \HH$, where $\HH$ is a topological groupoid, there 
    exists a unique morphism of topological groupoids $\tilde{F}:\AsCo(G)\to\HH$ such that the following diagram commutes:
   \[ \xymatrix{ G\ar[r]^F\ar[d] & \HH \\
                       \AsCo(G)\ar@{-->}[ur]_{\tilde{F}}}\]
\end{proposition}

\begin{proof}
    The source and target maps of $\AsCo(G)$ are given by
    \[ s([w_1,\dots,w_n]) = s(w_n) \quad\text{and}\quad t([w_1,\dots,w_n]) = t(w_1) .\]
    The multiplication map is juxtaposition.
    The unit at $x\in M$ is given by $[x]$.
    It remains to show that $[w_1,\dots,w_n]$ has an inverse.
    Because $G$ is inversional, every $w_i$ can be written as a product of invertible elements
    $w_i^1,\dots,w_i^{k_i}$.
    Then we have
    \[ [w_1,\dots,w_n]^{-1} = [(w_n^{k_n})^{-1},\dots,(w_n^1)^{-1},\dots,(w_1^{k_1})^{-1},\dots,(w_1^1)^{-1}] .\]
    These operations turn $\AsCo(G)$ into a groupoid. 
        
    Next observe that the projection $W(G)\to \AsCo(G)$ is an open map. This means that if $O\subset W(G)$ is an open set, then its saturation:
    \[ \tilde{O}:=\{w\in W(G): \exists\, w'\in O \text{ such that }w\sim w'\} \]
    is also open and this follows because multiplication is both continuous and an open map. 
    Now we have commutative diagrams:
    \[ 
    \xymatrix{
    W(G)\ar[d]\ar[r]^{\bar{s}} & M\\
    \AsCo(G)\ar[ru]_s}\quad
    \xymatrix{
    W(G)\ar[d]\ar[r]^{\bar{t}} & M\\
    \AsCo(G)\ar[ru]_t}\quad 
    \xymatrix{
    W(G) \tensor[_{\bar{s}}]{\times}{_{\bar{t}}}W(G) \ar[d]\ar[r]^----{\bar{m}} & W(G)\ar[d]\\
    \AsCo(G) \tensor[_s]{\times}{_t} \AsCo(G) \ar[r]_----m& \AsCo(G)}
    \]
    and since $\bar{s}$, $\bar{t}$ and $\bar{m}$ are both open and continuous, it follows that $s$, $t$ and $m$ are also open and continuous. 
    One shows similarly that the unit map and inversion map are continuous, so that $\AsCo(G)$ is a topological groupoid. 
    
    Finally, given a morphism of local topological groupoids $F:G\to \HH$, where $\HH$ is a groupoid, we define $\bar{F}:W(G)\to \HH$ by:
    \[ \bar{F}(w_1,\dots,w_n)=F(w_1)\cdots F(w_n)\]
    This is clearly continuous, so induces a continuous map $\tilde{F}:\AsCo(G)\to \HH$ such that:
       \[ \xymatrix{ W(G)\ar[r]^{\bar{F}}\ar[d] & \HH \\
                       \AsCo(G)\ar[ur]_{\tilde{F}}}\]
    It should be clear that $\tilde{F}:\AsCo(G)\to \HH$ is the unique map making the diagram in the statement commute. 
\end{proof}

If $F:G\to H$ is a morphism of local topological groupoids, there is an obvious map  $\AsCo(F):\AsCo(G)\to\AsCo(H)$, which is a morphism of topological groupoids, and makes the following diagram commutative:
\[
\xymatrix{
G\ar[d]\ar[r]^F & H\ar[d] \\
\AsCo(G)\ar[r]_{\AsCo(F)} & \AsCo(H)
}
\]
Hence, $\AsCo$ is a functor from the category of inversional local topological groupoids to the category of topological groupoids. By the proposition above, it is left-adjoint to the forgetful functor from the category of topological groupoids to the category of inversional local topological groupoids.


\subsection{Associators}
\label{ss:associators}

If $G$ is a local Lie groupoid, then $W(G)$ is a disjoint union of manifolds
(of different dimensions) and it is natural to wonder if $\AsCo(G)$ inherits a
quotient differential structure so that it becomes a Lie groupoid. The answer
to this question is intimately related to the properties of the kernel of the
map $G \to \AsCo(G)$ (i.e.\ those elements that are mapped to a unit), which we
now study.

The following elements are clearly in this kernel.

\begin{definition}
    Suppose that $G$ is a local Lie groupoid over $M$ and that $x\in M$.
    An element $g\in G_x = s^{-1}(x)\cap t^{-1}(x) \subset G$ is
    called an \emph{associator} at $x$ if there is a word
    \[w= (w_1,\dots,w_k) \in W(G) \]
    that can be evaluated to both $x$ and $g$, i.e., such that $(x)\le w$ and $(g)\le w$. We write
    $\Assoc_x(G)$
    for the set of all associators at $x$.
    We write $\Assoc(G)$ for the set of all associators in $G$.
\end{definition}


In general the kernel of $G \to \AsCo(G)$ will contain other elements. However, under mild connectedness assumptions, the kernel is made only of associators.

\begin{proposition}
\label{prop:equivalence-relation}
    Suppose $G$ is a bi-regular local Lie groupoid that has products connected to the axes.
    For any two words $w_1,w_2\in\W(G)$ one has $w_1 \sim w_2$ if and only if there is $w_3\in W(G)$ such that $w_1\le w_3$ and $w_2\le w_3$. 
\end{proposition}
    
\begin{proof}
    One implication is obvious. To prove the other implication, we show that if $w_1,w_2\in\W(G)$ and $w_1 \sim w_2$ then 
    there is $w_3\in W(G)$ using induction on the number of elementary equivalences that make up the equivalence $w_1 \sim w_2$.
    
    If $w_1 \sim w_2$ are elementary equivalent, then it is clear that either $w_1\le w_2$, so we take $w_3=w_2$, or $w_2\le w_1$, 
    so we take $w_3=w_1$, and we are done with the first step of the induction.
    
    Assume that we proved the result if two words are connected by $N$ elementary equivalences. Let $w_1\sim w_2$ through $N+1$ 
    elementary equivalences. Then we have that $w_1\sim w'_1$ through $N$ elementary equivalences, and 
    $w'_1\sim w_2$ is an elementary equivalence. By the induction hypothesis, there is a word $w'_2$ such that $w_1\le w'_2$ and $w_1'\le w'_2$. 
    Now there are two cases:
    \vskip 5 pt
    
    {\em Case 1: $w_2\le w'_1$.} We can take $w_3=w'_2$, so we have $w_1\le w_3$ and $w_2\le w'_1\le w_3$.
    \vskip 5 pt
    
    {\em Case 2: $w'_1\le w_2$.} In this case we can write: 
    \begin{align*}
    w'_1&=(u_1,\dots,u_m), \\
    w_2&=(u_1,\dots,u_{i-1},g,h,u_{i+1},\dots,u_m),\quad \text{with }u_i=gh,\\
    w'_2&=v_1\cdot \ldots \cdot v_m,\quad \text{with }(u_j)\le v_j\text{ for }j=1,\dots,m, 
    \end{align*}
    where $u_j\in G$, $v_j\in W(G)$ and the $\cdot$ means concatenation of words. Now, applying Proposition~\ref{prop:inverses:decomp}, 
    the pair $(g,h)$ either satisfies \eqref{eq:inverses:decomp:1} in that proposition, in which case we set:
    \[ w_3= v_1\cdot\ldots\cdot v_{i-1} \cdot (a_l,\dots,a_1,a_1^{-1},\dots,a_l^{-1})\cdot v_i\cdot\ldots\cdot v_m, \]
    or the pair $(g,h)$ satisfies \eqref{eq:inverses:decomp:2} in that proposition, in which case we set:
    \[ w_3= v_1\cdot\ldots\cdot v_i \cdot (b_1^{-1},\dots,b_l^{-1},b_l,\dots,b_1)\cdot v_{i+1}\cdot\ldots\cdot v_m. \]
    For example, in the first case, we check that:
    \begin{align*}
    w_1&\le w_2'=v_1\cdot \ldots \cdot v_m\\
    		     &\le v_1\cdot\ldots\cdot v_{i-1} \cdot (a_l,\dots,a_1,a_1^{-1},\dots,a_l^{-1})\cdot v_i\cdot\ldots\cdot v_m=w_3,\\
    w_2&=(u_1,\dots,u_{i-1},g,h,u_{i+1},\dots,u_m)\\
     	  &=(u_1,\dots,u_{i-1},a_l(\cdots(a_2 a_1)),a_1^{-1}(\cdots(a_l (gh))),u_{i+1},\dots,u_m)\\
    	  &\le v_1\cdot\ldots\cdot v_{i-1} \cdot (a_l,\cdots,a_1,a_1^{-1},\dots,a_l, u_i)\cdot v_{i+1}\cdot\ldots\cdot v_m\\
	  &\le v_1\cdot\ldots\cdot v_{i-1} \cdot (a_l,\dots,a_1,a_1^{-1},\dots,a_l^{-1})\cdot v_i\cdot\ldots\cdot v_m=w_3.
    \end{align*}
    A similar argument for the second case, shows that $w_3$ satisfies $w_1\le w_3$ and $w_2\le w_3$, so the proposition holds.
\end{proof}

\begin{corollary}
\label{cor:kernel:completion}
    For a bi-regular local Lie groupoid $G$ with products connected to the axes,
    the kernel of $G\to\AsCo(G)$ is precisely $\Assoc(G)$.
\end{corollary}

The product of two associators (if defined) is again an associator. In fact, $\Assoc(G)$ has the structure of a local groupoid 
(free from topology) where source and target coincide, so it maybe thought of as a bundle of local groups.

Restricting a local Lie groupoid does not change its associators. More precisely, we have:

\begin{lemma}
\label{lem:assoc:restriction}
    Let $G$ be a bi-regular local Lie groupoid over $M$, with multiplication $m : \U \to G$
    and inversion map $i : \V \to G$. Let $G'$ be a restriction of $G$ with 
    multiplication map $m' : \U' \to G$ and inversion map $i' : \V' \to G$.
    If both $G$ and $G'$ have products connected to the axes, 
    then for all $x\in M$ we have
    \[ \Assoc_x(G) = \Assoc_{x}(G') .\]
\end{lemma}

\begin{proof}
    We consider two equivalence relations on $W(G)$:
    we will write ${\sim}$ for the equivalence relation on $W(G)$ generated by contractions
    and expansions for the multiplication $m$ on $G$,
    and ${\sim'}$ for the equivalence relation on $W(G)$ generated by contractions and
    expansions for the multiplication $m'$ on $G'$.
    These are precisely the equivalence relations that were considered in the construction
    of the associative completions of $G$ and $G'$.
    Proposition~\ref{prop:equivalence-relation} tells us that
    \[ g\in \Assoc_{x}(G) \iff (g)\sim (x) \quad\text{and}\quad
    g\in \Assoc_{x}(G') \iff (g)\sim' (x) .\]
    It therefore suffices to show that ${\sim}$ and ${\sim'}$ are the same equivalence relation.

    Clearly, $w\sim' v \Rightarrow w\sim v$.
    We claim that the converse implication also holds, so that the two equivalence relations are equal.
    It suffices to show that two words that are \emph{elementarily} equivalent for ${\sim}$
    are also equivalent for ${\sim'}$.
    Let $(w_1,\dots,w_k,w_{k+1},\dots,w_n) \in W(G)$ such that $(w_k,w_{k+1}) \in \U$ and 
    write $w_kw_{k+1}$ for $m(w_k,w_{k+1})$.
    Following along with the proof of Proposition~\ref{prop:inverses:decomp} \emph{for the local groupoid $G'$},
    we find that we can expand, for the multiplication $m'$, the subword
    $(w_k,w_{k+1})$ into
    \[ (a_l,\dots,a_1,a_1^{-1},\dots,a_l^{-1},w_kw_{k+1})\quad\text{ or }\quad (w_kw_{k+1},b_1^{-1},\dots,b_l^{-1},b_l,\dots,b_1),   \]
    where the inverses are inverses for $G'$ (and thus also for $G$, but that
    is not important).  We can then contract this word for the multiplication
    $m'$ into $(w_kw_{k+1})$, so that 
    \[ (w_1,\dots,w_k,w_{k+1},\dots,w_n) \sim' (w_1,\dots,w_kw_{k+1},\dots,w_n). \]
    This shows that ${\sim}$ and ${\sim'}$ coincide, proving the result.
\end{proof}

We note, however, that shrinking a local Lie groupoid (i.e.\ replacing it with
a neighborhood of $M$ in $G$) \emph{can} change the associators drastically
(see Section~\ref{ss:associators:examples}).

\subsection{Smoothness of $\AsCo(G)$}
\label{ss:smothness of AC}

For a bi-regular local Lie groupoid $G$ with products connected to the axes, we have the following relations:
\begin{align*}
    G \to &\,\AsCo(G)\text{ is injective} \\
    \iff & \text{$G$ is globally associative} \\
    \Longrightarrow\; & \text{$\Assoc(G)$ is trivial} 
\end{align*}

Note that triviality of $\Assoc(G)$ does not imply that $G$ is globally associative:
a morphism of local Lie groupoids with trivial kernel is not necessarily injective.
For example, the restriction of the morphism of Lie groups $\R \to \R/\Z$
to the interval $(-0.6,0.6)$ has trivial kernel, but it is not injective.

\begin{definition}
    We say that $\Assoc(G) \subset G$ is \emph{uniformly discrete} if
    there is an open neighborhood $U$ of $M$ in $G$ such that
    $U\cap\Assoc(G) = M$.
\end{definition}

We can now describe a necessary and sufficient condition for $\AsCo(G)$ to be a Lie groupoid. The proof will be given in the 
next section.

\begin{theorem}[Smoothness of $\AsCo(G)$]
    If $G$ is a bi-regular local Lie groupoid with products connected to the axes,
    then $\AsCo(G)$ is
    smooth if and only if $\Assoc(G)$ is uniformly discrete in $G$.
    Moreover, in that case, $G \to \AsCo(G)$ is a local diffeomorphism so $G$ and $\AsCo(G)$ 
    have isomorphic Lie algebroids.
    \label{thm:smoothness-of-asco}
\end{theorem}

By ``$\AsCo(G)$ is smooth'', we mean, of course, that $\AsCo(G)$ has a smooth
structure as a quotient of $W(G)$, where $W(G)$ is considered as a manifold with
components of various dimensions (so the restriction of the projection to each component of $W(G)$ is a submersion).

The following corollary characterizes globalizable local Lie groupoids.
The result is entirely analogous to the situation for local Lie groups.

\begin{corollary}[Mal'cev's theorem for local Lie groupoids]
    A strongly connected local Lie groupoid is globalizable if and only if it is globally associative.
    \label{cor:malcev-for-groupoids}
\end{corollary}
\begin{proof}
    Call the local Lie groupoid $G$.
    If $G$ globalizable, it is clearly globally associative.
    If $G$ is globally associative, then restrict it to get a local Lie groupoid
    $G'$ with products connected to the axes.
    Then $G' \hookrightarrow \AsCo(G')$ is the inclusion
    of $G'$ onto an open set $U$ of a global Lie groupoid
    (injectivity follows from global associativity).
    This means that $G$ and $U$ have a common restriction $G'$, so that
    by strong connectedness and Lemma~\ref{lem:stronglyconnecteddeterminedbyrestriction}
    the inclusion $G\hookrightarrow\AsCo(G')$ is also the inclusion of a
    restriction of $U$.
    The Lie groupoid $\AsCo(G')$ therefore globalizes $G$.
\end{proof}

\begin{remark}
\label{rmk:Malcev:inverses} 
In \cite{EstLee}, a version of Mal'cev's theorem is proved for local groups (free
from topology). The authors observe that their proof works for groupoids.
Our version differs from the result in this paper in two significant ways: (i)
our notion of local groupoid does not assume that every element is invertible
(unlike \cite{EstLee}) and (ii) our result concerns smooth local Lie groupoids.
For these reasons, our result also requires different assumptions and is
similar in spirit to the version proved in \cite{olver}.
\end{remark}

Every Lie algebroid can be integrated to a local Lie groupoid \cite[Corollary 5.1]{integ-of-lie-article}.
Because the globalizability of a local Lie groupoid is related to its higher associativity,
this points to a link between associativity and integrability. The following corollary
is a first indication.

\begin{corollary}
    A Lie algebroid is integrable if and only if
    there is a local Lie groupoid integrating it that has uniformly discrete associators.
\end{corollary}
\begin{proof}
    Suppose a Lie algebroid is integrable. Then by definition,
    there is a Lie groupoid integrating it, and this is a local Lie groupoid with uniformly discrete
    (indeed, trivial) associators.

    Suppose that a Lie algebroid has a local Lie groupoid integrating it, with uniformly discrete associators.
    Then by taking a small enough open neighborhood of the units in this local Lie groupoid, we can get
    a bi-regular local Lie groupoid, with products connected to the axes and trivial associators, which integrates the Lie algebroid.
    The associative completion of this local Lie groupoid is a
    Lie groupoid integrating our Lie algebroid.
\end{proof}

\subsection{Proof of smoothness of $\AsCo(G)$}
\label{ss:proof-malcev}

    We now turn to the proof of Theorem~\ref{thm:smoothness-of-asco}. 
    Suppose first that $\AsCo(G)$ has a smooth structure for which $W(G) \to \AsCo(G)$ is a 
    submersion.
    Note that the fibers of $G \to \AsCo(G)$ are countable,
    so that $\dim(\AsCo(G)) \geq \dim(G)$. Because $W(G)$ has a component
    of dimension $\dim(G)$ and $\AsCo(G)$ is a quotient of $W(G)$,
    we must have $\dim(\AsCo(G)) \leq \dim(G)$, and therefore
    $\dim(\AsCo(G)) = \dim(G)$ (this is of course the only reasonable dimension
    to expect).
    The map $G \to \AsCo(G)$ is therefore a local diffeomorphism.
    The inverse image of $M \subset \AsCo(G)$ under $G \to \AsCo(G)$,
    which is precisely $\Assoc(G)$, is therefore
    an embedded submanifold of $G$. This shows that $\Assoc(G)$ is uniformly discrete
    in $G$.

    Suppose now that $\Assoc(G)$ is uniformly discrete.
    We will prove that $\AsCo(G)$ is smooth.
    Given an element $h\in\AsCo(G)$, represented by a word $(x_1,\dots,x_n)$, we will construct
    a chart near $h$ modeled on an open part of $G$.
    For each $x_i$, take a small submanifold $N_i \subset G$ of the same dimension as $M$
    through $x_i$ that is
    transverse to both source and target fibers (i.e.\ a local bisection of the local Lie groupoid near $x_i$).
    Then $s$ and $t$ define diffeomorphisms between the $N_i$ and open subsets of $M$.
    Write $s_i = s{\restriction_{N_i}}$ and $t_i = t{\restriction_{N_i}}$ for these diffeomorphisms.

    Let $k\in\{1,\dots,n\}$.
    Let $U$ be a small neighborhood of $x_k$.
    Then the chart we use near $h$ is defined as
    \[ \varphi : U \to \AsCo(G) : y \mapsto [y_1,\dots,y_{k-1},y,y_{k+1},\dots,y_n] \]
    where
    \[
        y_{k+1} = t_{k+1}^{-1}(s(y)), \ 
        y_{k+2} = t_{k+2}^{-1}(s(y_{k+1})),
        \dots, \ 
        y_n = t_n^{-1}(s(y_{n-1})),
    \]
    and
    \[
        y_{k-1} = s_{k-1}^{-1}(t(y)), \ 
        y_{k-2} = s_{k-2}^{-1}(t(y_{k-1})),
        \dots, \
        y_1 = s_1^{-1}(t(y_2)).
    \]
    If we pick $U$ sufficiently small, this is well-defined.
    Note that it maps $x_k$ to $h$.
    We claim that $\varphi$ is injective, if we pick $U$ small enough
    (this is where we need the uniform discreteness).
    Indeed, suppose that $\varphi(y) = \varphi(z)$.
    Write the representatives of $\varphi(y)$ and $\varphi(z)$ as given above as
    \[ (y_1,\dots,y_{k-1},y,y_{k+1},\dots,y_n) \quad\text{and}\quad (z_1,\dots,z_{k-1},z,z_{k+1},\dots,z_n) .\]
    Because $\varphi(y) = \varphi(z)$, we must have $s(\varphi(y)) = s(\varphi(z))$ and
    therefore $s(y_n) = s(z_n)$.
    But then $y_n$ and $z_n$ are both on $N_n$, and have the same source. Therefore, $y_n = z_n$.
    This implies that $s(y_{n-1}) = s(z_{n-1})$.
    But then $y_{n-1}$ and $z_{n-1}$ are both on $N_{n-1}$, and have the same source.
    Therefore, $y_{n-1} = z_{n-1}$.
    Continuing this way, we find that $y_i = z_i$ for $i > k$.
    The same argument applied starting from the left and using the target map shows that
    $y_i = z_i$ for $i<k$.
    We conclude that
    \[ [y_1,\dots,y_{k-1},y,y_{k+1},\dots,y_n] = \varphi(y) = \varphi(z) = [y_1,\dots,y_{k-1},z,y_{k+1},\dots,y_n] .\]
    Multiplying both sides by $[y_1,\dots,y_{k-1}]^{-1}$ (from the left) and by $[y_{k+1},\dots,y_n]^{-1}$ (from the
    right) shows that $[y] = [z]$.
    If $B$ is a neighborhood of $M$ in $G$ such that $B\cap \Assoc(G) = M$, then for $U$ small enough,
    we must have $z=y\delta$ for $\delta$ in $B$.
    But then $[\delta]$ must be trivial in $\AsCo(G)$ so that, by Corollary~\ref{cor:kernel:completion}, it is an associator.
    Therefore $\delta \in M$ and we conclude that $y=z$.

    Let us check smoothness of the transition maps.
    Given a point in $\AsCo(G)$, a chart is determined by
    \begin{itemize}
        \item a choice of representative,
        \item a choice of $k$,
        \item a choice of local bisections.
    \end{itemize}
    We will show smoothness for each of the following types of transition maps:
    \begin{enumerate}
        \item[(T1)] for a fixed representative and index $k$, the transition between charts coming
            from different choices of local bisections,
        \item[(T2)] for a fixed representative, the transition between charts for two different
            choices of $k$,
        \item[(T3)] the transition between charts constructed using different representatives of
            the same element of $\AsCo(G)$.
    \end{enumerate}
    Together, these show that all the transitions are smooth.

    Let us start by checking the first type of transition map (T1).
    Suppose $(x_1,\dots,x_n)$ represents $h\in\AsCo(G)$. Pick $k\in\{1,\dots,n\}$.
    Making one choice of local bisections $N_i$ leads to a chart
    \[ \varphi : y \mapsto [y_1,\dots,y_{k-1},y,y_{k+1},\dots,y_n] .\]
    Note that the $y_i$ depend smoothly on $y$.
    Now make a different choice of local bisections $N_i'$, leading to
    a second chart.
    We will calculate the transition map between the two charts.
    Suppose $y$ is near $x_k$ and let $z_n$ be the point on $N_n'$ with source $s(y_n) = s(h)$.
    Then
    \[ z_n = \eps_n y_n \]
    for some invertible $\eps_n$ near $t(y_n)$ (on the condition that
    $y$ is sufficiently close to $x_k$).
    Note that $\eps_n=\eps_n(y)$ depends smoothly on $y$ and $\eps_n(x_k) = t(y_n)$.
    The product $y_{n-1}\eps_n^{-1}$ is defined for $y$ close enough to $x_k$.
    Let $z_{n-1}$ be the point on $N_{n-1}'$ with source $t(z_n)$.
    Then
    \[ z_{n-1} = \eps_{n-1} (y_{n-1} \eps_n^{-1}) \]
    for some invertible $\eps_{n-1}$ near $t(y_{n-1})$ (on the condition that
    $y$ is sufficiently close to $x_k$).
    Note that $\eps_{n-1}$ depends smoothly on $y$, and $\eps_{n-1} = t(y_{n-1})$
    if $y = x_k$.
    Continuing this way, we construct a sequence $z_n, \dots, z_{k+1}$ on the $N_i'$ and
    $\eps_n, \dots, \eps_{k+1}$
    such that
    \[ z_i = \eps_i(y_i\eps_{i+1}^{-1}) ,\]
    where we will set $\eps_{n+1} = s(y_n)$ for ease of notation.
    All the $\eps_i$ depend smoothly on $y$ (we may have to restrict $y$ to be in a
    smaller and smaller neighborhood of $x_k$).
    Similarly, we start from the left and construct a sequence $z_1, \dots, z_{k-1}$
    on the $N_i'$ and $\eps_1, \dots, \eps_{k-1}$ such that
    \[ z_i = (\eps_{i-1}^{-1}y_i)\eps_i \]
    with all the $\eps_i$ depending smoothly on $y$.
    Note now that
    \begin{align*}
        \varphi(y) &= [y_1,\dots,y_{k-1},y,y_{k+1},\dots,y_n] \\
                   &= [y_1,\eps_1,\dots,\eps_{k-2}^{-1},y_{k-1},\eps_{k-1},y,\eps_{k+1}, y_{k+1},\eps_{k+2}^{-1},\dots,\eps_n,y_n] \\
                   &= [z_1,\dots,z_{k-1},(\eps_{k-1}y\eps_{k+1}),z_{k+1},\dots,z_n] .
    \end{align*}
    where the product in the middle will be defined if $y$ is sufficiently close to $x_k$.
    This shows that the transition map between our two charts is given by
    \[ y \mapsto \eps_{k-1}y\eps_{k+1} ,\]
    which is smooth in $y$ (because the $\eps_i$ are).
    This proves smoothness for type (T1).
    
    Let us now check smoothness for type (T2). We consider transition between charts for the same
    representative, but different choices of $k$. Clearly, it is enough to consider the case where 
    the two choices for $k$ differ by 1, so we look at two charts associated to the same representative
    $(x_1,\dots,x_n)$ of $h$ for index $k$ and index $k+1$.
    Again, write
    \[ y \mapsto [y_1,\dots,y,y_{k+1},\dots,y_n] \]
    for the chart associated to index $k$.
    Suppose that $y$ is near $x_k$.
    Let $y_k$ be the point on $N_k$ with target $t(y)$.
    This point depends smoothly on $y$.
    For $y$ sufficiently close to $x_k$, we may write
    $y_k = y\eps$
    where $\eps$ is near $s(x_k)$ and depends smoothly on $y$.
    Then
    \begin{align*}
        [y_1,\dots,y,y_{k+1},\dots,y_n] &= [y_1,\dots,y,\eps,\eps^{-1},y_{k+1},\dots,y_n] \\
                                   &= [y_1,\dots,y_k,(\eps^{-1}y_{k+1}),\dots,y_n]
    \end{align*}
    for $y$ sufficiently close to $x_k$.
    This shows that the transition map is $y\mapsto \eps^{-1}y_{k+1}$, which is smooth.

    We now check the smoothness of the third type of transition (T3). So we now have
    different representatives of the same point $h$.
    This type is the easiest to check.
    Indeed, we only need to check smoothness for the transition for a single contraction
    or expansion, but that is clear from the smoothness of multiplication in $G$.

    Note, finally, that $\AsCo(G)$ is second-countable because $W(G)$ is. It may fail to be Hausdorff (as does $G$), but its source
    fibers and target fibers, as well as unit manifold, are Hausdorff.  
    Hence, we have proved that $\AsCo(G)$ is smooth and this completes the proof of Theorem~\ref{thm:smoothness-of-asco}. \qed

\section{Classification of local Lie groupoids}
\label{sec:classification}

In this section, we classify bi-regular local Lie groupoids with integrable algebroids.
By ``classifying'' we mean that we give a construction that, starting from a
global Lie groupoid, can be used to obtain all bi-regular local Lie groupoids
with the same Lie algebroid.  This section is a generalization to groupoids of
results by Olver \cite[Theorems 19 and 21]{olver}.

\subsection{Coverings of local Lie groupoids}

Just like in the theory of Lie groups and Lie groupoids, coverings play an
important role for local Lie groupoids. 

For groupoids, the notion of cover does not refer to the space of arrows, but
rather to its source/target fibers. In this section we will assume that source
and target maps have connected fibers, as do their covers in the sense of the
following definition:

\begin{definition}
    Suppose $G$ is a local Lie groupoid over $M$, with source and target maps
    $s$ and $t$, and unit section $u$.
    A \emph{covering} (resp.\ \emph{generalized covering}) of $G$ is a tuple
    $(G',s',t',u')$ and a map $\varphi:G'\to G$, where:
    \begin{itemize}
        \item $u':M \into G'$ is an embedding,
        \item $s, t : G' \to M$ are surjective submersions such that $s'\circ u' =t'\circ u'= \id$,
        \item $\varphi : G' \to G$ is a smooth map such that $s' = s\circ \varphi$, $t' = t\circ \varphi$, $u=\varphi\circ u'$,
        \item for each $m \in M$, the map $\varphi{\restriction_{(s')^{-1}(m)}} : (s')^{-1}(m) \to s^{-1}(\varphi(m))$ is a covering map (resp.\ local diffeomorphism).
    \end{itemize}
\end{definition}

An important fact is the following:

\begin{theorem}
\label{thm:covers}
    Any generalized covering $\varphi : G' \to G$ of a bi-regular local Lie groupoid $G$ has a structure of a local Lie groupoid for which $\varphi$ is 
    an \'etale morphism of local Lie groupoids. Moreover, $G'$ is also bi-regular.
\end{theorem}

\begin{corollary}
\label{cor:universal:cover}
    Every bi-regular local Lie groupoid $G$ has a source-simply connected cover $\varphi:\widetilde{G}\to G$ which is unique up to restriction.
\end{corollary}

\begin{proof}[Proof of the Corollary]
    The construction of the source simply connected cover of a local groupoid is entirely similar to the case of Lie groupoids (see, e.g., \cite{lectures-integrability-lie}).
    Let $\Pi_1(\F(s))$ be the fundamental groupoid of the source-foliation of $G$. It is a smooth Lie groupoid over $G$ with source map
    \[ p : \Pi_1(\F(s)) \to G,\quad [\gamma]\mapsto \gamma(0). \] 
    We set $\widetilde{G}:=p^{-1}(M)$ and define the maps:
    \begin{align*} 
	s':\widetilde{G}\to M,&\quad [\gamma]\mapsto s(\gamma(0)),\qquad  t':\widetilde{G}\to M,\quad [\gamma]\mapsto t(\gamma(1)),\\
	&u':M\to\widetilde{G},\quad x\mapsto [x]. 
    \end{align*}
    Then $\varphi:\widetilde{G}\to G$, $[\gamma]\to \gamma(1)$, is a cover of $G$ with 1-connected source fibers, so the result follows from the theorem.
\end{proof}

\begin{remark}
    One can give a proof of the corollary independent of the theorem, by following the same method as in the proof of 
    Theorem~\ref{thm:universal:covering}, with $G$ replacing $\G(A)$. In the sequel, we will only make use of the corollary. 
    However, since the theorem is interesting on its own, and the proof in \cite{olver} of the corresponding statement for local Lie groups  
    contains a gap (see Remark~\ref{rmk:olver:gap} below), we will give a complete proof of the theorem.
\end{remark}

The remainder of this section is concerned with the proof of Theorem
\ref{thm:covers}. Suppose $G$ and $M$ are manifolds, and we have maps
\begin{itemize}
    \item $u : M\to G$, an embedding,
    \item $s : G \to M$, a surjective submersion such that $s\circ u = \id$,
    \item $t : G \to M$, a surjective submersion such that $t\circ u = \id$.
\end{itemize}
We will show, roughly speaking, that a bi-regular local Lie groupoid structure
on $G$ (with the given $s$, $t$, $u$) is equivalent to specifying left- and right-invariant Maurer-Cartan forms. From this the theorem follows immediately.

\begin{definition}
    An \emph{$s$-framing} for the tuple $(G,M,s,t,u)$ is an isomorphism of vector bundles $\omega_s : T^sG \to t^*(T^s_MG)$
    over the unit section such that $\omega_s{\restriction_M} = \id$.
    A \emph{$t$-framing} is an isomorphism of vector bundles $\omega_t : T^tG \to s^*(T^t_MG)$
    over the unit section such that $\omega_t{\restriction_M} = \id$.
\end{definition}

If $G$ is a bi-regular local Lie groupoid, the left-invariant (respectively,
right-invariant) Maurer Cartan form is an $s$-framing (respectively,
$t$-framing). For a bi-regular groupoid we will always use these framings.

Once framings are specified one can talk about invariant vector fields.

\begin{definition}
    Suppose we are given an $s$-framing $\omega_s$ and a $t$-framing $\omega_t$ for the tuple $(G,M,s,t,u)$. Then:
    \begin{enumerate}[(i)]
    \item A vector field $X$ tangent to the $s$-fibers is called \emph{right-invariant} if its image under the $s$-framing is
    of the form $t^*\sigma$ for some $\sigma \in \Gamma(T^s_MG)$.
    \item A vector field $X$ tangent to the $t$-fibers is called \emph{left-invariant} if its image under the $t$-framing is
    of the form $s^*\sigma$ for some $\sigma \in \Gamma(T^t_MG)$.
   \end{enumerate}
\end{definition}

Notice that in the case of a bi-regular local Lie groupoid this definition of right- and left-invariant vector field coincides with the one given in Section~\ref{sec:algebroid}. In particular, for a local Lie groupoid we always have that:
    \begin{itemize}
        \item the Lie bracket of right-invariant vector fields is right-invariant,
        \item the Lie bracket of left-invariant vector fields is left-invariant,
        \item the Lie bracket of a right-invariant and a left-invariant vector field is zero.
    \end{itemize}

The following is the characterization of local Lie groupoids we are aiming for. It is a direct generalization of \cite[Theorem 18]{olver}.

\begin{proposition}
\label{prop:characterization}
    Let $\omega_s$ and $\omega_t$ be $s$- and $t$-framings for the tuple $(G,M,s,t,u)$. There is a bi-regular local Lie groupoid structure on 
    $(G,M,s,t,u)$ whose Maurer-Cartan forms coincide with these framings if and only if:
    \begin{enumerate}[(i)]
        \item the Lie bracket of left-invariant vector fields is left-invariant,
        \item the Lie bracket of right-invariant vector fields is right-invariant,
        \item the Lie bracket of a left- and a right-invariant vector field is zero.
    \end{enumerate}
    If we demand that the local Lie groupoid be strongly connected, the structure is unique up to restriction.
\end{proposition}

\begin{remark}
\label{rmk:olver:gap}
Theorem 18 in \cite{olver} claims that giving a local Lie group structure in $G$ is equivalent to
prescribing right-invariant vector fields,
without specifying the left-invariant vector fields.
Explicit examples show, however, that this is not sufficient, and the left-invariant
vector fields need to be specified as well. More precisely, if only right-invariant
vector fields are specified, it may happen that no compatible left-invariant
vector fields exist. This is related to the fact that the Maurer-Cartan form on a local Lie group(oid) is not only invariant but satisfies an extra condition: the Maurer-Cartan equation. 

The mistake in the proof in \cite{olver} is that the
multiplication does not get defined near $G\times\{e\}$.
For an explicit example, one can look at $G = \R_{>0} \times \S^1$, with coordinates $(x,\theta)$. The vector fields $X = x\frac{\partial}{\partial x}$ and $Y = x\frac{\partial}{\partial \theta}$ satisfy $[X,Y] = Y$.
If we try to prescribe $X$ and $Y$ as right-invariant vector fields on $G$,
there is no way to find corresponding left-invariant vector fields to make $G$ into a
local Lie group: if $Z$ is a left-invariant vector field that equals $\frac{\partial}{\partial x}$ at $(1,1)$, we must have $[Y,Z]=0$, but the flow of $Y$
maps the point $(1,1)$ to itself after time $2\pi$,
mapping $\frac{\partial}{\partial x}$ to $\frac{\partial}{\partial x} + 2\pi\frac{\partial}{\partial \theta}$.
This shows that $Z$ cannot be invariant under the flow of $Y$.
\end{remark}

\begin{proof}[Proof of Proposition~\ref{prop:characterization}]
    The ``only if'' part of the proof is clear.
    For the proof of the ``if'' part, suppose that we have framings that satisfy the above conditions.
    We will prescribe the multiplication and inversion maps.
    Let us start by defining the multiplication on a neighborhood of
    $G \timesst M$.
    Pick $g\in G$ and write $x = s(g)$. If $h\in t^{-1}(x)$ is sufficiently close
    to $x$,
    then
    \[ h = \phi^1_{\omega_t^{-1}(s^*\sigma)}(x) \]
    for some local section $\sigma$ of $T^t_MG$ near $x$
    (here $\phi^1$ denotes the time-1 flow of a vector field).
    Note that $\omega_t^{-1}(s^*\sigma)$ is left-invariant.
    We set
    \begin{equation} 
    \label{eq:product}
    g\cdot h = \phi^1_{\omega_t^{-1}(s^*\sigma)}(g)
    \end{equation}
    if this flow exists.
    This product $g\cdot h$ is well-defined for $h$ small,
    in the sense that it does not depend on the choice of $\sigma$, as we now explain.
    Because the bracket of right-invariant vector fields is right-invariant,
    we can define a Lie algebroid $A\to M$ from these right-invariant vector fields.
    This algebroid $A$ acts on $G \xrightarrow{t} M$ via the right-invariant vector fields.
    We can integrate this Lie algebroid action to an action of a local Lie groupoid $H$
    integrating $A$.
    Using this action $H \curvearrowright G$, we can identify a neighborhood of $M$
    in $H$ with a neighborhood of $M$ in $G$, by $h \mapsto h \cdot s(h)$.
    Under this identification, the multiplication \eqref{eq:product} corresponds to
    the one in $H$. This shows that $g\cdot h$ is indeed well-defined using \eqref{eq:product}.
    We have defined the multiplication near $G \timesst M$.
    Similarly, one defines $g\cdot h = \phi^1_{\omega_s^{-1}(t^*\tau)}(h)$
    if $g = \phi^1_{\omega_s^{-1}(t^*\tau)}(x)$
    is small, and $\tau$ is a local section of $T^s_MG$.

    One thing remains to check to ensure we have a well-defined multiplication:
    if $g = \phi^1_{\omega_s^{-1}(t^*\tau)}(x)$ and $h = \phi^1_{\omega_t^{-1}(s^*\sigma)}(x)$ for local
    section $\tau$ and $\sigma$ as above, we need to ensure that the two definitions
    of $g\cdot h$ agree.
    The first definition defines $g\cdot h$ as
    \[ \phi^1_{\omega_t^{-1}(s^*\sigma)}(g) = \phi^1_{\omega_t^{-1}(s^*\sigma)}(\phi^1_{\omega_s^{-1}(t^*\tau)}(x)) \]
    and the second as
    \[ \phi^1_{\omega_s^{-1}(t^*\tau)}(h) = \phi^1_{\omega_s^{-1}(t^*\tau)}(\phi^1_{\omega_t^{-1}(s^*\sigma)}(x)) .\]
    Because left- and right-invariant vector fields commute, these definitions coincide
    for $g, h$ sufficiently close to $x$.
    Therefore, by restricting if necessary, we have
    defined a multiplication map on $G \timesst G$ near
    $G \timesst M \cup M \timesst G$.

    We leave it to the reader to check that this multiplication is locally associative
    (possibly after restricting
    further),
    and to define an inversion map near $M$.
    The uniqueness up to restriction follows immediately from
    Lemma~\ref{lem:stronglyconnecteddeterminedbyrestriction}.
\end{proof}

We can now give a very short proof of the existence of a groupoid structure on covers.

\begin{proof}[Proof of Theorem~\ref{thm:covers}]
   Let $G$ be a bi-regular local Lie groupoid over $M$ and let $\varphi:G'\to G$ be a generalized cover. If 
   $\wmc^L$ and $\wmc^R$ denote the left- and right-invariant Maurer-Cartan forms on $G$, then we obtain 
   $s$- and $t$-framings on $G'$ by setting:
   \[ \omega_s:=\varphi^*\wmc^L \quad \text{ and }\quad \omega_t:=\varphi^*\wmc^R. \]
   The left- and right-invariant vector fields on $G'$ defined by these framings are $\varphi$-related to the left- and 
   right-invariant vector fields on $G$. It follows that they satisfy conditions (i)-(iii) of Proposition~\ref{prop:characterization},
   and we conclude that $G'$ has the required bi-regular local groupoid structure for which $\varphi:G'\to G$ is an 
   \'etale morphism of local Lie groupoids. 
\end{proof}

\begin{remark}
\label{rmk:Covers:inverses}
The local groupoid structure constructed in Proposition~\ref{prop:characterization} has an inverse map with domain, in general, strictly smaller than $G$.
Therefore, even if every element in the local Lie groupoid $G$ has an inverse, elements in a cover $G'$ of $G$ may fail to have an inverse. This is yet another reason we have to allow for the domain of the inverse map of a local Lie groupoid $G$ to be strictly smaller than $G$.
\end{remark}

\subsection{Classification result}

In \cite{olver}, Olver shows roughly speaking that every local Lie group
is covered by a cover of a globalizable local Lie group.
We prove the analogous result for groupoids, and strengthen it slightly.

\begin{theorem}
    Suppose $G\tto M$ is a bi-regular $s$-connected local Lie groupoid with product connected to axes and
    integrable Lie algebroid $A$.
    Let $\tilde{G}$ be its source-simply connected cover,
    and let $\G(A)$ be the source-simply connected integration of $A$.
    Then there is a commutative diagram:
    \begin{center}
    \begin{tikzpicture}
        \matrix(m)[matrix of math nodes, row sep=2.2em, column sep=2.6em]{
            & \tilde{G} & \\
            G & & U\subseteq \G(A) \\
            & \AsCo(G) & \\
        };
        \path[->] (m-1-2) edge node[auto,swap]{$p_1$} (m-2-1.north east);
        \path[->] (m-1-2) edge node[auto]{$p_2$} (m-2-3.north west);
        \path[->] (m-2-1.south east) edge (m-3-2);
        \path[->] (m-2-3.south west) edge (m-3-2);
    \end{tikzpicture}
    \end{center}
    Here, $p_1$ is the covering map and $G \to \AsCo(G)$ is the completion map. 
    The map $p_2$ is a generalized covering of local Lie groupoids,
    which sends a point $g\in\tilde{G}$ to the class of the $A$-path associated
    to a $\tilde{G}$-path from $s(g)$ to $g$.
    \label{thm:classification-groupoids}
\end{theorem}

\begin{proof}
    Our first step is to show the existence of the upper half of the diagram.
    This is the statement that is proved in \cite[Theorem 21]{olver}
    for the case of local groups, and the proof is analogous.
    We use Cartan's method of the graph.
    The argument is similar to the usual argument for integrability of Lie algebroid
    morphisms to Lie groupoid morphisms \cite[Proposition 6.8]{moerdijk-mrcun}.

    On the fibered product $\tilde{G} \tensor[_t]{\times}{_t} \G(A)$ we consider
    the foliation $\F$ given by
    \[ \F_{(g_1,g_2)} = \left\{ (\xi\cdot g_1, \xi\cdot g_2) \mid \xi \in A_{t(g_1)} \right\} .\]
    Integrability of this distribution follows from the fact that $\tilde{G}$ and $\G(A)$
    have the same Lie algebroid.
    
    Fix a point $x\in M$ and let $N_x$ be the leaf of $\F$ through $(x,x)$. We claim that the projection 
    $\pi_1 : N_x \to s^{-1}(x)$ is a covering of the source fiber in $\tilde{G}$.
    Note that $\pi_1 : N_x \to s^{-1}(x)$ is a local diffeomorphism by bi-regularity of the local groupoid.
    Now if $(g_1,g_2) \in N_x$, and $U_1$ is a neighborhood of $(g_1,g_2)$ in $N_x$ that gets mapped
    diffeomorphically to a neighborhood $U_2$ of $g_1$ in $s^{-1}(x)$,
    then $U_2$ is uniformly covered. Indeed, if $(g_1,g_3)$ is another point on $N_x$,
    then by invariance of $\F$ under right multiplication in the $\G(A)$-direction,
    the translated submanifold $U_1 \cdot g_2^{-1} \cdot g_3$ is a neighborhood of $(g_1,g_3)$ in $N_x$
    that gets mapped diffeomorphically to $U_1$.
    Hence, $\pi_1 : N_x \to s^{-1}(x)$ is a covering map, and since
    $s^{-1}(x)$ is simply connected, it is a diffeomorphism.
    The inverse is a diffeomorphism $\phi_x : s^{-1}(x) \to N_x$.

    If $N = \bigcup_{x\in M} N_x$, we have a map $\phi : \tilde{G} \to N$, $g\to \phi_{s(g)}(g)$. This map is smooth,
    because it is the extension by holonomy of the map $x\mapsto (x,x)$.
    We define $p_2 := \pi_2 \circ \phi: \tilde{G}\to\G(A)$ and let $U$ be the image of $p_2$, so we obtain the required diagram:
    \begin{equation*}
    \begin{tikzpicture}[baseline=(current  bounding  box.center)]
        \matrix(m)[matrix of math nodes, row sep=2.2em, column sep=2.6em]{
            & \tilde{G} & \\
            G & & U\subseteq \G(A) \\
        };
        \path[->] (m-1-2) edge node[auto,swap]{$p_1$} (m-2-1.north east);
        \path[->] (m-1-2) edge node[auto]{$p_2$} (m-2-3.north west);
    \end{tikzpicture}
    \end{equation*}
    
    Now, to establish the existence of the lower half of the diagram
    we apply the functor $\AsCo$ to $p_1$:
    \begin{center}
    \begin{tikzpicture}[baseline=(current  bounding  box.center)]
        \matrix(m)[matrix of math nodes, row sep=2.2em, column sep=2.6em]{
            & \tilde{G} & \\
            G & & \AsCo(\tilde{G}) \\
               & \AsCo(G) & \\
        };
        \path[->] (m-1-2) edge node[auto,swap]{$p_1$} (m-2-1.north east);
        \path[->] (m-1-2) edge (m-2-3.north west);
        \path[->] (m-2-3.south west) edge node[auto]{$\AsCo(p_1)$} (m-3-2);
        \path[->] (m-2-1.south east) edge (m-3-2);
    \end{tikzpicture}
    \end{center}
    We claim that $\AsCo(\tilde{G}) \cong \G(A)$, and that under this identification the natural
    map $\tilde{G} \to \AsCo(\tilde{G})$ corresponds to $p_2$.

    Note that the associators of $\tilde{G}$ are discrete
    since they lie in the fiber of $\tilde{G} \to U$ over the unit section.
    Since $G$ has product connected to the arcs, so does $\tilde{G}$ and hence $\AsCo(\tilde{G})$ is a Lie groupoid.
    We have morphisms of Lie groupoids 
   \[ \AsCo(\tilde{G}) \to \AsCo(U) \to \AsCo(\G(A)) \cong \G(A), \] 
   which induce isomorphisms at the level of Lie algebroids. Since $\G(A)$ is source-simply connected, we must in fact have
    \[ \AsCo(\tilde{G}) \cong \AsCo(U) \cong \AsCo(\G(A)) \cong \G(A) .\]
    Following along these maps, we see that an element $g\in\tilde{G}$ is mapped to $\G(A)$
    as follows:
    \[ g\in\tilde{G} \mapsto [g]\in\AsCo(\tilde{G}) \mapsto [p_2(g)] \in \AsCo(U) \mapsto [p_2(g)] \in \AsCo(\G(A)) \mapsto p_2(g) \in \G(A) .\]
    Hence, the diagram in the statement of the theorem follows.
\end{proof}

\begin{corollary}
\label{cor:1-connected:integrable}
    If $G$ is a $s$-simply connected, bi-regular, local Lie groupoid with product connected to the axes and integrable algebroid, then $\Assoc(G)$ is uniformly discrete (and, in particular, $\AsCo(G)$ is smooth).
\end{corollary}
\begin{proof}
    The above result shows that $G$ is a generalized cover of a globalizable local
    Lie group. The associators of $G$ are therefore contained in the fiber of this
    generalized covering over the unit section, and this fiber is discrete.
\end{proof}

Note that this corollary is false without the assumption of $s$-simply connectedness.
For example, there are local Lie groups with non-discrete associators.
An example is given is the next section.

\section{Associativity and integrability}
\label{sec:integrability}

In this section, we discuss the relationship between the integrability of a Lie
algebroid and the associativity of a local groupoid integrating it. Starting
with a local Lie groupoid $G$ with Lie algebroid $A$, we would like to give a
precise relationship between the monodromy groups $\tilde{\N}_x(A)$,
controlling the integrability of $A$, and the associators $\Assoc_x(G)$.  Since
$\tilde{\N}_x(A)$ and $\Assoc_x(G)$ only depend on the restrictions of $G$ and
$A$ to the orbit through $x$, in this section we will make the following:
\begin{itemize}
\item {\bf Assumption.} $G\tto M$ is a \emph{transitive} local Lie groupoid,
    i.e., for every $x,y\in M$ there exist a well-formed word
    $w=(g_1,\dots,g_k)$ with $t(g_1)=x$ and $s(g_k)=y$.
\end{itemize}
This is equivalent to assuming that $A$ has surjective anchor (i.e., is
transitive). Of course, our results are also valid for a non-transitive local
Lie groupoid $G$: in the statements one must replace $M$ by an orbit of $G$.

Before we turn to the general theory, in Section~\ref{ss:associators:examples}
we discuss examples exhibiting properties that motivate our results. In
particular, one of these examples shows that it is possible to have
$\tilde{\N}_x(A)$ trivial (hence, $A$ integrable) and $\Assoc_x(G)\subset G_x$
non-discrete. Hence, in order to relate $\tilde{\N}_x(A)$ and $\Assoc_x(G)$ one
needs some extra assumptions on $G$. We will see that these can always be
achieved by shrinking $G$.

\subsection{Examples}
\label{ss:associators:examples}

We have seen in Lemma~\ref{lem:assoc:restriction} that restricting a local Lie groupoid does not change its associators. On the
other hand, shrinking a local Lie groupoid \emph{can} change the associators drastically. To illustrate this we will give an example of a local Lie group $G$ such that $\Assoc(G)$ is not discrete. Because every local Lie group has a neighborhood of the unit section that is globalizable
(and therefore has trivial associators), this shows that the associators can change from
non-discrete to discrete by shrinking the local groupoid!

\begin{example}
Let $B \subset \R^2$ be the ladder-shaped set
\[ B = (\{0\}\times \R) \cup (\{1\}\times \R) \cup ([0,1]\times \Z) \]
and let $G$ be its thickening
\[ G = \left\{ p \in \R^2 \mid \exists\, q \in B \text{ such that } d(p,q) < \frac1{10} \right\} .\]
(Here, $d$ is the Euclidean distance.)
We will equip $G$ with the structure of a local Lie group in such a way that $\Assoc(G)$ is not discrete.

For each $n \in \Z_{>0}$, let $f_n : \R\to \R_{>0}$ be a smooth function that
is 1 outside of $[\frac13, \frac23]$ and such that the time-$\frac{8}{10}$ flow of the vector field
$f_n\frac{\partial}{\partial x}$ on $\R$ maps $\frac1{10}$ to $\frac9{10} + \frac1{100n}$.
In other words, the flow of $f_n\frac{\partial}{\partial x}$ is just a little bit faster than
that of $\frac{\partial}{\partial x}$.
For each $n \in \Z_{\leq 0}$, let $f_n$ be the function on $[\frac1{10},\frac9{10}]$ that takes value 1 everywhere.
Now consider the following vector fields on $G$:
\begin{align*}
X(x,y) &= \begin{cases} f_n(x)\frac{\partial}{\partial x} &
        \text{if } x\in[\frac1{10},\frac9{10}] \text{ and } y\in(n-\frac1{10},n+\frac1{10}) \\
        \frac{\partial}{\partial x} & \text{otherwise}, \end{cases} \\
Y(x,y) &= \frac{\partial}{\partial y} .
\end{align*}

\begin{figure}[h]
    \begin{tikzpicture}[scale=2]
        \clip (-2,-0.6) rectangle (2,1.5);

        \foreach \y in {-2,-1,0,1} {
            \draw (0.1,\y+0.1) -- (0.9,\y+0.1) -- (0.9,\y+0.9) -- (0.1,\y+0.9) -- cycle;
        }
        \draw (-0.1,-2) -- (-0.1,2);
        \draw (1.1,-2) -- (1.1,2);

        \draw (0,0) -- (0,1) -- (1.05,1) -- (1.05,0) -- (0.05,0);
        \draw[thick, postaction={decorate}, decoration={markings,mark=at position 0.5 with {\arrow{Latex}}}] (0,0) -- (0,1);
        \draw[thick, postaction={decorate}, decoration={markings,mark=at position 0.5 with {\arrow{Latex}}}] (0,1) -- (1.05,1);
        \draw[thick, postaction={decorate}, decoration={markings,mark=at position 0.5 with {\arrow{Latex}}}] (1.05,1) -- (1.05,0);
        \draw[thick, postaction={decorate}, decoration={markings,mark=at position 0.5 with {\arrow{Latex}}}] (1.05,0) -- (0.05,0);

        \draw[fill=black] (0,0) circle (0.015);
        \draw[fill=white] (0.05,0) circle (0.015);

        \draw[dashed,line width=0.5pt] (0.025,0) circle (0.2);

        \coordinate (A) at (2.5cm/2/2,-0.2cm/2/2);
        \draw[line width=0.5pt] (0.106873, 0.182474) -- (-0.929381, 0.647422);
        \draw[line width=0.5pt] (0.0755409, -0.193509) -- (-1.02338, -0.480527);

        \pgflowlevel{\pgftransformshift{\pgfpoint{-2.5cm}{0.2cm}}};
        \pgflowlevel{\pgftransformscale{2}};
        \pgftransformscale{1.5};
        \draw[dashed] (0.025,0) circle (0.2);
        \clip (0.025,0) circle (0.2);

        \foreach \y in {-2,-1,0,1} {
            \draw (0.1,\y+0.1) -- (0.9,\y+0.1) -- (0.9,\y+0.9) -- (0.1,\y+0.9) -- cycle;
        }
        \draw (-0.1,-2) -- (-0.1,2);
        \draw (1.1,-2) -- (1.1,2);

        \draw[thick] (0,0) -- (0,1) -- (1.05,1) -- (1.05,0) -- (0.05,0);

        \draw[fill=black] (0,0) circle (0.015);
        \draw[fill=white] (0.05,0) circle (0.015);

        \draw[line width=0.2pt,{Latex[length=0.3mm, width=0.4mm]}-{Latex[length=0.3mm, width=0.4mm]}] (0,-0.03) -- (0.05,-0.03);
    \end{tikzpicture}
    \caption{The vector field $X$ is obtained by slightly modifying
        $\frac{\partial}{\partial x}$.
        The origin is the black dot.
        If we flow the origin along $Y$
        ($=\frac{\partial}{\partial y}$) for time 1, then along $X$ for time 1,
        then along $-Y$ for time 1, and finally along $-X$ for time 1, we end
        up slightly away from the origin, at the white dot.
        This white dot is an associator.
        The gap is exaggerated in the picture for clarity.}
\end{figure}
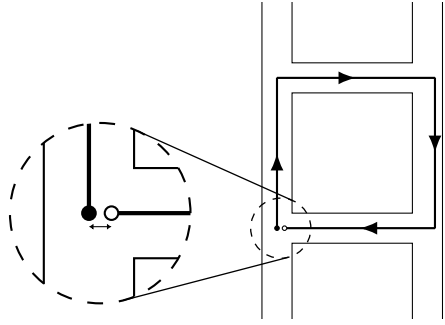
There is a unique local Lie group structure on $G$ over $M$ such that $X$ and $Y$
are bi-invariant (more precisely, unique up to restriction if we ask that the local groupoid has
with products connected to the axes).
We claim that for this local Lie group structure, $\Assoc(G)$ is not discrete.

Let $a = (\frac1{20},0)\in G$ and $b = (0,\frac1{20}) \in G$.
Then consider the product
\[ \underbrace{a^{-1}\cdots a^{-1}}_{20} \underbrace{b^{-1}\cdots b^{-1}}_{20n}
   \underbrace{b\cdots b}_{20n} \underbrace{a\cdots a}_{20} ,\]
where the numbers under the braces indicate the number of repetitions.
Clearly, this product results in the neutral element $(0,0)$ when evaluated from the inside out.
However, evaluating first the second half $c = b\cdots b a\cdots a$ from left to right,
and then evaluating the rest $a^{-1}\cdots a^{-1} b^{-1}\cdots b^{-1} c$ from right to left,
we get $\left(\frac1{100n},0\right)$.
This shows that $\left(\frac1{100n},0\right) \in \Assoc(G)$ for all positive $n$, so that the associators do not
form a discrete set.
\end{example}
\medskip

All the examples of local Lie groupoids we have discussed so far have integrable Lie algebroid. We consider now an example of a local Lie groupoid with non-integrable algebroid. It is a local Lie groupoid integrating a pre-quantum Lie algebroid $A_\omega$ where $\omega$ has a non-discrete group of spherical periods.

\begin{example}
Let $M = \S^2\times \S^2$, where each copy of $\S^2$
is equipped with the usual round metric.
Let
\[ G' = \{ ((y,y'),(x,x')) \in M\times M \mid x+y\neq 0\neq x'+y' \} \times \R .\]
We prescribe the source and target maps as
\[ s((y,y'),(x,x'),a) = (x,x') \in M \]
and
\[ t((y,y'),(x,x'),a) = (y,y') \in M .\]
In a similar fashion to what we did in the examples in Section~\ref{ss:example-non-globalizable}, we introduce a multiplication on $H$ by the formula
\begin{multline*}
((z,z'),(y,y'),a_1) \cdot ((y,y'),(x,x'),a_2)= \\ =((z,z'),(x,x'),a_1+a_2+A(\Delta xyz)+\lambda A(\Delta x'y'z')) ,
\end{multline*}
where $\lambda \in \R$ is a fixed parameter,
and which is defined whenever
\[ A(\Delta xyz) \in (-\pi,\pi) \quad\text{and}\quad A(\Delta x'y'z') \in \left(-\frac{\pi}{\lvert \lambda\rvert},\frac{\pi}{\lvert\lambda\rvert}\right),\]
with the convention that the condition on $A(\Delta x'y'z')$ is void if $\lambda = 0$. The algebroid of this local Lie groupoid is $A_\omega$ where $\omega=\d S\oplus \lambda \d S$, and so it is integrable precisely if $\lambda$ is rational.

To check that $H$ is 3-associative assume that $((z,z'),(y,y'),a_1)$, $((y,y'),(x,x'),a_2)$ and $((x,x'),(w,w'),a_3)$ are in $H$.
Then
\begin{multline*}
    ( ((z,z'), (y,y'),a_1) \cdot ((y,y'),(x,x'),a_2) )\cdot ((x,x'),(w,w'),a_3)=\\
    = ((z,z'),(x,x'),a_1+a_2+a_3+A(\Delta xyz)+A(\Delta wxz)+\lambda A(\Delta x'y'z') + \lambda A(\Delta w'x'z'))
\end{multline*}
and
\begin{multline*}
    ((z,z'), (y,y'),a_1) \cdot ( ((y,y'),(x,x'),a_2) \cdot ((x,x'),(w,w'),a_3) )=\\
    = ((z,z'),(x,x'),a_1+a_2+a_3+A(\Delta wxy)+A(\Delta wyz)+\lambda A(\Delta w'x'y') + \lambda A(\Delta w'y'z')) .
\end{multline*}
Now
\[ A(\Delta xyz) + A(\Delta wxz) = A(\Delta wxy) + A(\Delta wyz) \pmod{4\pi} \]
because both equal the area of the quadrangle $wxyz$ (which is defined up to $\pi$).
By the restriction on the areas of these triangles, equality must hold in $\R$ (not just in $\R/4\pi$).
Similarly, we will have
\[ \lambda A(\Delta x'y'z') + \lambda A(\Delta w'x'z') = \lambda A(\Delta w'x'y') + \lambda A(\Delta w'y'z') .\]
This proves 3-associativity.

This local Lie groupoid is not globally associative. The counterexample to 6-associativity
in Section~\ref{sss:non-glob} also works here,
by considering $G''$ as the local Lie subgroupoid $\{ (y,y'),(N,N) \} \times \R$ of $G'$,
where $N \in \S^2$ is the north pole.
Recall, however, that there was a neighborhood of $M \subset G''$ that was globally
associative. This is not the case for $G'$. Indeed, the following result will be a consequence of 
Theorem~\ref{thm:assoc-mono} below:

\begin{proposition}
    There is an open neighborhood of $M \subset G'$ that is globally associative
    if and only if $\lambda \in \Q$ (i.e.\ iff the Lie algebroid is integrable).
    \label{lem:nonint}
\end{proposition}

\end{example}

\subsection{Monodromy groups}
\label{ss:monodromy}

The integrability of a Lie algebroid $A\to M$ is controlled by its
\emph{monodromy groups} \cite{integ-of-lie-article,lectures-integrability-lie}.
The monodromy group $\tilde{\N}_x(A)$ at
$x\in M$ is a subgroup of $\G(\g_x)$ contained in the center $Z(\G(\g_x))$.
Its construction uses the language of $A$-paths, which we will assume the reader
is familiar with.
A detailed exposition can be found in \cite{integ-of-lie-article} and
\cite[sections 2.2, 3.3 and 3.4]{lectures-integrability-lie} and we refer to
these works for details.

Let $A$ be a Lie algebroid over $M$. We write $I$ for the interval $[0,1]$.
Recall that an $A$-path is a $C^2$-path $a:I\to A$ with base path
$\gamma:I\to M$, satisfying:
\[ \rho(a(t))=\frac{\d }{\d t} \gamma(t). \]
Alternatively, this condition states that $a\,\d t:TI\to A$ is a Lie algebroid morphism. 

Two $A$-paths $a_0,a_1:I\to A$ are said to be \emph{$A$-homotopic} if there exists a
Lie algebroid morphism $a\,\d t + b\,\d s : TI\times TI \to A$ satisfying the
boundary conditions:
\begin{itemize}
\item $a(0,t)=a_0(t)$ and $a(1,t)=a_1(t)$;
\item $b(s,0)=b(s,1)=0$.
\end{itemize}
The space of $A$-paths is a Banach manifold $P(A)$ and $A$-homotopy defines an
equivalence relation $\sim$ on $P(A)$. The quotient space:
\[ \G(A):=P(A)/\sim, \]
is a topological groupoid over $M$ with composition given by concatenation of
paths (after possible reparameterization). The main result of
\cite{integ-of-lie-article} states that:
\begin{enumerate}[(i)]
\item $A$ is an integrable Lie algebroid if and only if $\G(A)$ is a smooth
    quotient, in which case it is the source 1-connected integration of $A$;
\item $\G(A)$ is smooth if and only if the monodromy groups $\Mon_x(A)$ are
    uniformly discrete.
\end{enumerate}

Let us briefly recall the construction of $\Mon_x(A)$. 

\begin{definition}
    Suppose $a,a':I\to A$ are $A$-paths over the same base path.
    We say that $a$ and $a'$ are \emph{$A$-homotopic along a trivial sphere}
    if there is an $A$-homotopy between $a$ and $a'$ whose base homotopy
    determines a trivial element in $\pi_2(M)$.
\end{definition}

Then one can show that:

\begin{lemma}
Suppose $a_0:I\to A$ is an $A$-path with base path $\gamma_0 :I\to M$ and $S : I\times I \to M$ is such that
\begin{itemize}
    \item $S(0,t) = \gamma_0(t)$,
    \item $S(s,0) = \gamma_0(0)$ and $S(s,1) = \gamma_0(1)$ for all $s$,
\end{itemize}
Then there is an $A$-homotopy $a\,\d t + b\,\d s : TI\times TI \to A$ covering $S$.
Moreover, the class of the $A$-path $a_1(t):=a(1,t)$ modulo $A$-homotopies along
trivial spheres only depends on the homotopy class (rel boundary) of $S$.
\end{lemma}

The proof of this lemma is contained in the proof of Proposition 3.21 in \cite{lectures-integrability-lie}. We will write $\tilde{\partial}(a_0, S)$ for the class of $a_1$ modulo $A$-homotopy along trivial spheres.

If in the lemma above the path $t\mapsto S(1,t)$ is constant, say at $x\in M$,
then $\tilde{\partial}(a_0, S)$ is (the class of) a $\g_x$-path,
hence it determines an element of $\G(\g_x)$.
If we write $0_x$ for the trivial $A$-path at $x$, then:

\begin{definition}
The \emph{monodromy group} $\Mon_x(A)$ is defined as the image of
\[ \partial : \pi_2(M,x) \to \G(\g_x) : [\alpha] \mapsto \tilde{\partial}(0_x, \alpha) ,\]
where an element $[\alpha] \in \pi_2(\Orbit_x,x)$ is represented by a smooth map $\alpha:I\times I\to M$ mapping the boundary to $x$.
\end{definition}

The monodromy map has the following interpretation
(\cite{integ-of-lie-article}). The short exact sequence of algebroids (recall
we are assuming that $A$ is transitive):
\begin{equation}
\label{eq:short:exact:algebroid}
 \xymatrix{0\ar[r] & \Ker\rho \ar[r]& A \ar[r]^\rho & TM \ar[r] & 0} 
 \end{equation}
can be seen as a ``fibration'' to which one can apply the integration functor $\G(-)$. The result is a long exact sequence of ``isotropy groups'' whose connecting homomorphism is the monodromy map:
\begin{equation}
\label{eq:short:exact:groupoid}
\xymatrix{\cdots \ar[r] & \pi_2(M,x) \ar[r]^\partial & \G(\g_x)\ar[r] & \G_x(A) \ar[r] & \pi_1(M,x)\ar[r] & \cdots}.
\end{equation}
In particular, this leads to the short exact sequence of groups:
\begin{equation}
\label{eq:short:exact:mon}
\xymatrix{0\ar[r] & \Mon_x(A) \ar[r]& \G(\g_x)\ar[r] & \G_x(A) \ar[r] & 0}.
\end{equation}

\subsection{Simplicial complexes and associative completions}
\label{ss:expansion-contraction}

The associative completion of a local Lie groupoid $G$ admits a simplicial
description that turns out to be quite useful to relate associators and
monodromy.

Let $W_k$ be the ordered simplicial complex
(see Figure~\ref{fig:wk})
\[ \{ \{0\}, \dots, \{k\} \} \cup \{ \{0,1\}, \{1,2\},\dots,\{k-1,k\} \} .\]
We will use the same letter to denote the associated simplicial set.
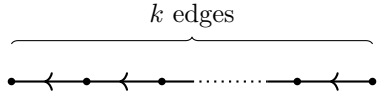
\begin{figure}[h]
\begin{center}
    \begin{tikzpicture}
        \draw[thick,dotted] (2.4,0) -- (3.4,0);
        \fill (0,0) circle (0.05);
        \fill (1,0) circle (0.05);
        \fill (2,0) circle (0.05);
        \fill (3.8,0) circle (0.05);
        \fill (4.8,0) circle (0.05);
        \begin{scope}[thick,decoration={
            markings,
            mark=at position 0.55 with {\arrow{<}}}
            ] 
            \draw[postaction={decorate}] (0,0) -- (1,0);
            \draw[postaction={decorate}] (1,0) -- (2,0);
            \draw[] (2,0) -- (2.4,0);
            \draw[] (3.4,0) -- (3.8,0);
            \draw[postaction={decorate}] (3.8,0) -- (4.8,0);
        \end{scope}
        \draw[decoration={brace,raise=0.5cm},decorate] (0,0) -- (4.8,0)
        node[pos=0.5,yshift=0.6cm,anchor=south] {$k$ edges};
    \end{tikzpicture}
\end{center}
\caption{The simplicial complex $W_k$ is just a sequence of $k$ edges, head to tail.}
\label{fig:wk}
\end{figure}
If $\NG=\{G^{(n)}\}$ is the nerve of a local Lie groupoid $G\tto M$
(cf.~Section~\ref{ss:nerve}), a simplicial map $W_k \to \NG$ is nothing but a
well-formed word on $G$ of length $k$. Moreover, we can also express
equivalence of words using simplicial notions, as we now explain.

For the following discussion we consider only ordered simplicial complexes $S$
which are 2-dimensional. Moreover, we assume that each edge has at most two
faces attached to it.
We will say that an edge is a \emph{boundary edge} of $S$ if it has at most one
face attached to it. 

If $\{u,w\}$ is a boundary edge of $S$ with $u < w$,
we can obtain a new ordered simplicial complex $S'$ by
\begin{enumerate}
    \item adding a new vertex $v$ to the simplicial complex
        with $u<v<w$,
    \item adding edges $\{u,v\}$ and $\{v,w\}$, and a face $\{u,v,w\}$.
\end{enumerate}
We will say that $S'$ can be obtained from $S$ by \emph{expansion}.

If $\{u,v\}$ and $\{v,w\}$ are boundary edges of $S$, with $u < v < w$,
and $\{u,w\}$ is not an edge of $S$,
we can obtain a new ordered simplicial complex $S'$ by
\begin{enumerate}
    \item adding the edge $\{u,w\}$,
    \item adding the face $\{u,v,w\}$.
\end{enumerate}
We will say that $S'$ can be obtained from $S$ by \emph{contraction}.

\begin{definition}
    We say that an ordered simplicial complex $S$ is a \emph{good complex}
    if there is an integer $k\geq 1$ such that $S$ can be obtained
    from $W_k$ by repeated expansion or contraction.
\end{definition}

Note that a good complex comes with a choice of two vertices:
there is one vertex that is minimal among boundary vertices (we will call it the \emph{source} of the good complex)
and one vertex that is maximal among boundary vertices (we will call it the \emph{target} of the good complex).

\begin{definition}
    If $S$ is a good complex, a \emph{boundary path} of $S$ is an ordered subcomplex $S'$ of $S$
    that is contained in its boundary,
    and for which there is an isomorphism of ordered complexes $W_k \to S'$
    that maps the source of $W_k$ to the source of $S$, and the target of $W_k$ to the target of $S'$.
\end{definition}

In other words, a boundary path of $S$ is just a path from source to target along the boundary of $S$
such that the vertices are increasing along the path.

\begin{figure}[h]
\begin{center}
    \begin{tikzpicture}
        \path (3,1) coordinate (a)
              (2,0) coordinate (b)
              (2,2) coordinate (c)
              (1.5,1.2) coordinate (d)
              (2.2,1) coordinate (h)
              (1,0.3) coordinate (e)
              (0.7,2) coordinate (f)
              (0,1) coordinate (g);
        \draw[line width=4.5, yellow!50, line cap=round] (a) -- (b) -- (e) -- (g);
        \draw[line width=4.5, green!50, line cap=round] (a) -- (c) -- (f) -- (g);
        \fill (a) circle (0.05);
        \fill (b) circle (0.05);
        \fill (c) circle (0.05);
        \fill (d) circle (0.05);
        \fill (h) circle (0.05);
        \fill (e) circle (0.05);
        \fill (f) circle (0.05);
        \fill (g) circle (0.05);
        \begin{scope}[thick,line cap=round,decoration={
            markings,
            mark=at position 0.55 with {\arrow{>}}}
            ] 
            \draw[postaction={decorate}] (a)--(b); 
            \draw[postaction={decorate}] (b)--(e);
            \draw[postaction={decorate}] (e)--(g);
            \draw[postaction={decorate}] (a)--(h); 
            \draw[postaction={decorate}] (h)--(b);
            \draw[postaction={decorate}] (h)--(e); 
            \draw[postaction={decorate}] (e)--(d); 
            \draw[postaction={decorate}] (d)--(g);
            \draw[postaction={decorate}] (h)--(d); 
            \draw[postaction={decorate}] (a)--(c); 
            \draw[postaction={decorate}] (c)--(h);
            \draw[postaction={decorate}] (c)--(d); 
            \draw[postaction={decorate}] (c)--(g); 
            \draw[postaction={decorate}] (c)--(f); 
            \draw[postaction={decorate}] (f)--(g);
        \end{scope}
    \end{tikzpicture}
\end{center}
\caption{A good complex. This one can be obtained from the yellow-marked $W_3$ on the bottom
    by (for example) an expansion, a contraction, an expansion,
    a contraction, an expansion, two contractions and an expansion.
    The source is the rightmost vertex, the target is the leftmost vertex.
    This complex has two boundary paths (one on the top, one on the bottom).
    They are highlighted in green and yellow.}
\end{figure}
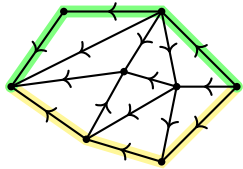

\begin{remark}
    A good complex can have many boundary paths. For example,
    the complex shown in Figure~\ref{fig:good-complex-not-disk} has eight of them.
    The good complexes that will be most relevant for us are those that are
    homeomorphic to a disk. These have precisely two boundary paths.
\end{remark}

\begin{figure}[h]
\begin{center}
    \begin{tikzpicture}
        \path (0,0) coordinate (a)
              (1,0) coordinate (b)
              (2,0) coordinate (c)
              (3,0) coordinate (d)
              (4,0) coordinate (e)
              (0.5,0.7) coordinate (f)
              (2.2,0.9) coordinate (g)
              (2.8,1) coordinate (h)
              (3.5,0.7) coordinate (i);
        \draw[line width=4.5, green!50, line cap=round] (e) -- (i) -- (d) -- (c) -- (b) -- (f) -- (a);
        \fill (a) circle (0.05);
        \fill (b) circle (0.05);
        \fill (c) circle (0.05);
        \fill (d) circle (0.05);
        \fill (e) circle (0.05);
        \fill (f) circle (0.05);
        \fill (g) circle (0.05);
        \fill (h) circle (0.05);
        \fill (i) circle (0.05);
        \begin{scope}[thick,line cap=round,decoration={
            markings,
            mark=at position 0.55 with {\arrow{>}}}
            ] 
            \draw[postaction={decorate}] (e)--(d);
            \draw[postaction={decorate}] (d)--(c);
            \draw[postaction={decorate}] (c)--(b);
            \draw[postaction={decorate}] (b)--(a);
            \draw[postaction={decorate}] (e)--(i);
            \draw[postaction={decorate}] (i)--(d);
            \draw[postaction={decorate}] (d)--(h);
            \draw[postaction={decorate}] (d)--(g);
            \draw[postaction={decorate}] (h)--(g);
            \draw[postaction={decorate}] (g)--(c);
            \draw[postaction={decorate}] (b)--(f);
            \draw[postaction={decorate}] (f)--(a);
        \end{scope}
    \end{tikzpicture}
\end{center}
\caption{A good complex that is not homeomorphic to a disk.
This one has eight boundary paths. One of them is highlighted.}
\label{fig:good-complex-not-disk}
\end{figure}
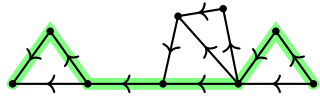

\begin{definition}
    Let $S$ be a good complex (considered as a simplicial set)
    and let $\phi : S \to \NG$ be a simplicial map.
    The \emph{boundary words} of $\phi$ are the words corresponding
    to the restriction of $\phi$ to the boundary paths of $S$.
\end{definition}

The following result now holds more or less by construction.

\begin{proposition}
    Two well-formed words $w_1,w_2\in W(G)$ are equivalent
    if and only if there is a good complex $S$ and a
    simplicial map $\phi : S \to \NG$
    such that both words are boundary words of $\phi$.
    \label{prop:equivalence-geometrically}
\end{proposition}

\begin{proof}
    If the two words $w$ and $w'$ are equivalent, then there is a sequence of expansions and contractions
    of words turning $w$ into $w'$.
    Let $w = w_0, w_1, \dots, w_n = w'$ be the shortest sequence of expansions and contractions
    from $w$ to $w'$.
    We can use this sequence to build a good complex $S$ and a map $\phi : S \to \NG$,
    as we explain now.

    Start with a map $\phi_0 : W_k \to \NG$ representing $w_0$.
    If $w_1$ is an expansion of $w_0$, we can expand $W_k$ in the corresponding edge
    to get a map $\phi_1 : S_1 \to \NG$ with boundary words $w_0$ and $w_1$.
    If $w_1$ is a contraction of $w_0$, we can contract $W_k$ in the corresponding edge
    to get a map $\phi_1 : S_1 \to \NG$ with boundary words $w_0$ and $w_1$.
    (Note that the edge that we need to add to the complex for the contraction is not already present.
    This edge can only be present if we have an expansion followed by a contraction
    undoing that expansion, but we chose the shortest sequence of words linking $w$ to $w'$.)

    Conversely, if we have a good complex $S$ and a map $\phi : S \to \NG$
    with boundary words $w_1$
    and $w_2$, we obtain a sequence of expansions and contractions of words
    turning $w_1$ into $w_2$.
\end{proof}

\subsection{Associators are monodromy elements}
\label{ss:assoc-is-mono}

We would like to give
a precise connection between the associators of a local
Lie groupoid $G$ and the monodromy groups of the Lie algebroid $A=A(G)$. For
that, we need to be able to compare the isotropy $G_x$ (where the
associators live) with the 1-connected Lie group $\G(\g_x)$ integrating the
isotropy Lie algebra $\g_x=\Ker\rho_x$ (which contains the monodromy groups).
For this reason, henceforth we will assume that $G$ has been shrunk so that:
\begin{itemize}
\item[(H1)] $G_x$ is a 1-connected local Lie group.
\end{itemize}

Under this assumption, we have a local Lie group homomorphism $G_x \to \G(\g_x)$ defined as follows. If $g\in G_x$, choose a path in $G_x$ from $x$ to $g$. Differentiate this path to get a $\g_x$-path and hence an element of $\G(\g_x)$. Since $G_x$ is simply-connected, this map is well-defined: any two paths
in $G_x$ from $x$ to $g$ are homotopic, and such a homotopy induces a $\g_x$-homotopy.  

By shrinking $G$ further, we can also assume that:
\begin{itemize}
\item[(H2)] the map $G_x\to \G(\g_x)$ is injective.
\end{itemize}
Actually, by shrinking $G$ even further, we can also assume that:
\begin{itemize}
\item[(H3)] the source fibers of $G$ are 1-connected.
\end{itemize}
Again, this allows us to construct a local homomorphism of topological groupoids $G\to \G(A)$:  if $g\in G$ has source $x$, choose a path in $s^{-1}(x)$ from $x$ to $g$. Differentiate this path to get an $A$-path and hence an element of $\G(A)$. Since $s^{-1}(x)$ is simply-connected, this map is well-defined, because any two paths in $s^{-1}(x)$ from $x$ to $g$ are homotopic, and this homotopy induces an $A$-homotopy. 

\begin{proposition}
    \label{prop:assoc-is-mon}
Let $G$ be a local Lie groupoid satisfying (H1)-(H3). Then under the natural map $G_x\to \G(\g_x)$ we have $\Assoc_x(G) \subset \Mon_x(A)$.

\end{proposition}

\begin{proof}
Applying the functor $\AsCo(-)$ to the morphism $G\to \G(A)$ we obtain a commutative triangle:
\[
\xymatrix{G \ar[rr]\ar[dr] &&\G(A) \\
& \AsCo(G)\ar[ur]}
\]
Passing to isotropies, it follows that the map $G_x\to \G_x(A)$ takes the associators $\Assoc_x(G)$ to the unit. But the last map factors as:
\[
\xymatrix{G_x \ar[rr]\ar[dr] &&\G_x(A) \\
& \G(\g_x)\ar[ur]}
\]
and the kernel of the morphism $\G(\g_x)\to\G_x(A)$ is the monodromy group $\Mon_x(A)$, so the result follows.
\end{proof}

\subsection{Monodromy elements are associators}
\label{ss:mono-is-assoc}

We would like to improve Proposition~\ref{prop:assoc-is-mon} and show that we actually have $\Assoc_x(G)= \Mon_x(A)\cap G_x$. For this, we will need to construct a lift of the map $G\to \G(A)$ to the space of $A$-paths:
\[
\xymatrix{
 & & P(A) \ar[d] \\
G\ar@{-->}[urr] \ar[rr] && \G(A) 
}
\]
which is multiplicative in the following sense: if $w = (g_1,\dots,g_k)\in
W(G)$ is a well-formed word on $G$, write $P(w) = P(g_k)\circ\cdots \circ
P(g_1)$ (concatenation of $A$-paths). Then if $w_1$ and $w_2$ are equivalent
words we would like $P(w_1)$ and $P(w_2)$ to be $A$-homotopic paths.  This will
require slightly stronger hypotheses than (H1)-(H3) which can still be
achieved by shrinking $G$ further.

The construction of a multiplicative lift requires the choice of some auxiliary
data, namely:
\begin{itemize}
\item a reparameterization function, i.e., a smooth, non-decreasing function $f:[0,1]\to[0,1]$ that is 0 near 0 and 1 near 1;
\item a Riemannian metric on $M$;
\item an $A$-connection  $\nabla$ on $A$.
\end{itemize}
The connection yields an exponential map $\exp_\nabla : A \to G$,
defined and injective in some neighborhood of the zero section $M$
(\cite[section 4.4]{lectures-integrability-lie}).
Hence, if $g\in G$ is sufficiently close to $M$, then it is in the image of
$\exp_{\nabla}$, say $g=\exp(\xi)$.
Writing $x=s(g)$, we then
have the $G$-path $\tau \mapsto \exp(f(\tau)\xi)$ and its 
associated $A$-path $\tilde{P}(g):I\to A$ (obtained by differentiating and applying the
Maurer-Cartan form). Moreover, if $g\in G$ is sufficiently close to $M$, the base path $\tilde{\gamma}(t)$ of $\tilde{P}(g)$
is contained completely in a uniformly normal subset of $M$ (\footnote{Recall that an open subset of a Riemannian manifold
is \emph{uniformly normal} if there exists some $\delta > 0$ such that the subset is contained in a geodesic ball of radius
$\delta$ around each of its points. See \cite[lemma 5.12]{lee-riemannian}.
Important for us will be the fact that between any two points in a uniformly normal
neighborhood, there is a unique minimizing geodesic, and it lies in this neighborhood.}).

Now we modify the $A$-path $\tilde{P}(g)$, which lies over $\tilde{\gamma}(t)$,
into an $A$-path $P(g)$ that lies
over the geodesic $\gamma_{s(g),t(g)}$ from $s(g)$ to $t(g)$.
We do this by prescribing
\[ P(g) = \partial(\tilde{P}(g), [\alpha]) ,\]
where $[\alpha] : [0,1] \times [0,1] \to M$ is the map such that
$s\mapsto \alpha(t,s)$ is the geodesic from $\tilde{\gamma}(t)$
to $\gamma_{s(g),t(g)}(t)$.
This way, we have associated to each $g\in G$ sufficiently close to $M$
an $A$-path $P(g)$ that lies over the geodesic from $s(g)$ to $t(g)$.
If $w = (g_1,\dots,g_k)\in W(G)$ is a well-formed word on $G$, 
we write $P(w) = P(g_k)\circ\cdots\circ P(g_1)$ (the concatenation of $P(g_k)$, \dots, $P(g_k)$).

Write $\Delta[k]$ for the standard $k$-simplex:
\[ \Delta[k] = \left\{ (\lambda_0,\dots,\lambda_k) \in \R^{k+1} \mid \lambda_i \geq 0, \sum_i \lambda_i = 1 \right\} .\]
The boundary of $\Delta[2]$ consists of three line segments. If $(g,h) \in \U$,
we will write $\beta_{(g,h)}:\partial\Delta[2] \to M$ for the map that sends
the first edge of $\Delta[2]$ to the base path of $P(g)$,
the second edge to the base path of $P(gh)$, and the third edge to the base path of $P(h)$.

Let $S''$ be a neighborhood of the diagonal in $M\times M$ such that for all $(x,y)\in S''$
there is a unique shortest geodesic from $x$ to $y$.
Let
\[ S' = \{ (x,y,z) \in M\times M\times M \mid (x,y), (y,z), (x,z) \in S'' \} .\]
If $(x,y,z) \in S'$, we will write $\beta_{(x,y,z)} : \partial \Delta[2] \to M$ for the
map that sends the edges of $\Delta[2]$ to the three geodesics between $x$, $y$ and $z$.

Consider now the space
\[ S = \left\{ (x,y,z,[\alpha]) \in S'\times \left( \frac{\text{maps $\Delta[2] \to M$}}{\text{homotopy rel $\partial$}} \right) :
    \alpha{\restriction_{\partial \Delta[2]}} = \beta_{(x,y,z)}
    \right\} .\]
For the natural topology on $S$, the map $\pi : S \to M^3$, $(x,y,z,[\alpha]) \mapsto (x,y,z)$ is continuous. For each $(x_0,y_0,z_0,[\alpha_0])$ there is an open neighborhood $U\subset M^3$ of $(x_0,y_0,z_0)$ and unique continuous section $s:U\to S$ of $\pi$ with $s(x_0,y_0,z_0)=(x_0,y_0,z_0,,[\alpha_0])$. Hence, $S$ has a unique smooth structure
for which $\pi$ is a local diffeomorphism.
Note that
\[ T = \{ (x,y,z) \in M\times M\times M \mid x=y \text{ or } y=z \text{ or } x=z \} \]
naturally sits inside $S$ by
\[ (x,y,z)\mapsto (x,y,z,[\text{maps whose image is contained in the image of $\beta_{(x,y,z)}$}]) .\]
We can choose a map
$\varphi_2 : U \to S$
from a neighborhood $U$ of $T$ in $M^3$
such that
\begin{itemize}
    \item $\pi \circ \varphi_2 = \id_{U}$,
    \item $\varphi_2{\restriction_T} = \id_T$.
\end{itemize}

\begin{definition}
    Suppose that $G$ is a local Lie groupoid over $M$, and that we have chosen
    auxiliary information $(f,\langle\cdot,\cdot\rangle,\nabla,\varphi_2)$ as above.
    We say that $G$ is \emph{shrunk} if
    \begin{enumerate}
        \item the isotropy groups $G_x$ are simply connected,
        \item for each $x\in M$, the map $G_x \to \G(\g_x)$
            is injective,
        \item the path $P(g)$ is defined for every $g\in G$,
        \item if $(g,h) \in \U$ then $\{s(h),t(h)=s(g),t(g)\} \in U$,
            so that we may write $\varphi_2(g,h) := \varphi_2(s(h),t(h),t(g))$,
        \item for every $(g,h) \in \U$, the class of $P(gh)$
            (as $A$-path modulo $A$-homotopy along trivial spheres)
            equals
            \[ \tilde{\partial}(P(g,h), \varphi_2(g,h)) .\]
            \label{condition-shrunk-essence}
    \end{enumerate}
\end{definition}

Every transitive local Lie groupoid has an open neighborhood of the identities
that is shrunk (after restricting).
The first three conditions can be satisfied by shrinking $G$.
The fourth can be satisfied by restricting the multiplication.
The last condition can be satisfied by shrinking and restricting:
it is automatically satisfied for a global Lie groupoid,
and over every contractible neighborhood $U$, the algebroid
is integrable (and over $U$ and near $M$, the local groupoid is
isomorphic to a global integration of $A{\restriction_U}$,
so that the condition holds for these small groupoid elements). 
Note that if a local Lie groupoid is shrunk it satisfies (H1)-(H3).

\begin{proposition}
    Suppose that $G$ is a shrunk local Lie groupoid over $M$.
    If $w = (w_1,\dots,w_k)$ and $w' = (w'_1,\dots,w'_{k'})$ are equivalent
    words, then the $A$-paths $P(w_1,\dots,w_k)$ and $P(w'_1,\dots,w'_{k'})$ are $A$-homotopic.
\end{proposition}
\begin{proof}
    It suffices to prove this if $w$ and $w'$ are elementarily equivalent,
    so we will assume that $w'$ is obtained from $w$ by multiplying the letters
    $w_i$ and $w_{i+1}$.
    The $A$-path
    \[ \tilde{\partial}(P(w_i,w_{i+1}),\varphi_2(w_i,w_{i+1})) \]
    is homotopic to $P(w_i,w_{i+1})$ by construction.
    Because $G$ is shrunk, it is also $A$-homotopic to
    the $A$-path $P(w_iw_{i+1})$.
    We conclude that $P(w_i,w_{i+1})$ and
    $P(w_iw_{i+1})$ are $A$-homotopic, which suffices to prove the result.
\end{proof}

\begin{remark}
The monodromy group $\Mon_x$ consists precisely of the $\g_x$-homotopy classes of 
$\g_x$-paths that are $A$-homotopic to the trivial path
$0_x$. Hence, the previous proposition gives another proof of the inclusion
$\Assoc_x(G) \subset \Mon_x$ for a local Lie groupoid that is shrunk.
\end{remark}

Suppose that $G$ is shrunk. We denote by $[n]$ ($n\in \mathbb{N}_0$) the objects of the simplicial category $\Delta$, so $[n]$ is a nonempty linearly ordered set of the form
\[ [n]=\{0,1,\dots,n\}. \]
To each simplicial map $[1] \to \NG$,
we can associate a curve in $M$ as before:
we take the base curve $P(g)$ where $g$ is the image of the 1-cell.
If we have a simplicial map $[2] \to \NG$,
we get a homotopy class of maps $\Delta[2] \to M$ rel boundary,
namely $\varphi_2(g,h)$ as before, where $(g,h) \in \U$ is the image of the
2-cell.
More generally, if we have a good complex $S$ and a simplicial map $\phi : S \to \NG$,
we get a homotopy class of maps $[\phi]:\realization{S} \to M$ rel boundary,
where $\realization{S}$ is the geometric realization of $S$,
by mapping each 1-cell and 2-cell as above.

\begin{proposition}
    Suppose that $G$ is a shrunk local Lie groupoid over $M$ and fix $x\in M$.
    Consider $G_x$ as a subset of $\G(\g_x)$ using the natural map
    $G_x \to \G(\g_x)$.
    Suppose that $S$ is a good complex homeomorphic to a disk,
    and that
    \[ \phi : S \to \NG \]
    is a simplicial map sending the boundary to $x$, so $\phi$ induces a class
    $[\phi] \in \pi_2(M,x)$.
    If one of the boundary words of $\phi$ is $(x)$,
    and the other boundary word is $(g_1,\dots,g_k)$ with $g_1,\dots,g_k\in G_x$
    and $g_1\cdots g_k\in G_x$,
    then $\partial([\phi]) \in \Assoc_x(G)$.
    \label{prop:mon-is-assoc}
\end{proposition}

For the proof we need the following two lemmas:
\begin{lemma}
    Let $G$ be a shrunk local Lie groupoid.
    Suppose $S$ is a good complex homeomorphic to a disk,
    and $\phi : S\to \NG$ is a simplicial map with boundary words
    $w$ and $w'$. Then
    \[ \tilde{\partial}(P(w),[\phi]) \text{ and } P(w') \]
    are $A$-homotopic along a trivial sphere.
    \label{lem:monodromy-word}
\end{lemma}
\begin{proof}
    If $S$ has only one face,
    this follows immediately from condition (\ref{condition-shrunk-essence})
    in the definition of a shrunk local groupoid.
    By induction on the number of faces of $S$, the statement follows.
\end{proof}

\begin{lemma}
    Let $G$ be a source-simply connected Lie groupoid,
    and let $U \subset G$ be an open neighborhood of the identities,
    considered as a local Lie groupoid.
    Suppose that $g_1, \ldots, g_n \in U$ are such that their product
    $g = g_1\cdots g_n$ lies in $U$.
    Then in $W(U)$ we have
    \label{lem:smalllemma}
    \[ (g_1,\ldots,g_n) \sim (g) .\]%
\end{lemma}
\begin{proof}
    Apply the functor $\AsCo$ to the inclusion $U \into G$.
    We get $\AsCo(U) \to \AsCo(G) \cong G$, which is a morphism of Lie groupoids.
    It is an isomorphism
    at the level of Lie algebroids,
    and $G$ is source-simply connected, so this map is an isomorphism.
    In particular, it is an injection.
    Because the words $(g_1, \dots, g_n), (g) \in W(U)$ are mapped to the same element of $G$,
    they must represent the same element of $\AsCo(U)$. This means that they are equivalent.
\end{proof}

Note that the last lemma is false if $G$ is not $s$-simply connected.
For example, if we take $U = (-0.5,0.5)/\Z \subset \R/\Z = G$,
then $0.25+0.25+0.25+0.25=0$ in $G$, but $(0.25,0.25,0.25,0.25) \not\sim (0)$ as words
in $W(U)$.

\begin{proof}[Proof of Proposition~\ref{prop:mon-is-assoc}]
    Let $w = (w_1,\dots,w_k)$ be the other boundary word of $\phi$.
    By Lemma~\ref{lem:monodromy-word}, $\partial([\phi])$ is the $\g_x$-homotopy class
    of the $\g_x$-path $P(w)$.
    Now working in $\G(\g_x)$, we see that
    \[ \partial([\phi]) = \prod_{i=1}^k w_i .\]
    By Lemma~\ref{lem:smalllemma} applied to $G_x \subset \G(\g_x)$, we have
    \[ (\partial([\phi])) \sim (w_1,\dots,w_k) \quad\text{in $W(G_x)$.}\]
    In particular, this equivalence holds in $W(G)$.
    However, the map $\phi$ shows that $(w_1,\dots,w_k) \sim (x)$
    in $W(G)$
    (by Proposition~\ref{prop:equivalence-geometrically}).
    This implies that $(\partial([\phi])) \sim (x)$ in $W(G)$,
    and so $\partial([\phi]) \in \Assoc_x(G)$.
\end{proof}

We now state and prove the main theorem relating associators to monodromy.

\begin{theorem}
    Suppose that $G$ is a shrunk local Lie groupoid over $M$ with Lie algebroid $A$.
    For $x\in M$,
    consider $G_x$ as a subset of $\G(\g_x)$ using the natural map
    $G_x \to \G(\g_x)$.
    \label{thm:assoc-mono} 
    Then
    \[ \Assoc_x(G) = \Mon_x(A) \cap G_x .\]
\end{theorem}
\begin{proof}
    The inclusion $\subset$ was Proposition~\ref{prop:assoc-is-mon}.
    We now prove the inclusion $\supset$.
    By Proposition~\ref{prop:mon-is-assoc}, it suffices to show that for every
    $[\alpha] \in \pi_2(M,x)$,
    there is a good complex $S$, homeomorphic to a disk,
    and a simplicial map $\phi : S \to \NG$, such that
    \begin{enumerate}
        \item the boundary of $S$ is mapped to $x$,
        \item one of the boundary words of $\phi$ is $(x)$,
        \item $[\phi] = [\alpha]$ as elements of $\pi_2(M,x)$.
    \end{enumerate}

    Write $\Delta = \{ (x,y) \in \R^2 \mid x\geq 0, y\geq 0, x+y\leq 1\}$.
    Represent $[\alpha]$ as a map $\alpha : \Delta \to M$, mapping the boundary of $\Delta$ to $x$.
    Fix a splitting $\sigma : TM \to A$ of the anchor.
    If $e$ is a sufficiently short oriented line segment in $\Delta$,
    from a point $v_1$ to a point $v_2$,
    we will associate to it an element $g_e$ of $G$, as follows.
    There is a tangent vector $V_e$ in $T_{\alpha(v_1)}M$ such that
    $t(\exp(\sigma(V_e))) = \alpha(v_2)$ (if $e$ is sufficiently short).
    Take $g_e = \exp(\sigma(V_e))$.
    The base path of $P(g_e)$ is then a geodesic from $\alpha(v_1)$ to $\alpha(v_2)$.

    If $k$ is a positive integer, we can subdivide $\Delta$ into $k^2$ triangles,
    as shown in Figure~\ref{fig:triangulation}.
    By choosing $k$ large enough, we can ensure that
    \begin{enumerate}[(a)]
        \item every edge $e$ of the triangulation is short enough to ensure that
            $g_e$ is defined,
        \item if $e_3,e_2,e_1$ are the edges of a face, the products
            $(g_{e_3}g_{e_2})g_{e_1}$ and $g_{e_3}(g_{e_2}g_{e_1})$ are both defined,
        \item if $e_3,e_2,e_1$ are the edges of a face, and $f_m,\dots,f_1$ are edges
            of the triangulation forming a path
            of total Euclidean length at most 10 (as measured in $\R^2$)
            with $t(e_3) = s(e_1) = t(f_m)$,
            then the product
            \begin{equation}
            \label{eq:laceproduct} (g^{-1}_{f_1} \cdots (g^{-1}_{f_m} (g_{e_3} g_{e_2} g_{e_1}) g_{f_m}) \cdots g_{f_1})
            \end{equation}
            is defined (the brackets that are not specified may be put in in any way).
    \end{enumerate}

    \begin{figure}[h]
        \centering
        \begin{tikzpicture}[scale=0.6]
            \foreach \x in {0,...,7}
            {\pgfmathtruncatemacro{\maxy}{7-\x}
              \foreach \y in {0,...,\maxy}
                \draw (\x,\y) -- (\x+1,\y) -- (\x,\y+1) -- (\x,\y);}
        \end{tikzpicture}
        \caption{The triangulation of $\Delta$, shown here for $k=8$.}
        \label{fig:triangulation}
    \end{figure}
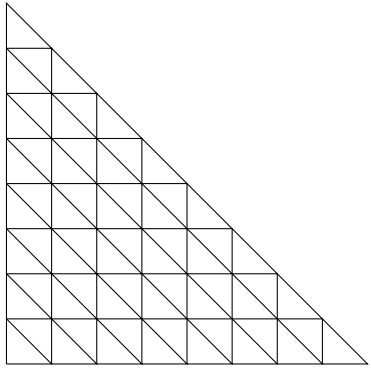

    If $k$ is large enough to guarantee these properties, we get an induced
    element of $\pi_2(M,x)$, as follows:
    we can map the 1-skeleton of the triangulated $\Delta$ to $M$ by mapping
    each vertex of the triangulation to its image under
    $\alpha$, and mapping each edge $e$ to the base path of $P(g_e)$.
    For each face of the triangulation, there is a preferred homotopy class of map
    into $M$ (rel boundary), namely the one determined by $\varphi_2$.
    All in all, we get a well-defined homotopy class of map from $\Delta$ to $M$,
    rel 1-skeleton. In particular, this defines a homotopy class of map from $\Delta$ to $M$,
    rel boundary, and hence an element of $\pi_2(M,x)$.
    After choosing $k$ larger if necessary, we may assume that this class is precisely $[\alpha]$.

    Let us write $\eps_{3k}, \dots, \eps_{1}$ for the $3k$ edges at the circumference of $\Delta$,
    in counterclockwise order starting at the origin $(0,0)$, and with the corresponding orientation
    (so $\eps_k,\dots,\eps_1$ are oriented to the right, $\eps_{2k},\dots,\eps_{k+1}$ to the upper left,
    and $\eps_{3k},\dots,\eps_{2k+1}$ down).
    We will work with ordered sequences of oriented edges of the triangulation.
    Such a sequence of oriented edges $(e_n, \dots, e_1)$ is \emph{well-formed}
    if $e_{i}$ ends where $e_{i+1}$ starts (for all $i$).
    A \emph{block} will be a well-formed sequence (possibly empty) of the form
    \[ (e_1^{-1}, \dots, e_n^{-1}, e_n, \dots, e_1) ,\]
    where $e_i^{-1}$ is the edge $e_i$ with its opposite orientation.
    This name is chosen to be the same as in the proof of Theorem~\ref{thm:smoothness-of-asco}.
    If $F$ is one of the faces of the triangulation, a \emph{lace around $F$} will be a
    sequence of the form
    \[ (e_1^{-1}, \dots, e_n^{-1}, c, b, a, e_n, \dots, e_1) ,\]
    where $a, b, c$ are the three sides of $F$, oriented counterclockwise,
    and $e_1$ starts at $(0,0)$. (So it loops around $F$ once, and encloses nothing else.)
    The words $(e_n,\dots,e_1)$ and $(e_1^{-1},\dots,e_n^{-1})$ are the \emph{ends} of the lace.

    Now let $(e_n, \dots, e_1)$ be a sequence of oriented edges of the triangulation
    that simultaneously satisfies the conditions
    \begin{enumerate}[(a)]
        \item the sequence can be obtained from $(\eps_{3k},\dots,\eps_2,\eps_{1})$ by repeated insertion
            of blocks,
        \item the sequence is a concatenation of laces $\ell_N, \dots, \ell_1$,
            one around each face of the triangulation,
            and each end of a lace has at most $2k$ edges (which certainly ensures that
            the total euclidean length of the edges of that end is less than 10).
    \end{enumerate}
    Such a sequence always exists.
    Figure~\ref{fig:sequence-in-triangulation} shows an example for $k=2$.

    \begin{figure}[h]
    \begin{center}
        \begin{tikzpicture}[scale=3]
            \pgfmathsetlengthmacro{\tick}{0.3mm}
            \draw[rounded corners=1] (0,-2 * \tick) -- (1cm,-2*\tick) -- (2cm+2*\tick,-2*\tick) --
            (1cm+\tick,1cm-\tick) -- (1cm+\tick,-\tick) -- (0,-\tick) -- (0,0) -- (1cm,0) -- (1cm,1cm) --
            (0,2cm) -- (0,1cm) -- (1cm-\tick, 1cm) -- (1cm-\tick,\tick) -- (0,\tick) -- (0,2*\tick) --
            (1cm-2*\tick, 2*\tick) -- (1cm-2*\tick, 1cm-\tick) -- (\tick, 1cm-\tick) --
            (1cm-3*\tick,3*\tick) -- (0,3*\tick) -- (0,4*\tick) -- (1cm-5.4142*\tick, 4*\tick) --
            (0,1cm-1.4142*\tick) -- (0,5*\tick);
            \draw[fill=black] (0,-2*\tick) circle[radius=0.5*\tick]; 
            \draw[fill=white] (0,5*\tick) circle[radius=0.5*\tick]; 
            \node at (1.33,0.27) {$F_1$};
            \node at (0.33,1.3) {$F_2$};
            \node at (0.67,0.67) {$F_3$};
            \node at (0.28,0.33) {$F_4$};
        \end{tikzpicture}
    \end{center}
    \caption{A sequence of edges that satisfies the conditions in the proof of
        Theorem~\ref{thm:assoc-mono}, for $k=2$,
        from the black dot to the white dot.
        The sequence is both of the form $(\text{lace around $F_1$}, \text{lace around $F_2$},
        \text{lace around $F_3$},\text{lace around $F_4$})$, and of the form
        $(\eps_1, \eps_2, \eps_3, \text{block}, \eps_4, \eps_5, \text{block}, \text{block}, \eps_6)$.}
    \label{fig:sequence-in-triangulation}
    \end{figure}
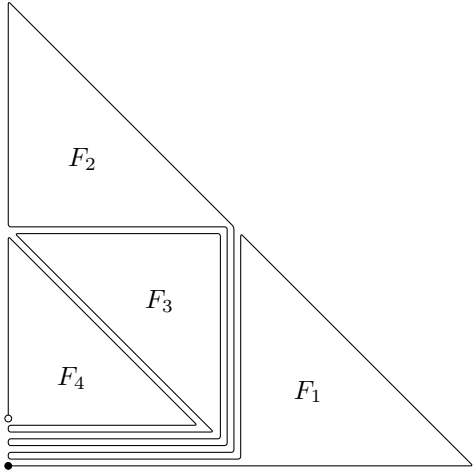

    The video at
    \url{http://math.uiuc.edu/~michiel2/edge-sequence/}
    illustrates an appropriate sequence of edges for $k=5$,
    which can immediately be generalized to
    work for any value of $k$.
    Note that for each of the laces, the product of the elements associated to the edges
    is defined if we use the order as in (\ref{eq:laceproduct}).
    We will call this the \emph{product} of the lace.

    Now this sequence allows us to build up the required good complex $S$ and
    the simplicial map $\phi : S \to \NG$.
    Start with the complex $W_1$, mapped to the element $x$.
    Expand $3k-1$ times, to get a good complex with two boundary paths: one is the original $W_1$,
    the other is a copy of $W_{3k}$. On the complex we have so far,
    $\phi$ will still map everything to $x$.

    Recall that $(e_n,\dots,e_1)$ can be obtained from $(\eps_{3k},\dots,\eps_1)$ by insertion of blocks.
    This means that the word $(g_{e_n},\dots,g_{e_1})$ can be obtained from $(x,\dots,x)$ by expansion.
    Expanding our good complex further, we therefore get a complex and a map $\phi$
    with boundary words $(x)$ and $(g_{e_n},\dots,g_{e_1})$.

    Now we group the $(e_n,\dots,e_1)$ by laces.
    Because the corresponding products are defined, we may now continue building our good complex using
    contractions.
    This results in a complex and map $\phi$ with boundary words
    $(x)$ and $(p_N,\dots,p_1)$,
    where $p_i$ is the product of the $i$'th lace.
    By construction, this simplicial map $\phi$ induces $[\alpha]$, proving the result.
\end{proof}

\section{Simplicial monodromy groups}
\label{s:simplicial:monodromy}

The results of the previous section have an even more conceptual simplicial interpretation as we now explain. We will continue to assume that $G\tto M$ denotes a transitive local Lie groupoid.

\subsection{The simplicial viewpoint}
\label{ss:simplicial:introd}
The nerve $\NG=\{G^{(n)}\}$ of a groupoid is a simplicial set, so we have its geometric realization $|\NG|$ (we are ignoring the manifold structure of $\NG$, treating it simply as a set). Since $M$  sits naturally inside $|\NG|$, for each $x\in M$ we have the $n$-homotopy group of $|\NG|$ based at $x$ which we denote by $\pi_n(\NG,x)$. As for any topological space, we also have the fundamental groupoid $\Pi_1(|\NG|)\tto |\NG|$ and we consider its restriction to $M$ which we denote by $\Pi_1(\NG)\tto M$. Note that the isotropy group at $x$ of $\Pi_1(|\NG|)$ is precisely $\pi_1(\NG,x)$.

Because, in general, $\NG$ is not Kan the usual combinatorial description of the groups $\pi_n(\NG,x)$ does not apply. One could consider the Kan completion and then proceed with the usual description. For our purposes, we only need to deal with $\pi_1$ and $\pi_2$ for which one has a relatively simple combinatorial description due to Moore and Smith (see \cite{Smith,EstLee}), which we will recall below. Using these descriptions we give simplicial interpretations of our previous  constructions. For example, the associative completion has the following interpretation:

\begin{theorem}
\label{thm:AC:homotopy}
If $G$ is a strongly connected local Lie groupoid then there is an isomorphism of groupoids:
\[ \AsCo(G)\simeq \Pi_1(\NG). \]
\end{theorem}

Now, given a local Lie groupoid $G\tto M$, we can  consider the short exact sequence associated with the map $(t,s):G\to M\times M$: if $U\subset (M\times M)$ is the image of $(t,s)$ (a local subgroupoid of the pair groupoid) then this exact sequence is:
\begin{equation}
\label{eq:short:exact:local:grpd}
\xymatrix{
1\ar[r] & I(G)\ar[r] & G\ar[r] & U\ar[r] & 1
}
\end{equation}
where $I(G)=\bigcup_{x\in M} G_x$ denotes the bundle of isotropies. \emph{Assuming} that this induces a fibration of the associated simplicial sets, we obtain a long exact sequence:
\[
\xymatrix@C20pt{\cdots \ar[r] & \pi_2(\mathbf{U},x) \ar[r]& \pi_1(\mathbf{I(G)},x)\ar[r] & \pi_1(\mathbf{G},x) \ar[r] & \pi_1(\mathbf{U},x)\ar[r] & \cdots}
\]
By Theorem~\ref{thm:AC:homotopy}, this can be rewritten as:
\begin{equation}
\label{eq:long:exact:local:grpd}
\xymatrix{\cdots \ar[r] & \pi_2(\mathbf{U},x) \ar[r]& \AsCo(G_x) \ar[r] & \AsCo_x(G) \ar[r] & \AsCo_x(U)\ar[r] & \cdots}
\end{equation}
Comparing these exact sequences with the sequences \eqref{eq:short:exact:algebroid}, \eqref{eq:short:exact:groupoid} and \eqref{eq:short:exact:mon},
it is tempting to call the connecting homomorphism 
\[ \partial^s:\pi_2(\mathbf{U},x) \to \AsCo(G_x)\] 
the \emph{simplicial monodromy map}. The main issue with this approach is whether \eqref{eq:short:exact:local:grpd} actually defines a simplicial fibration. 

Below, for a general transitive local groupoid, we will construct a monodromy map $\partial^s:\pi_2(\mathbf{U},x) \to K_1/\delta(K_2)$, for which we have a long exact sequence as above. We will see also that that there is a natural group homomorphism $\AsCo(G_x)\to K_1/\delta(K_2)$. For a strongly connected groupoid we show that this map is surjective and we conjecture that it is also injective. We are able to prove this for a shrunk groupoid, in which case we also show:

\begin{theorem}
\label{thm:main:monodromy}
Let $G$ be a local Lie groupoid which is shrunk. Then there is a map $\partial^s:\pi_2(\mathbf{U},x) \to \AsCo(G_x)$ giving a long exact sequence \eqref{eq:long:exact:local:grpd}, and fitting into a commutative diagram:
\[
\xymatrix{
\pi_2(M,x)\ar[r]^\partial \ar[d] & \G(\g_x)\\
\pi_2(\mathbf{U},x) \ar[r]_{\partial^s} & \AsCo(G_x)\ar[u]}
\]
where the vertical arrows are isomorphisms. 
\end{theorem}

The second vertical arrow in the square above is just the isomorphism induced from the canonical map $G_x\to \G(\g_x)$ that exists for a shrunk local Lie groupoid. Then this result explains more conceptually the equality $\Assoc_x(G)=\Mon_x(A)\cap G_x$ stated in Theorem~\ref{thm:assoc-mono}. In fact, the proof of Theorem \ref{thm:main:monodromy} will use arguments from the proof of Theorem~\ref{thm:assoc-mono}. 

The following is an interesting consequence of Theorem \ref{thm:main:monodromy}. We will give the proof later.

\begin{corollary}
\label{cor:shrunk}
If $G$ is a local Lie groupoid which is shrunk then there is an isomorphism of topological groupoids $\AsCo(G)\simeq \G(A)$.
\end{corollary}

\subsection{Simplicial homotopy groups}
\label{ss:simplicial:homotopy}

We recall the combinatorial description of the simplicial homotopy groups in degree 1 and 2 for a (possibly non-Kan) simplicial set, due to Moore and Smith (see \cite{Smith,EstLee}).

First, for a set $S$ we will denote by $F(S)$ the free group on $S$. More generally, for a quiver $X^1\tto X^0$, we have the associated free groupoid $F(X^1)\tto X^0$ (sometimes also called the path groupoid of the quiver): it consists of reduced well-formed words of elements of $X^1$ and their formal inverses, with multiplication obtained by concatenation of reduced words (note that a set $S$ can be thought of as quiver with $X^1=S$ and $X^0=\{*\}$). A typical element of $F(X^1)$ is written as:
\[ w=w_1^{\eps_1}\bullet \dots \bullet w_k^{\eps_k}, \quad (w_i\in X^1, \eps_i=\pm 1), \]
where one defines: $t(w_i^{-1}):=s(w_i)$ and $s(w_i^{-1}):=t(w_i)$. For each $x\in X^0$ one denotes by $1_x$ the ``empty word'' at $x$, which by definition has source and target $x$ and is such that $w_i\bullet w_i^{-1}=1_{t(w_i)}$ and $w_i^{-1}\bullet w_i=1_{s(w_i)}$. We denote by $\cdot$
the groupoid multiplication on $F(X^1)$. Note also that if $X^1\tto X^0$ is a topological quiver then $F(X^1)\tto X^0$ is a topological groupoid in a natural way.

Now, given any simplicial set $\mathbf{X}=\{X^n\}$, fix $x\in X$. One defines groups $\Gamma_i(x)$ as follows:
\begin{itemize}
\item $\Gamma_1(x):=F(X^1)_{x}$ (the isotropy group of $F(X^1)\tto X^0$ at $x$);
\item $\Gamma_2(x):=F(Y_{x})/\sim$, where $Y_{x}$ is the set:
\[ Y_{x}:=\left\{ (w,\sigma)\in F(X^1)\times X^2: t(w)=x, s(w)=\d_0 \d_1\sigma\right\}, \]
and $\sim$ is the equivalence relation generated by:
\[ (w_1,\sigma_1)\bullet(w_1\cdot w_2,\sigma_2)\sim (w_1\cdot \d\sigma_1\cdot w_2,\sigma_2)\bullet (w_1,\sigma_1).  \]
Here, given $\sigma\in X^2$, we denote by $\d \sigma\in F(X^1)_{\d_0\d_1\sigma}$ its boundary:
\[ \d \sigma:=\d_0\sigma\bullet \d_2\sigma\bullet  (\d_1\sigma)^{-1}. \]
\item  $\Gamma_3(x)=F(Z_{x})/\sim$, where $Z_{x}$ is the set:
\[  Z_{x}:=\left\{ (\xi,(w,\tau))\in \Gamma_2(x)\times(F(X^1)\times X^3): t(w)=x, s(w)=\d_0 \d_1\d_2\tau\right\}, \]
and $\sim$ is the equivalence relation generated by:
\[ (\xi\cdot(w_1,\sigma),(w_1\cdot w_2,\tau))\sim (\xi,(w_1\cdot \d\sigma_1\cdot w_2,\tau)).  \]
\end{itemize}

We also define a group homomorphism $\delta: \Gamma_2(x)\to \Gamma_1(x)$ by setting on generators:
\[ (w,\sigma)\mapsto w \cdot \d\sigma \cdot w^{-1}, \]
and another  group homomorphism $\delta: \Gamma_3(x)\to \Gamma_2(x)$, by setting on generators:
\[ (\xi,(w,\tau))\mapsto \xi\cdot (w,\d_3 \tau)\cdot (w,\d_1\tau)^{-1}\cdot (w,\d_2\tau)\cdot (w\cdot (\d_2\d_3\tau)^{-1},\d_0\tau)
\cdot \xi^{-1}. \]

These homomorphisms define a complex of groups:
\[ \xymatrix{\Gamma_3(x)\ar[r]^{\delta} & \Gamma_2(x)\ar[r]^{\delta} & \Gamma_1(x)} \]
such that $\im\delta\subset \Gamma_i(x)$ are normal subgroups, $\Ker\delta\subset \Gamma_2(x)$ is a central subgroup and one has:
\[ \pi_1(|\mathbf{X}|,x)=\Gamma_1(x)/\im \delta,\quad \pi_2(|\mathbf{X}|,x)=\Ker\delta/\im \delta. \]
Moreover, one can define a groupoid over $X^0$ by setting:
\[ \Pi_1(\mathbf{X})=F(X^1)/\approx, \]
where $\approx$ is the equivalence relation 
\[ w_1\approx w_2 \quad\text{ if and only if }\quad 
\left\{
\begin{array}{l}
s(w_1)=s(w_2),\\ 
t(w_1)=t(w_2),\\ 
w_2=w_1\cdot \delta(\Gamma_2(s(w_1)))
\end{array}
\right.
\]
and one finds that this coincides with the restriction of the fundamental groupoid of $|\mathbf{X}|$ to $X^0$.

Theorem~\ref{thm:AC:homotopy} now follows from the following:

\begin{proposition}
Let $G$ be a strongly connected local Lie groupoid and let $w_1,w_2\in W(G)$. Then $w_1\sim w_2$ if and only if $\bar{w}_1\approx \bar{w}_2$, where for a word $w$, we denote by $\bar{w}$ the corresponding element in $F(G^{(1)})$. 
\end{proposition}

\begin{proof}
The proof is similar to the group case, which is given in \cite{EstLee}. We
redo here the direct implication and leave the reverse implication for the
reader.

In order to prove the direct implication, we can assume that $w_1$ and $w_2$ differ by an elementary contraction:
\[ w_1=(g_1,\dots,g_i,\dots,g_k),\quad w_2=(g_1,\dots,u,v,\dots,g_k), \]
where $g_i=uv$. Now define an element in $\eta\in\Gamma_2(s(w_1))$ by setting:
\[ \eta:=((g_i\bullet\cdots\bullet g_k)^{-1},(u,v)). \]
Then:
\[ \delta(\eta)= (g_i\bullet\cdots\bullet g_k)^{-1}\bullet u \bullet v\bullet (uv)^{-1} \bullet g_i\bullet\cdots\bullet g_k \]
Hence, we find:
\[ \bar{w}_2= \bar{w}_1\cdot\delta(\eta)\quad \Leftrightarrow \quad \bar{w}_1\approx \bar{w}_2.\]
\end{proof}

\subsection{Simplicial monodromy groups}
\label{ss:simplicial:monodromy}

Consider a (transitive) local groupoid $G\tto M$. The map
\[ \phi:=(t,s):G\to M\times M, \]
has as image an open $U\subset M\times M$ containing the diagonal. We can view $U$ as a local groupoid over $M$ (a subgroupoid of the pair groupoid). Fixing $x\in M$, the surjective local groupoid morphism $\phi$ induces surjective group homomorphisms:
\[ \phi_*:\Gamma_i^G(x)\to \Gamma_i^U(x)\quad (i=1,2,3). \]
Denoting the kernel of this map by $K_i$ and omitting the reference to the base point $x$, we obtain a commutative diagram with exact rows:
\[
\xymatrix{
          &                         & \Gamma_3^G \ar[d]^{\delta} \ar[r]_{\phi_*} & \Gamma_3^U \ar[d]^{\delta} \ar[r] & 1 \\
1\ar[r] & K_2\ar[r]\ar[d]^{\delta} & \Gamma_2^G \ar[d]^{\delta} \ar[r]_{\phi_*} & \Gamma_2^U \ar[d]^{\delta} \ar[r] & 1 \\
1\ar[r] & K_1\ar[r]         & \Gamma_1^G \ar[r]_{\phi_*}               & \Gamma_1^U \ar[r] & 1.
}
\]
It follows that we have a long exact sequence:
\[
\xymatrix@C20pt{\cdots \ar[r] & \pi_2(\mathbf{U},x) \ar[r]^{\partial^s} & K_1/\delta(K_2)\ar[r] & \pi_1(\mathbf{G},x) \ar[r] & \pi_1(\mathbf{U},x)\ar[r] & \cdots}
\]

\begin{definition}
The morphism $\partial^s: \pi_2(\mathbf{U},x) \to K_1/\delta(K_2)$ is called the \emph{simplicial monodromy map} of $G$ at $x$. Its image is called the \emph{simplicial monodromy group} of $G$ at $x$.
\end{definition}

Note that, in general, $K_i\not=\Gamma_i^{G_x}$. However, we have:

\begin{lemma}
\label{lem:simpl:mon:morph}
Assume that $G$ is strongly connected. Then the map:
\[ G_x \to K_1/\delta(K_2),\quad g\mapsto ({g}\bullet {x}^{-1})\delta(K_2). \]
defines a morphism of local groupoids that extends to a surjective group homomorphism $\AsCo(G_x)\to K_1/\delta(K_2)$.
\end{lemma}

\begin{proof}
First we check that the map is a local group homomorphism. If $g_1,g_2\in G_x$ and $g_1 g_2$ is defined, we need to show that:
\[ ({g}_1\bullet {x}^{-1})\delta(K_2)\cdot ({g}_2\bullet {x}^{-1})\delta(K_2)=({g_1 g_2}\bullet {x}^{-1})\delta(K_2) \]
This follows because $\delta(K_2)\subset \Gamma_1^G$ is a normal subgroup and we have:
\begin{align*}
({g}_1\bullet {x}^{-1})\bullet ({g}_2\bullet {x}^{-1})\bullet({g_1 g_2}\bullet {x}^{-1})^{-1} &=
{g}_1\bullet {x}^{-1}\bullet {g}_2\bullet {x}^{-1}\bullet {x} \bullet({g_1 g_2})^{-1}\\
&={g}_1\bullet {x}^{-1}\bullet {g}_2 \bullet({g_1 g_2})^{-1}\\
&={g}_1\bullet {x}^{-1}\bullet {{g}_1}^{-1}\bullet{g}_1\bullet{g}_2 \bullet({g_1 g_2})^{-1}\\
&=\delta((g_1,x)^{-1}\bullet(g_1,g_2))\in \delta(K_2),
\end{align*}
since $\phi_*((g_1,x)^{-1}\bullet(g_1,g_2))=(x,x)^{-1} \bullet(x,x)= \text{empty word}$. 

To show that the map is surjective, we will check that:
\begin{enumerate}[(i)]
\item $K_1/\delta(K_2)$ is a topological group which is connected;
\item The image of $G_x\to K_1/\delta(K_2)$ is an open neighborhood of the unit.
\end{enumerate}
It follows that the image of $G_x$ generates $K_1/\delta(K_2)$, so $\AsCo(G_x)\to K_1/\delta(K_2)$ is a surjective group homomorphism.

The topology on $K_1/\delta(K_2)$ is the quotient topology induced from the subspace topology on $K_1\subset \Gamma^G_1(x)$ ($F(G)$ is a topological groupoid, so $\Gamma^G_1(x):=F(G^1)_{x}$ is a topological group). The image of the map $G_x\to \AsCo(G_x)$ is an open neighborhood of the unit and this map factors through the map $G_x\to K_1/\delta(K_2)$. Hence, the image of $G_x\to K_1/\delta(K_2)$ is an open neighborhood of the unit. To check that $K_1/\delta(K_2)$ is connected, choose any word $w\in K_1$
\[ w=g_1^{\eps_1}\bullet \dots \bullet g_{k}^{\eps_k}, \quad (g_i\in G, \eps_i=\pm 1). \]
Here $k$ must be even and there exists an involution $\sigma\in S_k$ such that for each $g_i$ with $\eps_i=1$ there exists a $g_{\sigma(i)}$ with $\eps_{\sigma(i)}=-1$, $s(g_i)=s(g_{\sigma(i)})$ and $t(g_i)=t(g_{\sigma(i)})$. Since $G$ is strongly connected, we can find a path of words $w(t)\in K_1$ with $w(1)=w$ and
\[ w(0)=x^{\eps_1}\bullet \dots \bullet x^{\eps_k}. \]
Now, $w(0)\in\delta(K_2)$ since we have that $x=\delta(x,x)$ and $(x,x)^{\eps_1}\bullet \dots \bullet (x,x)^{\eps_k}\in K_2$. So any element in $K_1/\delta(K_2)$ can be connected to the unit.
\end{proof}

We conjecture that the map $\AsCo(G_x)\to K_1/\delta(K_2)$ is also injective, so that $\AsCo(G_x)\simeq K_1/\delta(K_2)$. This would mean that we could view the simplicial monodromy map as a homomorphism:
\[ \partial^s: \pi_2(\mathbf{U},x) \to \AsCo(G_x). \]
We do not know how to prove injectivity for a general strongly connected local groupoid.

\subsection{Simplicial monodromy for shrunk $G$}
\label{ss:simplicial:shrunk}

When $G$ is shrunk we have:

\begin{proposition}
If $G$ is a shrunk local Lie groupoid then there are group isomorphisms:
\[ K_1/\delta(K_2)\simeq \AsCo(G_x)\simeq \G(\g_x). \]
\end{proposition}

\begin{proof}
Since $G$ is shrunk, we have a local group morphism $G_x\to \G(\g_x)$. Applying the functor $\AsCo(-)$ we obtain a morphism of topological groupoids $\AsCo(G_x)\to \G(\g_x)$. But $G_x$ is 1-connected local Lie group, so by Corollary~\ref{cor:1-connected:integrable} we have that $\AsCo(G_x)$ is a Lie group and $\AsCo(G_x)\to \G(\g_x)$ is a Lie group homomorphism. Since $\AsCo(G_x)$ and $\G(\g_x)$ have the same Lie algebra and the latter is 1-connected, it follows that this map is actually an isomorphism.

We show that the isomorphism $\AsCo(G_x)\simeq \G(\g_x)$ factors into group homomorphisms:
\[
\xymatrix{
 & K_1/\delta(K_2) \ar[dr]\\
\AsCo(G_x)\ar[ru] \ar[rr] & & \G(\g_x)
}
\]
The morphism $\AsCo(G_x)\to K_1/\delta(K_2)$ was given in Lemma~\ref{lem:simpl:mon:morph} and it is surjective. So to complete the proof it remains to construct a morphism $K_1/\delta(K_2)\to \G(\g_x)$ making the diagram commute.

Write an element $w\in K_1$ in the form:
\[ w=g_1^{\eps_1}\bullet \dots \bullet g_{k}^{\eps_k}, \quad (g_i\in G, \eps_i=\pm 1). \]
Using the fact that $\phi_*(w)$ is the empty word (the unit in $\Gamma_1^U(x)$), we associate to $w$ a $\g_x$-path (which represents an element in $\G(\g_x)$) by proceeding inductively as follows:
\begin{enumerate}[(i)]
\item[Step 1.] We first choose all the successive pairs of arrows in $w$ of the form:
\[ \xymatrix{\ar@{.}[r]& y\ar[r]^{g} &z & y\ar[l]_{h}\ar@{.}[r] &} \]
Notice that, since $\phi_*(w)$ is the empty word, there must exist at least one such pair. The $A$-paths $P(g)$ and $P(h)$ have opposite base paths (the geodesics connecting $y$ and $z$). Hence, we can find an $A$-homotopy along trivial spheres from $P(g)\circ P(h)^{-1}$ to a $\g_{y}$-path $p(t)$. In this way, we have ``removed'' all such  vertices $z$ from the word $w$ and replaced them by $\g_y$-paths.
\item[Step 2.] After step 1, we will have new successive pairs of arrows (at least one, since  $\phi_*(w)$ is the empty word), but now the vertices will have $\g_y$-paths and/or $\g_{z}$-paths attach to them:
\[ \xymatrix{\ar@{.}[r] & y\ar@(ul,ur)[]^{p_1(t)}\ar[r]^{g} &z \ar@(ul,ur)[]^{p_2(t)}& y\ar@(ul,ur)[]^{p_3(t)}\ar[l]_{h}\ar@{.}[r]&} \]
Again, the $A$-paths $P(g)$ and $P(h)$ have opposite base paths, while $p_i(t)$ are $A$-paths with constant base paths, so we can find an $A$-homotopy along trivial spheres from $p_1\circ P(g)\circ p_2 \circ P(h)^{-1}\circ p_3$ to a $\g_{y}$-path $p(t)$.
\end{enumerate}
After applying step 2 as many times as possible, we are left in the end with a $\g_x$-path whose $\g_x$-homotopy class defines an element $[p(t)]\in\G(\g_x)$. 

In order to have a well-defined morphism $K_1/\delta(K_2)\to \G(\g_x)$ we need to check that this construction associates to an element in $\delta(K_2)$ a $\g_x$-path which is $\g_x$-homotopic to the trivial path. But this follows from property (5) in the definition of a shrunk Lie groupoid (see Definition~\ref{condition-shrunk-essence}) and the fact that every element in $K_2$ contains pairs $(w_1,(g_1,h_1))$ and $(w_2,(g_2,h_2)^{-1})$, where $(g_1,h_1),(g_2,h_2)\in G^{(2)}$ satisfy $\varphi_2(g_1,h_1)=\varphi_2(g_2,h_2)$.
\end{proof}

Using this result, we can now complete the proof of Theorem~\ref{thm:main:monodromy}.

\begin{proof}[Proof of Theorem~\ref{thm:main:monodromy}]
From the previous proposition, we already know that there is a well-defined map $\partial^s:\pi_2(\mathbf{U},x) \to \AsCo(G_x)$ giving a long exact sequence \eqref{eq:long:exact:local:grpd}. We turn to the proof of the existence of a commutative diagram:
\[
\xymatrix{
\pi_2(M,x)\ar[r]^\partial \ar[d] & \G(\g_x)\\
\pi_2(\mathbf{U},x) \ar[r]_{\partial^s} & \AsCo(G_x)\ar[u]}
\]
with vertical arrows which are isomorphisms. Note that the only map that is
missing is the right vertical arrow. 

Since $G$ is shrunk, for every pair $(x,y)\in U$ there is a unique geodesic connecting $x$ to $y$, so the source fibers of $U\tto M$ are simply connected. It follows that:
\[ \AsCo(U)\simeq \Pi_1(M), \]
and, in particular, $\pi_1(\mathbf{U},x)\simeq\pi_1(M,x)$. We are looking for the degree 2 version of this isomorphism.

Given a element $[\alpha]\in\pi_2(M,x)$, represented by a map $\alpha:\Delta\to M$, we subdivide $\Delta$ into $k^2$ triangles as in the proof of Theorem~\ref{thm:assoc-mono}. The subdivision is such that $e_3$, $e_2$ and $e_1$ are the edges of a face then $(e_2,e_1)\in U^{(2)}$. For the resulting triangulation, we can associate to each lace around a face $F$:
\[ (e_1^{-1}, \dots, e_n^{-1}, c, b, a, e_n, \dots, e_1) ,\]
the element $(w_F,\sigma_F)$ where $w_F=e_1^{-1}\bullet \dots\bullet e_n^{-1}$ and $\sigma_F=(b,a)$. If $F_1,\dots,F_{k^2}$ is the ordered list of faces of $\Delta$, then we obtain an element:
\[  (w_{F_1},\sigma_{F_1})\bullet\cdots \bullet (w_{F_{k^2}},\sigma_{F_{k^2}})\in \Gamma^U_2(x). \]
We leave it as an exercise to check that this defines a morphism $\pi_2(M,x)\to \pi_2(\mathbf{U},x)$. 

Now consider the construction in the proof of Theorem~\ref{thm:assoc-mono} which to $[\alpha]$ associates a simplicial map $\phi : S \to \NG$ with $[\phi]=[\alpha]$. One of the boundary words of $\phi$ is $(x)$ and the other gives an element $(g_1,\dots,g_k)\in\AsCo(G_x)$. In the present language, this yields the composition of 
the map $\pi_2(M,x)\to \pi_2(\mathbf{U},x)$ that we have just constructed with the simplicial monodromy map $\partial^s:\pi_2(\mathbf{U},x) \to \AsCo(G_x)$. Hence, the proof of Theorem~\ref{thm:assoc-mono} shows that the diagram above is commutative and that the first vertical arrow must be an isomorphism.
\end{proof}

\begin{proof}[Proof of Corollary~\ref{cor:shrunk}]
We already know that the morphism $\AsCo(G)\to \G(A)$ is surjective. We only need to show that it is injective, i.e., that it restricts to an isomorphism on isotropy groups. But the two exact sequences \eqref{eq:short:exact:groupoid} and \eqref{eq:long:exact:local:grpd}, together with the diagram in Theorem~\ref{thm:main:monodromy}, show that this is indeed the case.
\end{proof}


\bibliographystyle{plain}

\begin{thebibliography}{10}

\bibitem{bailey-gualtieri}
M.~{Bailey} and M.~{Gualtieri}.
\newblock {Integration of generalized complex structures}.
\newblock {\em ArXiv e-prints 1611.03850}, November 2016.

\bibitem{CMS}
A.~Cabrera, I.~M{\u a}rcu{\c t}, and M.~A. Salazar.
\newblock On local integration of lie brackets.
\newblock {\em Journal f{\"u}r die reine und angewandte Mathematik (Crelles
  Journal) on-line} May 2018.

\bibitem{Cartan}
E.~Cartan.
\newblock Le troisi\'eme th\'eor\`eme fondamental de lie.
\newblock In {\em Oeuvres compl{\`e}tes. {P}artie {I}. {G}roupes de {L}ie},
  volume~1, pages 1307--1330. Gauthier-Villars, Paris, 1952.

\bibitem{integ-of-lie-article}
M.~Crainic and R.~L. Fernandes.
\newblock Integrability of {L}ie brackets.
\newblock {\em Ann. of Math. (2)}, 157(2):575--620, 2003.

\bibitem{lectures-integrability-lie}
M.~Crainic and R.~L. Fernandes.
\newblock Lectures on integrability of {L}ie brackets.
\newblock In {\em Lectures on {P}oisson geometry}, volume~17 of {\em Geom.
  Topol. Monogr.}, pages 1--107. Geom. Topol. Publ., Coventry, 2011.

\bibitem{coframes}
R.~L. Fernandes and I.~Struchiner.
\newblock The classifying {L}ie algebroid of a geometric structure {I}:
  {C}lasses of coframes.
\newblock {\em Trans. Amer. Math. Soc.}, 366(5):2419--2462, 2014.

\bibitem{Goldbring}
I.~Goldbring.
\newblock Hilbert's fifth problem for local groups.
\newblock {\em Ann. of Math. (2)}, 172(2):1269--1314, 2010.

\bibitem{Henriques}
A.~Henriques.
\newblock Integrating {$L_\infty$}-algebras.
\newblock {\em Compos. Math.}, 144(4):1017--1045, 2008.

\bibitem{lee-riemannian}
J.~M. Lee.
\newblock {\em Riemannian manifolds}, volume 176 of {\em Graduate Texts in
  Mathematics}.
\newblock Springer-Verlag, New York, 1997.
\newblock An introduction to curvature.

\bibitem{malcev}
A.~Malcev.
\newblock Sur les groupes topologiques locaux et complets.
\newblock {\em C. R. (Doklady) Acad. Sci. URSS (N.S.)}, 32:606--608, 1941.

\bibitem{moerdijk-mrcun}
I.~Moerdijk and J.~Mr\v{c}un.
\newblock {\em Introduction to foliations and {L}ie groupoids}, volume~91 of
  {\em Cambridge Studies in Advanced Mathematics}.
\newblock Cambridge University Press, Cambridge, 2003.

\bibitem{olver}
P.~J. Olver.
\newblock Non-associative local {L}ie groups.
\newblock {\em J. Lie Theory}, 6(1):23--51, 1996.

\bibitem{sharpe}
R.~W. Sharpe.
\newblock {\em Differential geometry}, volume 166 of {\em Graduate Texts in
  Mathematics}.
\newblock Springer-Verlag, New York, 1997.
\newblock Cartan's generalization of Klein's Erlangen program, With a foreword
  by S. S. Chern.

\bibitem{Smith}
P.~A. Smith.
\newblock The complex of a group relative to a set of generators. {I}.
\newblock {\em Ann. of Math. (2)}, 54:371--402, 1951.

\bibitem{Tao}
T.~Tao.
\newblock {\em Hilbert's fifth problem and related topics}, volume 153 of {\em
  Graduate Studies in Mathematics}.
\newblock American Mathematical Society, Providence, RI, 2014.

\bibitem{EstLee}
W.~T. van Est and M.~A.~M. van~der Lee.
\newblock Enlargeability of local groups according to {M}al'cev and
  {C}artan-{S}mith.
\newblock In {\em Action hamiltoniennes de groupes. {T}roisi{\`e}me
  th{\'e}or{\`e}me de {L}ie ({L}yon, 1986)}, volume~27 of {\em Travaux en
  Cours}, pages 97--127. Hermann, Paris, 1988.

\bibitem{Zhu2009}
C.~Zhu.
\newblock {$n$}-groupoids and stacky groupoids.
\newblock {\em Int. Math. Res. Not. IMRN}, (21):4087--4141, 2009.

\end{thebibliography}

\end{document}